\documentclass[11pt,reqno]{amsart}

\usepackage{amsfonts,latexsym,amssymb,amsmath,amsthm}
\usepackage{todonotes}
\usepackage[a4paper,margin=30mm,marginparwidth=26mm,marginparsep=2mm]{geometry}
\usepackage[shortlabels]{enumitem}
\usepackage{mathrsfs}
\usepackage[backend=biber]{biblatex}
\usepackage{comment}
\usepackage{braket}
\usepackage{bbm}
\usepackage{hyperref}
\usepackage[nameinlink,capitalise]{cleveref}
\usepackage{microtype}
\usepackage{mathabx}

\usepackage{tikz}
\usetikzlibrary{arrows.meta,calc}

\numberwithin{equation}{section}

\hypersetup{
    colorlinks,
    linkcolor={red!50!black},
    citecolor={blue!50!black},
    urlcolor={blue!80!black}
}

\newcommand{\diag}{\mathrm{diag}}
\newcommand{\eps}{\varepsilon}
\newcommand{\tens}{\otimes}

\newcommand{\N}{\mathbb{N}}
\newcommand{\Z}{\mathbb{Z}}
\newcommand{\R}{\mathbb{R}}
\newcommand{\C}{\mathbb{C}}

\newcommand{\Dcal}{\mathcal D}
\newcommand{\Jcal}{\mathcal J}
\newcommand{\Hcal}{\mathcal H}
\newcommand{\Cinf}{\mathit C^{\infty}}
\newcommand{\Cdotinf}{\dot{\mathit C}^{\infty}}
\newcommand{\Scal}{\mathcal S}
\newcommand{\Mcal}{\mathcal{M}}
\newcommand{\Dscr}{\mathscr{D}}
\newcommand{\Escr}{\mathscr{E}}
\newcommand{\Ascr}{\mathscr{A}}
\newcommand{\Nscr}{\mathscr{N}}
\newcommand{\Sscr}{\mathscr{S}}
\newcommand{\Fock}{\mathfrak{F}}
\newcommand{\Fscr}{\mathscr{F}}
\newcommand{\sob}{\mathsf{H}}

\newcommand{\comps}{\mathrm{comp}}
\newcommand{\locs}{\mathrm{loc}}
\newcommand{\pc}{\mathrm{pc}}
\newcommand{\fc}{\mathrm{fc}}
\newcommand{\spc}{\mathrm{spc}}
\newcommand{\sfc}{\mathrm{sfc}}
\newcommand{\scs}{\mathrm{sc}}

\newcommand{\WF}{\mathsf{WF}}
\newcommand{\bWF}{\partial\WF}
\newcommand{\im}{\mathrm{i}}
\newcommand{\vol}{\mathrm{vol}}
\newcommand{\n}{\mathsf{n}}
\newcommand{\sym}[1]{\mathit{S}^{#1}}
\DeclareMathOperator{\ran}{\mathrm{ran}}
\DeclareMathOperator{\dist}{\mathrm{dist}}

\DeclareMathOperator{\Sf}{\mathbb S}
\DeclareMathOperator{\Id}{\mathrm{Id}}
\DeclareMathOperator{\Ell}{\mathrm{Ell}}
\DeclareMathOperator{\Char}{\mathrm{Char}}
\DeclareMathOperator{\dom}{\mathrm{dom}}
\DeclareMathOperator{\FT}{\mathcal F}
\DeclareMathOperator{\supp}{\mathrm{supp}}
\DeclareMathOperator{\singsupp}{\mathrm{singsupp}}
\DeclareMathOperator{\Op}{\mathrm{Op}}
\DeclareMathOperator{\Diff}{\mathrm{Diff}}
\DeclareMathOperator{\LS}{\mathrm{LS}}
\DeclareMathOperator{\spec}{\mathrm{spec}}

\renewcommand{\bar}[1]{\overline{#1}}
\renewcommand{\tilde}[1]{\widetilde{#1}}
\renewcommand{\check}[1]{\widecheck{#1}}
\renewcommand{\phi}{\varphi}
\renewcommand{\hat}[1]{\widehat{#1}}
\renewcommand{\b}{{}^{\mathit{b}}\!}
\DeclareMathOperator{\diverg}{\mathrm{div}}

\def\setquotient#1#2{{}^{#1}\!\!\diagup_{#2}}

\newcommand{\de}{\mathop{}\!\mathrm{d}}
\newcommand{\VF}{\mathfrak{X}}
\newcommand{\Sstar}{\mathbb{S}^{\ast\!}}
\newcommand{\Tstar}{\mathit{T}^{\ast\!}}
\newcommand{\bT}{\b{\,T}}

\newcommand{\bTs}[1]{{}^{\mathit b}T^{\ast}#1}          
\newcommand{\bSs}[1]{{}^{\mathit b}\mathbb{S}^{\ast}#1} 
\newcommand{\compcosphere}[1]{{\tilde{\Sf}^\ast}\!{#1}}

 \newtheorem{theorem}{Theorem}[section]
 \newtheorem*{theorem*}{Theorem}
 \newtheorem{lemma}[theorem]{Lemma}
 \newtheorem{corollary}[theorem]{Corollary}
 \newtheorem{prop}[theorem]{Proposition}

 \theoremstyle{definition}
 \newtheorem{definition}[theorem]{Definition}
 \newtheorem{example}[theorem]{Example}
 \newtheorem{remark}[theorem]{Remark}

\allowdisplaybreaks
\title[Hadamard states and timelike boundaries]{Hadamard states for spacetimes with timelike boundaries}

\author[A.P.~Contini]{Alessandro Pietro Contini}
\address{Leibniz University Hannover, Institute of Analysis, 30167 Hannover, Germany}  \email{alessandro.contini@math.uni-hannover.de}

\author[A.~Strohmaier]{Alexander Strohmaier}
\address{Leibniz University Hannover, Institute of Analysis, 30167 Hannover, Germany}  \email{a.strohmaier@math.uni-hannover.de}

\subjclass[MSC2020]{81T20, 35A21}
\keywords{Microlocal analysis, Hadamard states, spectral theory, quantum field theory on curved spacetimes, timelike boundary}

\thanks{APC thanks Nikolas Eptaminitakis, Arne Hofmann and Rafi Islam for useful conversations throughout the project}

\date{\today}

\begin{document}

\begin{abstract}
	We formulate and prove existence of pure quasifree Hadamard states for the Klein--Gordon field with Dirichlet boundary condition on a globally hyperbolic spacetime with timelike boundary. The microlocal condition is imposed on the one-particle Hilbert-space-valued distribution $p\circ G_D$ rather than on the two-point function. This is expressed by a $b$-wave-front set condition on the compressed future cone. The analysis on the manifold with corners $\bar M\times\bar M$ is then almost completely avoided. We prove that our version of the Hadamard condition is universal in the class of quasi-free Hadamard states. We also prove a propagation theorem for positive-energy microlocal splittings which uses only a microlocal splitting in the small $b$-calculus and well-posedness of the Dirichlet problem, and in particular requires no propagation theorem for reflected or gliding rays; combined with an explicit computation in the ultrastatic case and a Fulling--Narcowich--Wald deformation argument this yields the construction of a pure quasi-free Dirichlet--Hadamard state.
\end{abstract}

\maketitle


\section{Introduction and statement of main results}

The Hadamard condition is the microlocal form of the positive-frequency condition for quantum fields on curved spacetimes. On globally hyperbolic spacetimes without boundary, Radzikowski's theorem identifies the Hadamard condition with a precise wave-front set condition on the two-point function \cite{radzikowski_micro-local_1996}. This formulation is now one of the basic tools in the construction of physically meaningful states in algebraic quantum field theory on curved spacetimes. The construction of Hadamard states, and the analysis of their properties for wave-type equations, can be understood as a microlocal splitting of the canonical propagators and fundamental solutions into past and future directed components. Indeed, Radzikowski's work is based heavily on the work of Duistermaat and H\"ormander on distinguished parametrices and Fourier integral operators \cite{duistermaat_FIO_II_1972}. This provides a very interesting connection between microlocal analysis of PDEs and quantum field theory on curved spacetimes, see for example \cite{brunetti_microlocal_1996,brunetti_microlocal_2000}

The purpose of this paper is to clarify the corresponding notion for spacetimes with timelike boundary and to construct microlocal splittings for the Klein--Gordon equation in this setting. We consider a globally hyperbolic spacetime with timelike boundary $\bar M$, and the real scalar field governed by a formally self-adjoint normally hyperbolic operator
\begin{equation*}
P=\Box_g+V
\end{equation*}
with Dirichlet boundary condition and $V$ a smooth potential, bounded from above and bounded from below by a positive constant. The presence of the boundary leads to several related difficulties. First, the correct functional-analytic setting is more delicate than in the boundaryless case: supports may meet the boundary, and the relevant Cauchy data spaces are fractional Sobolev spaces, in which boundary conditions are less direct to formulate. Second, the ordinary two-point function is a distribution on $\bar M\times \bar M$, a manifold with corners. A direct formulation of the Hadamard condition on this product would therefore require microlocal analysis on a cornered space. Finally, a direct use of propagation of singularities for boundary value problems leads to the delicate analysis of reflected and gliding rays, and in particular to the behaviour of the null-geodesic flow at glancing points.

One of the main points of this paper is that these difficulties can be separated. The correct definition of a Hadamard state with Dirichlet boundary condition is obtained by formulating the microlocal condition at the level of the one-particle distribution, rather than directly at the level of the scalar two-point function. This Hilbert-space-valued point of view is in the spirit of the use of Hilbert-space-valued distributions and their wave-front sets in \cite{strohmaier_reeh-schlieder_2002}. More precisely, if
\begin{equation*}
G_D=G^+_D-G^-_D
\end{equation*}
is the Dirichlet Pauli--Jordan propagator and if $p$ denotes the projection onto the positive spectral subspace determined by a one-particle structure, then the relevant object is the Hilbert-space-valued distribution
\begin{equation*}
\Phi=p\circ G_D .
\end{equation*}
The two-point function is recovered from it by
\begin{equation*}
\omega_2(f,h)=\left(\Phi(\bar f),\Phi(h)\right).
\end{equation*}
Thus the microlocal analysis is carried out on $\bar M$, not on $\bar M^2$. This is the first simplification: the use of Hilbert-space-valued distributions allows us to avoid corners altogether.

The second point is that the correct boundary singularity is not the full $b$-wave-front set as such. The $b$-wave-front set provides the invariant language in which the analysis is carried out, but for distributions which are normally regular the relevant boundary object is the boundary wave-front set
\begin{equation*}
\partial\WF(u)=\WF_b(u)\cap \b\Tstar_{\partial M}\bar M .
\end{equation*}
This wave-front set measures the failure of strong tangential regularity at the boundary. In the present problem normal regularity is supplied by the fact that the boundary is non-characteristic for $P$. The remaining singular information is therefore tangential, and is encoded by $\partial\WF$; see \cite{chazarain_reflection_1976,de_gosson_microlocal_1982}.

The third point is the identification of the correct function spaces. Test functions have to be allowed to have compact support meeting the boundary. In particular, $\Cinf_c(M)$ is not invariant under time evolution: singularities and supports may propagate to the timelike boundary. The correct test space for the Dirichlet field is therefore $\Cinf_c(\bar M)$, together with the quotient by the image of $P$ appropriate to the Dirichlet Pauli--Jordan propagator.

On a Cauchy surface $\bar\Sigma$, the corresponding one-particle space is not an ad hoc Sobolev space. It is forced by the spectral calculus of the Dirichlet operator. If
\begin{equation*}
A=\Delta_D+V ,
\end{equation*}
with $\Delta_D\ge 0$, we use the convention
\begin{equation*}
\Dcal_s=\dom(A^{s/2}) .
\end{equation*}
Then, in the range relevant here,
\begin{equation*}
\dot\sob^{1/2}_{\comps}(\bar\Sigma)\oplus \sob^{-1/2}_{\comps}(\bar\Sigma) \subset \Dcal_{1/2}\oplus \Dcal_{-1/2} \subset
\dot\sob^{1/2}_{\locs}(\bar\Sigma)\oplus \sob^{-1/2}_{\locs}(\bar\Sigma)
\end{equation*}
and equality in case the spacetime is spatially compact.
The dot in $\dot\sob^{1/2}$ is important: this is the fractional Dirichlet space, equivalently the Lions--Magenes space $\sob^{1/2}_{00}(\Sigma)$, not a naive zero-trace space \cite{lions_non-homogeneous_1972_I}. This identification is one of the basic inputs in the construction of the field and in the proof that $p\circ G_D$ is a Hilbert-space-valued distribution on $\bar M$.

We define a quasifree state to be Dirichlet--Hadamard if its one-particle distribution satisfies the microlocal positive-energy condition. More precisely, the Hilbert-space-valued distribution $\Phi$ is required to be normally regular, its interior wave-front set is required to have positive time-orientation, and its boundary wave-front set is required to satisfy
\begin{equation*}
\partial\WF(\Phi)\subset \tilde{V}^{+}_{\partial M}\bar M ,
\end{equation*}
where $\tilde{V}^{+}_{\partial M}\bar M$ denotes the compressed future causal cone over the boundary. Equivalently, the $b$-wave-front set of $\Phi$ is contained in the compressed future causal cone inside the $b$-cotangent bundle:
\begin{equation*}
\WF_b(\Phi)\subset \tilde{V}^{+}\bar M .
\end{equation*}
By the calculus of Hilbert-space-valued distributions, this condition implies a corresponding microlocal condition for the scalar two-point function. It is therefore the natural analogue, in the presence of a timelike boundary, of the usual microlocal spectrum condition.

There is also a second, purely microlocal, result in the paper which may be of independent interest. It concerns the propagation of positive energy for solutions of the Dirichlet problem. Let the distribution $u$ solve
\begin{equation*}
Pu=0,\qquad u|_{\partial M}=0.
\end{equation*}
Suppose that near one Cauchy surface the singularities of $u$ have only future time-orientation, in the sense that for a neighbourhood $U$ of this Cauchy surface one has
\begin{equation*}
\WF_b(u)\cap \b\Tstar U\subset \tilde{V}^{+}U,
\end{equation*}
where $\tilde{V}^+U$ is the compressed future causal cone. We prove that the same inclusion holds globally:
\begin{equation*}
\WF_b(u)\subset \tilde{V}^{+}\bar M .
\end{equation*}
Equivalently, absence of negative-energy singularities near one Cauchy surface implies absence of negative-energy singularities everywhere.

The proof of this propagation result is deliberately economical. Although related statements can be obtained from the general propagation-of-singularities theorems of Melrose and Sj\"ostrand \cite{melrose_singularities_BVP_I_1978,melrose_singularities_BVP_II_1982}, we do not use a full propagation theorem for boundary value problems, and in particular do not need to follow reflected or gliding rays. Instead we use the small $b$-calculus only to perform a microlocal splitting. A $b$-pseudo-differential partition of unity gives, modulo smooth terms,
\begin{equation*}
u=u_+ + u_- + u_s ,
\end{equation*}
where $u_+$ is microlocalised in the future cone, $u_-$ in the past cone, and $u_s$ in the spacelike region. Boundary elliptic regularity for the Dirichlet problem shows that the spacelike term is smooth. If the original solution has no past singularities near a Cauchy surface, then the past component $u_-$ is smooth there. Well-posedness and ordinary propagation of smoothness for the Dirichlet problem then imply that $u_-$ is smooth everywhere. Hence
\begin{equation*}
u=u_+ \mod \Cinf(\bar M),
\end{equation*}
which is the desired global positive-energy condition.

The $b$-calculus is used to make the future/past/spacelike splitting invariant up to the boundary; after this splitting has been made, the global argument requires only smooth well-posedness of the Dirichlet problem. Thus the result can be viewed as a propagation theorem for positive-energy microlocal splittings, rather than as a full boundary propagation theorem.

We then prove existence of Dirichlet--Hadamard states. In the ultrastatic case $\bar M=\R\times\bar\Sigma$, the Dirichlet spectral calculus gives an explicit one-particle structure. The positive-frequency field is described by the functional calculus of $\Delta_D+V$, and the estimates on Dirichlet eigenfunctions show that the associated Hilbert-space-valued distribution has $b$-wave-front set contained in the compressed future cone:
\begin{equation*}
\WF_b(p\circ G_D)\subset \tilde{V}^{+}\bar M .
\end{equation*}
This gives a Dirichlet--Hadamard state in the ultrastatic model.

The general case is obtained by a deformation argument in the spirit of Fulling--Narcowich--Wald \cite{fulling_singularity_1981}. One deforms the given globally hyperbolic spacetime with timelike boundary to an ultrastatic one in the past, constructs the Dirichlet--Hadamard state there, and then propagates the positive-energy microlocal condition by the theorem described above. Applying the same argument across the region where the deformed spacetime agrees with the original one gives the desired state on $\bar M$.

The main existence result is therefore the following.

\begin{theorem*}[A]
	Let $\bar M$ be a globally hyperbolic spacetime with timelike boundary, let $V$ be a smooth function on $\bar{M}$, and let
	\begin{equation*}
	P=\Box_g+V
	\end{equation*}
	be a scalar, formally self-adjoint, normally hyperbolic operator of order two with Dirichlet boundary condition. Then the corresponding Dirichlet field algebra admits a Dirichlet--Hadamard state.
\end{theorem*}

\subsection*{Comparison with existing literature}

Concerning well-posedness and regularity for non-characteristic hyperbolic mixed problems, the classical references are the works of \textcite{massey_abstract_1972}, \textcite{rauch_differentiability_1974}, and more recently the comprehensive book by \textcite{benzoni-gavage_multi-dimensional_2007}. In the geometric setting of spacetimes with timelike boundaries, many people have analysed initial-boundary-value problems related to our setup. For example, \textcite{ginoux_cauchy_friedrichs_2022} discussed the Cauchy problem for Friedrichs-type hyperbolic systems on globally hyperbolic spacetimes with timelike boundary. The case of the Dirac operator has also been investigated, both for local boundary conditions by \textcite{grosse_well-posedness_2020} and for local-in-time, nonlocal-in-space conditions by \textcite{baer_cauchy_2022}. In this context, we supply a proof of the existence of propagators for Klein--Gordon operators on general globally hyperbolic spacetimes timelike boundary.

The study of the propagation of singularities for solutions to hyperbolic non-characteristic boundary-value problems has a long history. Following the pioneering works of Lax--Nirenberg \cite{nirenberg_lectures_1979}, \textcite{taylor_reflection_1975} and \textcite{chazarain_reflection_1976}, \textcite{melrose_singularities_BVP_I_1978, melrose_singularities_BVP_II_1982} introduced the $b$-wave-front set. Its systematic study, including the extension to non-normally regular distributions, was undertaken by \textcite{melrose1982transformationofboundaryproblems} in his first presentation of the $b$-calculus. The general, geometric framework also provided understanding for the boundary wave-front set as studied by \textcite{chazarain_reflection_1976} and \textcite{de_gosson_microlocal_1982}. Our argument uses no propagation-of-singularities theorem for boundary-value problems. It relies only on boundary elliptic regularity, smooth well-posedness of the Dirichlet problem, and the definition of the $b$-wave-front set.

In the context of quantum field theory on spacetimes with timelike boundaries, \textcite{dappiaggi_maxwell_timelike_2020} analysed the Maxwell operator on differential forms and, conditionally on the existence of propagators, constructed the algebras of observables. Taking advantage of the deformation argument of \textcite{fulling_singularity_1981}, the Hadamard condition for the Dirac field on a manifold with timelike boundary has been studied by \textcite{drago_moller_2022}. Pseudo-differential methods in the construction of Hadamard states go back to the work of Junker \textcite{junker_hadamard_1996} (see also \textcite{junker_adiabatic_2002}) and were substantially developed by \textcite{gerard_construction_2014} and others. To the best of our knowledge, however, these constructions have not been generalised to the setting of a spacetime with boundary. On the other hand, working on asymptotically anti-deSitter spacetimes, \textcite{wrochna_holographic_2017} proposed and analysed a Hadamard-like condition via the methods of \textcite{vasy_wave_2012} and \textcite{hintz_2016_quasilinear_wave}. However, the geometric setting is quite different from ours, since, there, the boundary exists only conformally. On the other hand, the classical approach (as described for example in the survey by \textcite{khavkine_algebraic_qft_2015}), based on the construction of the Hadamard parametrix, was recently employed in a preprint by \textcite{costeri_hadamard_2025} in order to construct propagators and Hadamard states on Minkowski half-space with respect to the Robin boundary condition. However, the authors do not go beyond this local model. To our knowledge, ours is the first construction of Hadamard states for the Klein--Gordon field on a general globally hyperbolic spacetime with timelike boundary.

\subsection*{Conventions}
We collect here the conventions in force throughout the paper.
\begin{itemize}
	\item The Lorentzian signature is $(+,-,\dots,-)$. A vector $X$ is timelike if $g(X,X)>0$, spacelike if $g(X,X)<0$ or $X=0$, and lightlike if $X\neq0$ and $g(X,X)=0$; a covector is timelike, spacelike or lightlike according to the character of its metric dual. A vector/covector is causal if it is either lightlike or timelike. In an adapted frame, for $\alpha=\tau\de t+\zeta_j\de z^j$ one has $g^{-1}(\alpha,\alpha)=\tau^2-|\zeta|^2$, and $\alpha$ is future-directed if it is causal and $g^{-1}(\alpha,\de t)>0$, i.e. $\tau>0$.
	\item $\Box_g$ denotes the Laplace--Beltrami (d'Alembert) operator of the Lorentzian metric $g$, while on a Riemannian manifold $\Delta_g$ denotes the \textit{positive} Laplacian, $\Delta_g=-\diverg_g\operatorname{grad}_g$. With the above signature these two conventions fit together: an ultrastatic metric $g=\de t^2-g_{\bar\Sigma}$ gives
	\begin{equation*}
		\Box_g=\partial^2_{tt}+\Delta_{g_{\bar\Sigma}},
	\end{equation*}
	which is the form of the wave operator used throughout.
	\item The Fourier and Mellin transforms are normalised as in \cref{sect:distributions on manifolds with boundary}: the oscillatory factors are $e^{-\im z\zeta}$ and $x^{-\im\lambda-1}$, and each transform carries the prefactor $(2\pi)^{-d/2}$, where $d$ is the number of variables transformed.
	\item The scalar product of a complex Hilbert space is written $\left(\cdot,\cdot\right)$ and is conjugate-linear in the \textit{first} and linear in the \textit{second} variable. The pairing between a space of distributions and its dual is \textit{bilinear} and is written $\braket{\cdot,\cdot}$.
	\item $\tilde G^\pm_D$ denote the retarded and advanced inverses of the Dirichlet problem $(P,\gamma_0)$ (see \cref{theo:propagators}) with arbitrary boundary datum; $G^\pm_D=\tilde G^\pm_D(\,\cdot\,,0)$ denote their restrictions to vanishing boundary datum; and $G_D=G^+_D-G^-_D$ is the Dirichlet Pauli--Jordan propagator.  Where the difference has to be taken with arbitrary boundary datum (in the exact sequence of \cref{lemma:exact sequence}) we write $\tilde G_D=\tilde G^+_D-\tilde G^-_D$, so that $G_D=\tilde G_D(\,\cdot\,,0)$.
\end{itemize}

The paper is organised as follows. We begin by recalling the distribution spaces on manifolds with boundary that are used throughout the paper, including supported and extensible distributions, normally regular distributions, and the Sobolev spaces associated with the Dirichlet operator. We then develop the Hilbert-space-valued version of the $b$-wave-front set and explain the relation with distributions on the product space. The next section discusses globally hyperbolic spacetimes with timelike boundary and the Dirichlet Pauli--Jordan propagator, and proves the propagation theorem for positive-energy splitting. After recalling the construction of the field algebra of the Dirichlet--Klein--Gordon equation, we construct and analyse the ultrastatic model, proving the positive-energy microlocal estimate there. Finally, we apply the deformation argument to obtain the existence of Dirichlet--Hadamard states. An appendix recalls the parts of the small $b$-calculus used in the proof.

\section{Distributions on a manifold with boundary}\label{sect:distributions on manifolds with boundary}

Let $\bar{M}$ be a smooth manifold with (embedded) boundary, let $M$ be its interior, and assume a volume form has been chosen. In the main part of this paper, $\bar{M}$ will always be a smooth submanifold (possibly with boundary) of a Lorentzian spacetime. It is thus no restriction to forgo any mention of density bundles and assume a global density has been chosen.

Recall that a \textit{boundary-defining function} is a (possibly only locally defined) smooth function $x\colon \bar{M}\to \R$ that is nonnegative and vanishes to exactly first order at boundary points. When we write down expressions in coordinates near a point $p\in\partial M$, we will stick to the following convention: a coordinate chart will be written as $(x,y^{1},\dots,y^{n-1})$, where $x$ is a local boundary-defining function near $p$ and $y^j$ are coordinates on $\partial M$, lifted to $\bar{M}$ near $p$ using $x$. It is always possible, using the \textit{collar neighbourhood theorem}, to find a \textit{globally defined} boundary-defining function. The coordinates chosen in such a way are called \textit{adapted coordinates}.

\subsection*{Supported and extensible distributions}The basic spaces of distributions are the supported $\dot\Dscr'(\bar{M})$ and extensible $\Dscr'(\bar{M})$ distributions, defined as
\begin{align*}
	\dot\Dscr'(\bar{M})&\equiv \left(\Cinf_c(\bar{M})\right)',\\
	\Dscr'(\bar{M})&\equiv \left(\Cdotinf_c(\bar M)\right)'.
\end{align*}
In accordance with the above notation, for a distribution type $\Fscr$, we will always denote in what follows the set of supported elements of type $\Fscr$ by $\dot\Fscr$ and simply by $\Fscr$ the space of extensible objects (see \cref{lemma:exact sequence of distributions} for the explanation of the name). Moreover, the notation $\Fscr_{\comps}$ will be used for $\Fscr$-distributions of compact support and $\Fscr_{\locs}$ for objects that are, locally near any point of $\bar{M}$, distributions of type $\Fscr$. In the case of compact $\bar M$, these two classes clearly coincide. Finally, we denote by $\dot{\Fscr}_K(\bar{M})$ the set of distributions of type $\Fscr$, supported by $\bar{M}$, whose support is entirely contained in the set $K\subset \bar{M}$.

The relation between supported and extensible spaces is encoded in
\begin{lemma}[see Section I.1 in \cite{melrose1982transformationofboundaryproblems}]\label{lemma:exact sequence of distributions}
	\begin{enumerate}
		\item The following sequence is exact
		\begin{equation}\label{eq:exact sequence for supported/extensible distributions}
			0\to \dot\Dscr'_{\partial M}(\bar{M})\hookrightarrow\dot\Dscr'(\bar{M})\xrightarrow{\cdot|_M}\{u|_{M}\colon u\in\Dscr'(\bar{M})\}\to 0,
		\end{equation}
		\item A distribution $u\in\Fscr(M)$ is in ${\Fscr}(\bar{M})$ if, and only if, for any enlargement $N\supset \bar{M}$ we find $U\in\Fscr(N)$ such that $U|_{M}=u$.
		\item A distribution $u\in\Fscr(M)$ is in $\dot\Fscr(\bar{M})$ if, and only if, for any enlargement $N\supset \bar{M}$ we have $\tilde u\in \Fscr(N)$, where $\tilde u(\phi)=u(\phi|_{M})$ for all $\phi\in\Cinf_c(N)$ is the ``extension-by-zero'' of $u$.
		\item In particular, on any enlargement $N$ of $\bar{M}$, the ``extension-by-zero'' of supported distributions $u\in\dot\Fscr(\bar M)$ realises them as elements of $\Fscr(N)$.
	\end{enumerate}
\end{lemma}
In view of the above Lemma, the support and singular support of $u\in\Dscr'(\bar{M})$, as an ordinary distribution on $M$, are defined to be
\begin{align*}\label{eq:support of extensible distribution}
	\supp u&=\bigcap\{\supp\tilde u\colon \tilde u\in\dot\Dscr'(\bar{M}),\ \tilde u|_{M}=u\},\\
	\singsupp u&=\bigcap\{\singsupp\tilde u\colon \tilde u\in\dot\Dscr'(\bar{M}),\ \tilde u|_{M}=u\}.
\end{align*}

The spaces of ``negligible'' (to be made more precise later on) distributions in the analysis of non-characteristic boundary value problems are, in Hörmander's notation, $\dot\Ascr(\bar{M})=I(\bar{M},\nu^\ast\partial M)$, namely the space of (supported) conormal distributions to the boundary on $\bar{M}$, of any order and of type $(1,0)$, and its extensible counterpart $\Ascr(\bar{M})$. More explicitly, elements of $\dot\Ascr(\bar{M})$ are given, in local boundary coordinates $(x,y^1,\dots,y^{n-1})\in\bar{M}$ and canonical dual variables $(\zeta,\eta_1,\dots,\eta_{n-1})\in T^\ast_{(x,y)}\bar{M}$, by an oscillatory integral
\begin{equation}
	\label{eq:conormal distributions}
	u(x,y)=\int_{\R} e^{\im x\zeta}a(y,\zeta)\de\zeta,
\end{equation}
where $a$ is a symbol in the Hörmander class with respect to $\zeta$, i.e. $a$ is a smooth function on $\R^{n-1}\times \R$ for which an $m\in\R$ exists, satisfying
\begin{equation}\label{eq:symbol estimates}
	|\partial^\alpha_y \partial^k_\zeta a(y,\zeta)|\leq C_{K\alpha k}\braket{\zeta}^{m-k}
\end{equation}
for any multi-index $\alpha\in\N^{n-1}$ and any $ k\in\N$, where $y$ varies in the compact set $K\subset\R^{n-1}$.
Typical elements of $\dot\Ascr(\bar{M})$ are $\phi(y)\tens\delta^{(k)}(x)$, that is, tensor products of smooth functions on $\partial M$ and derivatives of $\delta(x)$, the Dirac distribution in the normal variable. On the other hand, distributions which are singular along the boundary are not elements of $\dot\Ascr(\bar{M})$, as the regularity of these objects would worsen if we were to act with vector fields along the boundary. That is, elements of $\dot\Ascr(\bar{M})$ are \textit{tangentially regular} in the terminology of \cite{de_gosson_microlocal_1982}.

Elements of the dual space $\Ascr'(\bar{M})=(\dot\Ascr_{\comps}(\bar{M}))'$, equipped with the weak-$\ast$ topology, are known as \textit{dual-conormal distributions}. We refer to \cite{alvarez_lopez_topology_2024} for a detailed study of the topology of these spaces.

\begin{lemma}[Equation (3.25) in \cite{melrose1982transformationofboundaryproblems}, see also \cite{alvarez_lopez_topology_2024}]\label{lemma:conormal+dual-conormal=smooth}
	We have
	\begin{equation}
		\dot\Ascr(\bar{M})\cap\Ascr'(\bar{M})=\Ascr(\bar{M})\cap\Ascr'(\bar{M})=\Cinf(\bar{M}).
	\end{equation}
\end{lemma}
\begin{remark}
	\cref{lemma:conormal+dual-conormal=smooth} (1) is an instance of the principle that the regularity of a distribution can be established by looking separately at tangential and normal directions (see for example \cite{de_gosson_microlocal_1982}). However, distributions in $\Ascr'(\bar{M})$, while admitting distributional traces at $\partial M$ of all orders, are not in general ``smooth in $x$'' (see \cref{example:A' vs N'} below).
\end{remark}

\subsection*{Sobolev spaces}Assume for simplicity that $\bar{M}$ is compact. If $g$ is a Riemannian metric on $\bar{M}$, we let $\sob^k(\bar{M})$ denote the intrinsic Sobolev spaces defined by $g$ (that is, by requiring that all covariant derivatives up to order $k$ of the elements be square-integrable with respect to the volume form of $g$). This is independent of the choice of $g$. The fractional spaces, i.e. $\sob^s(\bar{M})$ for $s\in\R$, are defined to be the interpolation spaces between $L^2(\bar{M})$ and $\sob^k(\bar{M})$ for some $k\geq s$. In accordance with \cref{lemma:exact sequence of distributions}, $\sob^s(\bar{M})$ is the set of restrictions to $M$ of distributions in $\sob^s(N)$ for any Riemannian enlargement $N$ (with compatible metric). Similarly, $\dot{\sob}^s(\bar{M})$ is the set of $\sob^s(N)$ distributions with support inside $\bar{M}$ (for arbitrary enlargement $N$). Equivalently, $u\in\dot\sob^s(\bar{M})$ if, for any enlargement $N$ of $\bar{M}$, the extension by 0 of $u$ is in $\sob^{s}(N)$.

Let $\Delta_g$ be the Laplace--Beltrami operator, defined as an unbounded operator on $L^2(\bar{M})$ with the dense domain $\Cinf_c(M)$ (in our notation $\Delta_g=-\operatorname{div}_g\operatorname{grad}_g$ is a positive operator). By $\Delta_D$ we denote the Dirichlet extension of $\Delta_g$, whose domain is $\dot\sob^1(\bar{M})\cap\sob^2(\bar{M})$. This is the same as the Friedrichs extension. We shall consider the Schrödinger-type operator
\begin{equation}\label{eq:laplacian with potential}
	A=\Delta_D+V,
\end{equation}
where $V$ is a bounded smooth function on $\bar{M}$, bounded away from zero (for example $V=m^2$ for some $m>0$). Under these assumptions, $A$ is positive and invertible, so we can construct the operators $A^s$ for each $s\ge0$ using the spectral theorem. The domains of these powers are denoted by
\begin{equation*}\label{eq:domain of powers of Laplacian}
	\Dcal_s=\dom ((\Delta_D+V)^{s/2})=\left\{f\in\Dscr'(\bar{M})\colon\int_\R \mu^{s}\de m_f(\mu)<\infty\right\},
\end{equation*}
where $m_f$ denotes the spectral measure associated with $A$ and $f$. Owing to the boundedness of $V$, the theorem of Kato--Rellich implies that, for $0\le s\leq 1$, $\Dcal_s$ coincides with the domain of $\Delta_D^{s/2}$. We also consider the dual spaces $\Dcal_{-s}$ of $\Dcal_s$, on which the negative powers $A^{-s/2}$ act as isomorphisms to $L^2$.

Interpolation theory (see Lemma 4.11 in \cite{chandler-wilde_interpolation_2015}, or \cite{bergh_interpolation_1976} for a general discussion) gives
\begin{lemma}\label{lemma:domains of powers of Laplacian}
	We have
	\begin{enumerate}
		\item For every $s\in[0,1]$, $\Dcal_s=\dot\sob^s(\bar{M})$;
		\item For every $s\in[-1,0]$, $\Dcal_s=\sob^{s}(\bar{M})$;
		\item For every $s\in\R$, $\dot\sob^s(\bar{M})=\sob^{-s}(\bar{M})'$ with respect to the $L^2(M)$ pairing.
	\end{enumerate}
\end{lemma}
\begin{corollary}
	\label{cor:symplectic domain of Laplacian}
	For every $s\in[-1,1]$, in particular for $s=1/2$, the space $\dot\sob^s(\bar{M})\oplus\sob^{-s}(\bar{M})$ is a symplectic vector space.
\end{corollary}
\begin{proof}
	Any vector space of the form $V\oplus V'$ is symplectic with respect to $\omega((f,\alpha),(h,\beta))=\braket{\alpha, h}-\braket{\beta, f}$. Since $\dot\sob^s(\bar{M})=\left(\sob^{-s}(\bar{M})\right)'$, the proof is complete.
\end{proof}
We remark that, for all $s\notin \Z+\frac{1}{2}$, $\dot\sob^{s}(\bar{M})=\sob^s_0(\bar{M})$, the closure of $\Cinf_c(M)$ in the topology of $\sob^s(\bar{M})$, while at half-integers we have $\dot\sob^{k+1/2}(\bar{M})\subsetneq \sob^{k+1/2}_0(\bar{M})$. This latter space will not be used.

\begin{remark}\label{remark:sobolev spaces}
	\begin{enumerate}
		\item The space $\dot\sob^{1/2}(\bar{M})$ is the classical space $\sob^{1/2}_{00}(M)$ of Lions and Magenes. See Theorem 11.7 in \cite{lions_non-homogeneous_1972_I} for an explicit definition via Gagliardo seminorms and the characterisation as interpolation space in the scale $\sob^s_0(\bar{M})$.
		\item Since the boundary is smooth, $\dot\sob^s(\bar{M})=\tilde{\sob}^s(\bar{M})$ for every $s\in\R$ in the notation of \cite{mclean_strongly_2000}, which reconciles \cref{lemma:domains of powers of Laplacian} with Theorem 3.30 in the reference.
		\item Notice that $\Dcal_s\neq\dot\sob^s(\bar{M})$ and $\Dcal_{-s}\neq\sob^{-s}(\bar{M})$ if $s$ is sufficiently large. Indeed we have $\bigcap_{s}\dot\sob^s(\bar{M})=\Cdotinf(\bar{M})$, while for elements of $\bigcap_s\Dcal_s$ only $\gamma_0 A^j u$ has to vanish. This is related to the fact that $\dom(\Delta_D^k)\subsetneq\dot\sob^1(\bar{M})\cap\sob^{2k}(\bar{M})$ for $k\ge 2$, since the smaller space encodes boundary conditions of higher order. For our discussion of $s=1/2$ this does not play a role.
	\end{enumerate}
\end{remark}

If $\bar{M}$ is noncompact, one needs to discuss separately compactly supported and local versions of the distribution spaces above. The meaningful differences are that the dual to a \textit{loc} space is a \textit{comp} space, and that the domains $\Dcal_s$ for positive $s$ in general will lie between $\dot\sob^s_{\comps}(\bar{M})$ and $\sob^s_{\locs}(\bar{M})$. Since an explicit description is cumbersome and depends on the ``geometry at infinity'' of $\bar{M}$, we retain the notation $\Dcal_s$ whenever possible (see \cref{sect:ultrastatic}). Particularly relevant is the analogue of \cref{cor:symplectic domain of Laplacian}, which has the same proof.
\begin{lemma}\label{lemma:symplectic domain of laplacian - noncompact}
	For any $s\in \R$, the space $\Dcal_s\oplus\Dcal_{-s}$ is symplectic.
\end{lemma}

The restriction map to $M$ relates the supported and extensible Sobolev spaces in a more nuanced way than general distributions.
\begin{lemma}\label{lemma:restriction and Sobolev spaces}
	Consider the restriction map to $M$ as in \cref{lemma:exact sequence of distributions}. For any $s\in\R$, it extends to bounded operators
	\begin{align*}
		\dot\sob^s_\comps(\bar{M})&\to \sob^s_\comps(\bar{M})\\
		\dot\sob^s_\locs(\bar{M})&\to \sob^s_\locs(\bar{M}).
	\end{align*}
	Moreover, the above extensions are surjective if $s<1/2$ and injective if $s >-1/2$, so they are in particular isomorphisms in the range $s\in(-1/2,1/2)$.
\end{lemma}

\subsection*{$b$-objects} 
There is (see Section 6 in \cite{alvarez_lopez_topology_2024}) another description of dual-conormal distributions, which emphasises the role of weights. For it, we define $b$-Sobolev spaces on $\bar{\R^n_+}$ in the same vein as the ordinary ones, however using the singular measure $dx/x$. That is, we have $L^2_b(\bar{\R^n_+})\equiv L^2(\frac{dx}{x}dy)=\sqrt{x}L^2(\bar{\R^n_+})$ and, for $s\in\N$,
\begin{equation}\label{eq:b-Sobolev spaces}
	\sob^s_b(\bar{\R^n_+})\equiv \{u\in L^2_b\colon \forall k\in\N,\alpha\in\N^{n-1}, k+|\alpha|\leq s\implies (x \partial_x)^k \partial_y^\alpha u\in L^2_b(\bar{\R^n_+})\}.
\end{equation}
The fractional and negative-order ones are defined as usual via interpolation and duality, and the intersection and union of them all are denoted, as usual, by $\sob^\infty_b$ and $\sob^{-\infty}_b$. Then, the weighted $L^2_b$-based Sobolev spaces are the spaces $x^l\sob^s_b$ for $l\in\R$, and one has
\begin{lemma}[see Corollary 6.49 in \cite{alvarez_lopez_topology_2024}]
	Let $K\subset\bar{\R^n_+}$ be compact. We have:
	\begin{enumerate}
		\item $\Dscr'_K(\bar{\R^n_+})=\bigcup_{\alpha\in\R}x^\alpha\sob^{-\infty}_{b,K}(\bar{\R^n_+})$;
		\item $\Ascr_K(\bar{\R^n_+})=\bigcup_{\alpha\in\R}x^\alpha \sob^{\infty}_{b,K}(\bar{\R^n_+})$;
		\item $\dot\Ascr'_K(\bar{\R^n_+})=\bigcap_{\alpha\in\R}x^\alpha \sob^{-\infty}_{b,K}(\bar{\R^n_+})$;
		\item $\Cdotinf_K(\bar{\R^n_+})=\bigcap_{\alpha\in\R}x^\alpha\sob^{\infty}_{b,K}(\bar{\R^n_+})$.
	\end{enumerate}
\end{lemma}

The analysis of singularities of solutions to boundary-value problems takes place in the $b$-cotangent bundle $\bTs\bar{M}$. We recall its definition in Appendix \cref{app:b-calculus}, whose notation we use here.

\subsection*{The Mellin transform}
The importance of $\b\Tstar\bar{M}$ resides in the fact that it is the carrier of the singularities of boundary value problems. Namely, for distributions $u\in\dot\Dscr'(\bar{M})$, there is a conic subset $\WF_b(u)\subset \b\Tstar\bar{M}$, the $b$-wave-front set, which describes the location and direction of certain singularities, and is analogous to the ordinary wave-front set. For certain classes of distributions, it can be defined similarly to $\WF(\cdot)$, using a mixed Mellin--Fourier transform. Our convention for these transforms is the following, $d$ denoting in each case the number of variables transformed:
\begin{align*}
	\FT(f)(\zeta)&=(2\pi)^{-d/2}\int_{\R^d}e^{-\im z\zeta}f(z)\de z,\\
	\Mcal(f)(\lambda)&=(2\pi)^{-d/2}\int_0^\infty x^{-\im\lambda-1}f(x)\de x,
\end{align*}
where $f\in\Cinf_c(\R^d)$ or $f\in\Cdotinf_c([0,\infty)^d)$, respectively. Here, the Mellin parameter $\lambda$ is complex, while the Fourier variables are real. However, even for the Mellin variable we are interested in the decay along horizontal lines $\Im\lambda=L$, in which case we write consistently $\lambda=\xi+\im L$ with $\xi=\Re\lambda$ and study the behaviour as $|\xi|\to\infty$.

We will not be very precise as to which are the maximal domains of $\FT$ and $\Mcal$ (we refer the interested reader to \cite{szmydt_mellin_1992} for an extensive discussion). We record below only the properties used in the sequel. In particular, if $u$ is a distribution of compact support (extensible or supported), the local Mellin--Fourier transform at a point $q\in\partial M$ is defined in local coordinates by the pairing (where we write $\hat{u}$ for $\Mcal\FT(u)$)
\begin{equation}\label{eq:mellin-fourier transform of extensible distribution}
	\hat{u}(\lambda,\eta)=(2\pi)^{-n/2}\braket{u,x^{-\im \lambda-1}e^{-\im y\eta}}.
\end{equation}
\begin{remark}\label{rem:mellin and supported distributions}
	The Mellin transform is not injective on $\dot\Dscr'(\bar{\R^n_+})$, since its kernel is exactly $\dot\Dscr'_{\{x=0\}}(\bar{\R^n_+})$: if $U,V\in\dot\Dscr'(\bar{\R^n_+})$ and $u=U-V\in\dot\Dscr'_{\{x=0\}}(\bar{\R^n_+})$, then near each point $p\in \{x=0\}$ we have $u=\sum_{k=0}^Na_k(y)\partial^k\delta(x)$ for some distributions $a_k$ on $\{x=0\}$. For large $\Im \lambda$, pairing this against $x^{-\im\lambda-1}$ makes sense and returns 0, irrespective of $\lambda$. As we will see below, $\Mcal u(\lambda)$ depends complex analytically on the parameter, therefore $\Mcal u$ is the zero function. Therefore, $\Mcal (U)$ only depends on the equivalence class of $U$ in $\Dscr'(\bar{\R^n_+})$. In view of \cref{lemma:conormal+dual-conormal=smooth}, injectivity is restored for dual-conormal distributions. This is the crux of why we can test the $b$-wave-front set via $\Mcal$.
\end{remark}
In \cref{lemma:properties of Mellin transform} we collect rather standard properties of $\Mcal$ and $\Mcal\FT$ on $\bar{\R^n_+}$.
\begin{lemma}\label{lemma:properties of Mellin transform}
	We have:
	\begin{enumerate}
		\item Whenever both sides make sense, $\Mcal[(xD_x)u](\lambda)=\lambda\hat{u}(\lambda)$ and, for all $\alpha\in\R$, $\Mcal[x^\alpha u](\lambda)=\hat u(\lambda+\im\alpha)$;
		\item $\Mcal\FT\colon x^\alpha\sob^s_b\to \braket{(\Re\lambda,\eta)}^{-s}L^2(\{\Im\lambda={-\alpha}\}\times\R^{n-1})$ is an isometric isomorphism, in particular $\Mcal\FT\colon L^2(\frac{dx}{x}dy)\xrightarrow{\sim}L^2(d\xi d\eta)$.
	\end{enumerate}
\end{lemma}
As for the Fourier transform, there are results of Paley--Wiener type that characterise the image of $\Mcal$ when restricted to certain spaces of compactly supported distributions. Since they are the motivation for the study of singularities, we recall below those that are useful to us. First, notice that, for a distribution on $\bar{\R_+}$ with compact support (whether extensible or supported), we always have that $\Mcal u$ is holomorphic in a half-plane $\Im\lambda>L$, with $L$ depending on the order of $u$. Thus, we interpret here $\Mcal u$ as a $\Dscr'(\R^{n-1})$-valued holomorphic function in some half-plane, namely, for every $\psi\in\Cinf_c(\R^{n-1})$ we have that $\braket{\Mcal u(\lambda),\psi}$ is a holomorphic function on $\{\Im \lambda>L\}$ for some $L\gg0$.

\begin{theorem}[Paley--Wiener for $\Escr'$]\label{theo:paley-wiener for extensible}
		The Mellin transform of a distribution $u\in\Escr'(\bar{\R^n_+})$ satisfies the following:
	\begin{enumerate}
		\item There exists $R>0$ such that, for all $\lambda$ for which $\Mcal u$ is defined, $\supp\Mcal u(\lambda)\subset B(0,R)\subset\R^{n-1}$;
		\item There exist constants $C_1, C_2, r, N>0$ such that
		\begin{equation}\label{eq:paley-wiener bound}
			\|\Mcal u(\lambda)\|_{\sob^{-r}(\R^{n-1})}\leq C_1e^{C_2|\Im \lambda|}(1+|\lambda|)^{N}.
		\end{equation}
	\end{enumerate}
	Conversely, if a $\Dscr'(\R^{n-1})$-valued function, holomorphic in a half-plane $\Im\lambda>K$, satisfies the two conditions above, then it is the Mellin transform of a unique $u\in\Escr'(\bar{\R^n_+})$.
\end{theorem}

\begin{theorem}[Paley--Wiener for $\Ascr$]\label{theo:paley-wiener for conormal}
	A distribution $u\in\Escr'(\bar{\R^n_+})$ is in $\Ascr(\bar{\R^n_+})$ if, and only if,
	its Mellin transform $\Mcal u$ takes values in smooth functions on $\R^{n-1}$ with rapid decay (together with all their derivatives) along every horizontal line in the half-plane on which it is holomorphic.
\end{theorem}

\begin{theorem}[Paley--Wiener for $\Ascr'$]\label{theo:paley-wiener for dual-conormal}
	A distribution $u\in\Escr'(\bar{\R^n_+})$ is in $\Ascr'(\bar{\R^n_+})$ if, and only if, $\Mcal u$ extends to a meromorphic function with simple poles at $-\im\N$, which has the property: for every $L>0$ we find $C,r>0$ such that if $|l|\leq L$
	\begin{equation}\label{eq:paley-wiener for dual-conormal}
		\left\|\prod_{j=0}^{\lfloor L\rfloor}(\xi+\im (l+j))\Mcal u(\xi+\im l)\right\|_{\sob^{-r}}\leq  C\braket{\xi}^{r}.
	\end{equation}
\end{theorem}

Table \cref{tab:mellin properties} gives an overview of the dictionary provided by $\Mcal$.
{\renewcommand{\arraystretch}{1.25}
	\setlength{\tabcolsep}{9pt}
\begin{table}[h]
	\centering
	\begin{tabular}{p{5cm}p{7cm}}
		\hline
		Property & Mellin transform \\
		\hline
		Extensible & distribution-valued holomorphic function in a half-plane and \eqref{eq:paley-wiener bound} \\
		\hline
		Conormal to $\{x=0\}$ & values in $\Cinf$+rapid decay \\
		\hline
		Dual-conormal to $\{x=0\}$ & Meromorphic with simple poles at $-\im\N$ + polynomial bound in $\Re\lambda$ \\
		\hline
		$\partial^k_x u$ vanishes at $\{x=0\}$ & removable singularity at $\lambda=-\im k$ \\
		\hline
	\end{tabular}
	\vspace{5pt}
	\caption{Mellin dictionary}
	\label{tab:mellin properties}
\end{table}}

For later purposes, we recall the following generalisation of \cref{theo:paley-wiener for dual-conormal}. In it, we treat more generally distributions on a quadrant $[0,\infty)_x^k\times \R^{n-k}_y$, as we will need to apply the result to a bidistribution in \cref{prop:properties of b-WF for HS-distributions}. The definitions of supported and extensible distributions apply verbatim to this situation, and so does the space $\dot\Ascr([0,\infty)^k\times\R^{n-k})$ of supported distributions, conormal to every hyperplane $x_j=0$, $j=1,\dots, k$. The dual space $\Ascr'([0,\infty)^k\times\R^{n-k})$ also has the same definition.
\begin{theorem}[see \cite{melrose_differential_nodate}, Proposition 4.8.2]\label{theo:mellin transform of A'}
	Assume $u\in\Ascr'([0,\infty)^2\times\R^{n-2})$ has compact support. Then its Mellin transform $\Mcal{u}$ defines a unique meromorphic function
	\begin{equation}\label{eq:dual-conormal has mermorphic Mellin}
		\C^2\setminus (-\im\N\times \C\cup\C\times -\im \N)\to \Escr'(\R^{n-2}),
	\end{equation}
	such that
	\begin{enumerate}
		\item for all $\lambda\in\C^2\setminus(-\im\N\times \C\cup\C\times -\im \N)$, $\supp \Mcal{u}(\lambda,\cdot)\subset B(0,R)\subset\R^{n-2}$ for some $R>0$ not depending on $\lambda$;
		\item for every $\epsilon>0$ there exists an integer $p$ and positive constants $C_1,N,C_2$ such that for all $\lambda, \mu$ with $\Im\lambda,\Im\mu\ge \epsilon >0$ we have
		\begin{equation}\label{eq:subexponential bound of Mellin transform}
			\|\Mcal(u)(\lambda,\mu,\,\cdot\,)\|_{\sob^{-p}}\leq C_1e^{C_2(|\Im \lambda|+|\Im\mu|)}\braket{(\Re\lambda,\Re\mu)}^{N}.
		\end{equation}
	\end{enumerate}
	In this case, on every strip $\{|\Im\lambda|,|\Im\mu|\leq L\}$ and denoting by $\eta$ the Fourier variable of $y$, the Fourier transform of $\prod_{j=0}^{\lfloor L\rfloor}(\lambda+\im j)(\mu+\im j)\Mcal(u)$, as a family of distributions on $\R^{n-2}$, is polynomially bounded in $\braket{(\Re\lambda,\Re\mu,\eta)}$.
\end{theorem}

\subsection*{The $b$-wave-front set}
We may now give
\begin{definition}\label{def:b-wave-front set}
	Let $u\in\Dscr'(\bar{M})$. We say that $(x_0,y_0,\xi_0,\eta_0)\notin\WF_b(u)$ if there exists a cutoff $\chi\in\Cinf_c(\bar{M})$, $\chi\equiv1$ near $(x_0,y_0)$, an open cone $\Gamma\subset \R^{n}\setminus\{0\}$ around $(\xi_0,\eta_0)$, and an $L>0$ such that $\Mcal (\chi u)(\lambda)$ is holomorphic in a neighbourhood of the line $\{\Im  \lambda=L\}$ and, for all $(\xi,\eta)\in \Gamma$ and all $N\in\N$, we find a constant $C_{N\chi\Gamma}$ such that
	\begin{equation}\label{eq:b-wave-front set}
		|\widehat{\chi u}(\xi+\im L,\eta)|\le C_{N\chi\Gamma} \braket{(\xi,\eta)}^{-N}.
	\end{equation}
\end{definition}

That the definition does not depend on the line $\Im\lambda=L$, as long as $\Mcal u$ is holomorphic there, is guaranteed by
\begin{lemma}\label{lemma:independence of the line}
	For some $\eps>0$ let $f\colon \{a-\eps< \Im z< b+\eps\}\to\C$ be a holomorphic function, polynomially bounded in $\Re z$ on the closed substrip $\{a\leq \Im z\leq b\}$. Assume furthermore that $f$ is rapidly decreasing along the line $\Im z=a$. Then $f$ is rapidly decreasing along every line $\Im z=l$ for $l\in (a,b)$.
\end{lemma}
\begin{proof}
	It suffices to prove the lemma for the case $a=0$, $b=1$, $L=0$. It also suffices to show that we have rapid decay along $\Im z=1/2$. Choose $k$ so that $|f(z)|\leq C|z+\im|^{k}$ is bounded on the strip (the existence of $k$ is guaranteed by the assumed polynomial control). We will apply Hadamard's three lines theorem to deduce that $f$ is rapidly decaying along every line in the strip $\Im z\in(a,b)$. For an arbitrary $N\in \N$, let $K>\max (N,k)$. On the strip, we have $1-\im z\neq 0$ and $|1-\im z|\sim 1+|\Re z|\sim\braket{\Re z}$. Setting for any $\alpha\in\R$
	\begin{equation}
		F_\alpha(z)\equiv f(z)(1-\im(z-\alpha))^{-K},
	\end{equation}
	we have that $F_\alpha$ is bounded and holomorphic on the strip. On the sides of the strip we estimate as follows. First, if $\Im z=0$, we know that $f(z)$ is rapidly decaying as $|z|\to\infty$. Second, Peetre's inequality gives that $\braket{x-\alpha}^{-K}\lesssim \braket{x}^K\braket{\alpha}^{-K}$. Thus we find
	\begin{align*}
		\sup_{\Im z=0}|F_\alpha(z)|&\lesssim \sup_{x\in\R}|f(x)(1+|x-\alpha|)^{-K}|\lesssim |f(x)|\braket{x-\alpha}^{-K}\\
		&\lesssim |f(x)|\braket{x}^K\braket{\alpha}^{-K}\lesssim \alpha^{-K}.
	\end{align*}
	On the other hand, for $\Im z=1$, notice that $\braket {x-\alpha}^{-K}=\braket{x-\alpha}^{-m}\braket{x-\alpha}^{m-K}$ with the second factor bounded. If we use Peetre's inequality to get $\braket{x-\alpha}^{-m}\lesssim \braket{x}^{-m}\braket{\alpha}^{m}$, we deduce
	\begin{align*}
		\sup_{\Im z=1}|F_\alpha(z)|&\lesssim \braket{x}^m \braket{x-\alpha}^{-K}\\
		&\lesssim \braket{x}^m\braket{x-\alpha}^{-m}\lesssim \braket{\alpha}^m.
	\end{align*}
	We now apply Hadamard's three-lines theorem to $F_\alpha$ along a line $\Im z=l\in (0,1)$ to see that
	\begin{equation*}
		|F_\alpha(\alpha+\im l)|\lesssim \braket{\alpha}^{-K(1-l)+ml}.
	\end{equation*}
	Since the factor $|(1-\im(z-\alpha))|^{-K}=|1+l|^{-K}$ is bounded in the strip, we obtain
	\begin{equation*}
		|f(\alpha+\im l)|\lesssim \braket{\alpha}^{-K(1-l)+ml}.
	\end{equation*}
	Since $K$ can be chosen arbitrarily large, we see that $f$ decays rapidly along the line $\Im z=l$ for any $l\in(0,1)$. This finishes the proof.
\end{proof}

We may also define the Sobolev $b$-wave-front set, as done in \cite{hintz_2016_quasilinear_wave}.
\begin{definition}
	Let $u\in\Dscr'(\bar{M})$. We say that $(x_0,y_0,\xi_0,\eta_0)\notin\WF_b^s(u)$ if there exists a cutoff $\chi\in\Cinf_c(\bar{M})$, $\chi=1$ near $(x_0,y_0)$, a conic cutoff $\psi$, $\psi=1$ near $(\xi_0,\eta_0)$, supported in an open cone $\Gamma\subset \R^{n}\setminus\{0\}$ around $(\xi_0,\eta_0)$, and an $L>0$ such that $\Mcal (\chi u)(\lambda)$ is holomorphic in a neighbourhood of $\{\Im \lambda=L\}$ and
	\begin{equation}\label{eq:Sobolev b-wave-front set}
		\psi(\Re\lambda,\eta)\braket{(\Re \lambda,\eta)}^s\widehat{\chi u}(\Re\lambda+\im L,\eta)\in L^2(\R^{n}).
	\end{equation}
\end{definition}
Notice that both the $b$-wave-front set and its Sobolev variant are conic subsets of $\bTs{\bar{M}}$, thus can be equivalently interpreted as subsets of the fibrewise projectivisation $\bSs{\bar{M}}=(\bTs\bar{M}\setminus O)\diagup \R^+$.
As for the ordinary wave-front set on boundaryless manifolds, the union of the Sobolev $b$-wave-front sets is dense in the $b$-wave-front set.
\begin{lemma}
	\label{lemma:ordinary and sobolev wave-front set}
	For any $u\in \Dscr'(\bar{M})$ we have
	\begin{align*}
		\WF_b^s(u)&\subseteq \WF_b^{r}(u),\quad s\leq r,\\
		\WF_b(u)&=\bar{\bigcup_{s\in\R}\WF_b^s(u)}.
	\end{align*}
\end{lemma}
\begin{proof}
	The monotonicity of the Sobolev wave-front sets is clear. Also clear is that $\WF_b(u)\supseteq \bar{\bigcup_{s\in\R}\WF_b^s(u)}$. To prove the reverse inclusion, consider a point $(q_0,p^0)\in\WF_b(u)$ and assume it does not belong to the right-hand side. Then, we find a conic neighbourhood of $(q_0,p^0)$ in $\b{T^\ast \bar{M}}$ which does not intersect any $\WF_b^s(u)$. That is, we find a cutoff $\chi$, supported near $q_0$, an $L>0$, and an open cone around $p^0$, both independent of $s$, such that for every $(\xi+\im L,\eta)\in\Gamma$ we have $\braket{(\xi,\eta)}^s\hat{\chi u}(\xi+\im L,\eta)\in L^2(\R^n)$ for every $s\in\R$. That is, $\chi u\in\sob^{\infty,L}$. In view of the analyticity of the Mellin--Fourier transform of $\chi u$, this implies rapid decay in $\Gamma$. That is, $(q_0,p^0)\notin\WF_b(u)$, a contradiction.
\end{proof}
\begin{remark}
	Notice that, when applying the Mellin transform to an element $u\in \Ascr'(\bar{M})$, we usually have a pole at 0, corresponding to the fact that these distributions do not necessarily vanish at the boundary (in any reasonable sense). Therefore, we would not have that $\Mcal u$ is square-integrable along $\Im \lambda=0$, as in Lemma 2.1 of \cite{hintz_2016_quasilinear_wave}. For distributions in $\dot\Ascr'(\bar{M})\equiv \bigcap_{\alpha}x^\alpha\sob_b^{-\infty}(\bar{M})$, our discussion above clarifies why the $b$-wave-front set is essentially independent of the weight, since $x^ \alpha\sob^{-\infty}_b\subset x^{\alpha'}\sob^{-\infty}_b$ if $\alpha'\leq \alpha$.
\end{remark}
We now show the equivalence of our definition with the more traditional one using pseudo-differential operators.
\begin{theorem}\label{theo:b-wave front set and pseudos}
	Let $u\in\Dscr'(\bar{\R^n_+})$. Then
	\begin{equation}
		\label{eq:b-wave-front set with pseudos}
		\WF_b(u)=\bigcap_{\substack{A\in \Psi_b(\bar{\R^n_+})\\Au\in\dot\Ascr(\bar{\R^n_+})}}\Char_b(A).
	\end{equation}
\end{theorem}
\begin{proof}
	The argument is modelled after the proof of the analogous equivalence for the ordinary wave-front set in Appendix A of \cite{shubin2001pseudodifferential}. We denote by $(q_0,p^0)$ a fixed point in the $b$-cotangent bundle.

	We begin by assuming that we have $\tilde A\in \Psi_b^0$ with $\sigma_b(\tilde A)(q_0,p^0)\neq 0$ and $\tilde Au\in\dot\Ascr$. First, by a microlocal parametrix construction, we may as well assume that there exist an open set $\tilde U\ni q_0$, lying entirely inside a coordinate patch $\Omega$ over which the $b$-cotangent bundle trivialises, a conic neighbourhood $\tilde{\Gamma}$ of $p^0$ and an operator ${A}\in \Psi_b^0$ such that the full symbol of $A$ equals 1 in $\tilde U\times \tilde{\Gamma}$ in the induced coordinates on $\Omega\times \R^n$. Indeed, if $B$ is a microlocal parametrix of $\tilde A$ in any small conic neighbourhood $\tilde U\times \tilde{\Gamma}$ of $(q_0,p^0)$, then $A=B\tilde A$ has full symbol equal to 1 in $\tilde U\times\tilde{\Gamma}$.

	Pick now $\chi\in\Cinf_c(\bar{\R^n_+})$ and $\psi\in\Cinf(\R^n)$ with $\chi\equiv1$ near $q_0$ and $\psi\equiv1$ near $p^0$. We may choose them such that
	\begin{equation}\label{eq:support of cutoffs for rapid decay}
		\supp (\chi\psi)\subseteq\{(x,y,\xi,\eta)\colon \sigma_b(A)(x,y,\xi,\eta)=1\}.
	\end{equation}
	It follows that in any coordinate system we have $\psi\chi\sigma_b(A)-\psi\chi\in \sym{-\infty}_b$, hence, if $M_\chi$ denotes the operator of multiplication by $\chi$,
	\begin{equation}
		\label{eq:quantisation of conic cutoff}
		\psi(\b D)M_\chi A=\psi(\b D)M_\chi\mod{\Psi^{-\infty}_b}.
	\end{equation}
	Since $Au\in\dot\Ascr$, we find using the calculus properties that $\psi(\b D)\chi Au\in\dot\Ascr$, so by the above also $\psi(\b D)\chi u\in\dot\Ascr$. Applying $\Mcal\FT$ to it we find then
	\begin{equation}\label{eq:rapid decay in a cone}
		\psi(\xi,\eta)\hat{\chi u}\quad \text{is rapidly decreasing},
	\end{equation}
	that is, $\hat{\chi u}$ decreases rapidly in the cone $\Gamma$ around $p^0$. This proves the inclusion $\subseteq$.

	On the other hand, assume we find a conic neighbourhood $\tilde U\times\tilde \Gamma$ of $(q_0,p^0)$ and functions $\chi\in\Cinf_c(\tilde U)$, $\chi(q_0)\neq 0$ and $\psi\in\Cinf(\tilde{\Gamma})$, $\psi(p^0)\neq 0$ and positively 0-homogeneous, such that together they witness the rapid decay of $\psi\hat{\chi u}$ in $\tilde U\times\tilde\Gamma$. Then, $\tilde a\equiv \chi\psi$ is a Hörmander symbol of order 0. As shown in \cite{hormander1994analysispseudodifferential}, we can find a symbol $a\in\sym{0}_{b}$ (see \cite{hintz_2016_quasilinear_wave} for the definition of the class) such that $\tilde{a}-a\in\sym{-\infty}$, in particular we may arrange $a$ to be equal to 1 in a smaller conic neighbourhood $U\times\Gamma$ of $(q_0,p^0)$. Again up to an error in $\sym{-\infty}$, we arrange that $A=\Op_{b}(a)$ is properly supported. Then $A$ is $b$-elliptic at $(q_0,p^0)$, equals $\Op_{b}(\psi)M_\chi$ up to an error in $\Psi^{-\infty}_b$, and is therefore microlocally the identity in the neighbourhood $U\times\Gamma$. Now, $\hat{Au}=\psi\hat{\chi u}+r$, where $r$ is the Mellin--Fourier transform of an element of $\dot\Ascr(\bar{\R^n_+})$. On the other hand, $Au$ is a compactly supported distribution of extensible type, therefore $\hat{Au}$ is holomorphic in a half-plane and satisfies the Paley--Wiener bound, in particular is polynomially bounded in $\{(\lambda,\eta)\colon\Im \lambda>K, (\Re\lambda,\eta)\in\Gamma\}$ for some $K$. Moreover, by assumption, it is rapidly decaying along a line $\Im\lambda =L$ with $L>K$ (hence, on all such lines, by \cref{lemma:independence of the line}). Thus, since $\psi=0$ outside $\Gamma$, we find that $\hat{Au}$ is rapidly decaying there. At the same time, we have rapid decay inside $U\times \Gamma$ by assumption. Thus, $\hat{Au}$ is rapidly decaying in all directions, which by \cref{theo:paley-wiener for conormal} means that the extensible class of $Au$ belong to $\Ascr(\bar{\R^n_+})$. That is, $Au\in\dot\Ascr(\bar{\R^n_+})$. Summing up, $(q_0,p^0)$ does not belong to the $b$-wave-front set of $u$. This shows the reverse inclusion and finishes the proof.
\end{proof}
\begin{corollary}\label{cor:uniform microlocal cutoff}
	Let $u\in \Ascr'(\bar{M})$, $(q_0,p^0)\in\bTs{\bar{M}}\setminus O$ and let $W$ be an open conic neighbourhood
	of $(q_0,p^0)$. Then there exist $L>0$ and
	\begin{itemize}
		\item a properly supported $A=\Op_{la}(a)\in\Psi^0_b(\bar{M})$
		with $\mu\text{-}\supp_b(A)\subset W$ and $a\equiv1$ on an open conic neighbourhood
		$W'\Subset W$ of $(q_0,p^0)$,
		\item a cutoff $\chi\in\Cinf_c(\bar{M})$ with $\chi\equiv1$ near
		$q_0$,
		\item a $0$-homogeneous $\psi$ with $\psi\equiv1$ near $p^0$ and
		$\supp\psi\subset\Gamma$ for some open cone $\Gamma$ around $p^0$,
	\end{itemize}
	such that, if $Au\in\Ascr(\bar{M})$, one has
	\begin{equation*}
		\sup_{(\xi,\eta)\in\Gamma}\braket{(\xi,\eta)}^{N}
		\bigl|\widehat{\chi u}(\xi+\im L,\eta)\bigr|<\infty
		\qquad\text{for every }N\in\N .
	\end{equation*}
\end{corollary}
\begin{proof}
	Independently of $u$, choose $A$, $\chi$ and $\psi$ as in the first half of the proof of \cref{theo:b-wave front set and pseudos}, so that $\psi(\b D)M_\chi A=\psi(\b D)M_\chi$ modulo $\Psi^{-\infty}_b(\bar{M})$. Since $Au\in\Ascr(\bar{M})$, we find $\psi(\b D)\chi u\in\Ascr(\bar{M})$, and \cref{theo:paley-wiener for conormal} gives the claimed rapid decay in $\Gamma$. On the other hand, $\Mcal (\chi u)$ is holomorphic in a half-plane, so that by \cref{lemma:independence of the line} its choice is immaterial. This completes the proof.
\end{proof}

With the characterisation of \cref{theo:b-wave front set and pseudos} we now have the following
\begin{lemma}
	\label{lemma:properties of b-wave-front set}
	If $u\in \Dscr'(\bar{M})$, it holds true:
	\begin{enumerate}[(a)]
		\item $\WF_b(u)=\emptyset$ if, and only if, $u\in\Ascr(\bar{M})$;
		\item If $T\in \Diff(\bar{M})$ or $T\in \Psi_b(\bar M)$, then $\WF_b(Tu)\subset \WF_b(u)$;
		\item If $T\in \Psi_b(\bar{M})$, then $\WF_b(u)\subset \WF_b(Tu)\cup \Char_b(T).$
		\item If $x$ is a boundary-defining function and $T$ is a differential operator of order $m$ with smooth coefficients on $\bar{M}$, then $\WF_b(u)\cap \bTs{}_{\partial M}\bar{M}\subset \WF_b(x^mTu)\cup \Tstar\partial M.$
		\item For $u,v\in\Dscr'(\bar{M})$ we have
		\begin{equation}\label{eq:b-wave-front set of sum}
			\WF_b(u+v)\subset \WF_b(u)\cup\WF_b(v).
		\end{equation}
	\end{enumerate}
\end{lemma}
\begin{proof}
	Except for the last part, this is part of Theorem 18.3.27 and Proposition 18.3.28 in \cite{hormander1994analysispseudodifferential}. For (e), notice first that $\Mcal\FT(u+v)$ is holomorphic on a half-plane and that rapid decay on one line propagates to every other line above it by \cref{lemma:independence of the line}. Then, if $\Im \lambda=L_u$ and $\Im\lambda=L_v$ are two lines along which $\Mcal\FT(u)$ and $\Mcal\FT(v)$ are respectively rapidly decreasing, we certainly have that both are rapidly decreasing on $\Im\lambda =\max\{L_u,L_v\}$. Thus, if a certain covector does not belong to either $\WF_b(u)$ or $\WF_b(v)$, then it also does not belong to $\WF_b(u+v)$. This finishes the proof.
\end{proof}

\begin{remark}
	\begin{enumerate}
		\item In view of \cref{lemma:conormal+dual-conormal=smooth}, (a) above can be improved if we know $u\in\Ascr'(\bar{M})$: it suffices to look at the characteristic sets of $A\in \Psi_b(\bar{M})$ such that $Au\in\Cinf(\bar{M})$, and therefore $\WF_b(u)=\emptyset$ if, and only if, $u\in\Cinf(\bar{M})$. This is the case, for example, if $Pu=0$ for some non-characteristic differential operator on $\bar{M}$, see \cref{cor:kernel of NHO is normally regular} below.
		\item Point (d) in \cref{lemma:properties of b-wave-front set} appears discouraging at first sight, since it seems that we cannot control the direction of singularities of solutions of differential equations along $\partial M$. However, notice that it does not take any boundary conditions into account. We will give a vast improvement over it in a moment.
	\end{enumerate}
\end{remark}

A reader well versed in distribution theory may have noticed that the original definition of $\WF_b$ in \cite{melrose1982transformationofboundaryproblems} and \cite{hormander1994analysispseudodifferential} is given for $u\in\dot\Dscr'(\bar{M})$, while our discussion is limited to extensible distributions. According to \cref{rem:mellin and supported distributions}, the Mellin transform does not detect $\dot\Dscr_{\partial M}'(\bar{M})$ and we would not expect to be able to control these singularities through it. However, for distributions in $\Ascr'(\bar{M})$ (and therefore also in the class $\Nscr(\bar{M})$ of the next subsection), this distinction is inessential:
\begin{lemma}
	\label{lemma:extensible=supported for dual-conormal}
	We have
	\begin{equation}\label{eq:dual-conormal meets no boundary distribution}
		\dot\Ascr'(\bar{M})\cap \dot\Dscr'_{\partial M}(\bar{M})=\{0\}.
	\end{equation}
	Consequently, the restriction map of
	\eqref{eq:exact sequence for supported/extensible distributions} restricts to a
	bijection $\dot\Ascr'(\bar{M})\xrightarrow{\ \cdot|_{M}\ }\Ascr'(\bar{M})$ (cf. \cite{alvarez_lopez_topology_2024}). Finally, if $u\in\Ascr'(\bar{M})$ and $U\in\dot\Ascr'(\bar{M})$ denotes the corresponding	supported representative, then
	\begin{equation*}
		\WF_b(u)=\WF_b(U).
	\end{equation*}
\end{lemma}
\begin{proof}
	The first part is shown in Chapter I, Section 3 of \cite{melrose1982transformationofboundaryproblems}, in the discussion preceding Proposition 3.28.

	Consider the restriction map in \eqref{eq:exact sequence for supported/extensible distributions}. After \cite{alvarez_lopez_topology_2024}, we show that it restricts to a bijection between $\dot\Ascr'(\bar{M})$ and $\Ascr'(\bar{M})$. The injectivity is clear: if $U_1, U_2\in\dot\Ascr'(\bar{M})$ are supported distributions restricting to the same $u\in\Ascr'(\bar{M})$, then $U_1-U_2\in \dot\Dscr'_{\partial M}(\bar{M})$, so that, by \eqref{eq:dual-conormal meets no boundary distribution}, it has to vanish. For the surjectivity, notice that any test function on $\bar{M}$ is automatically supported conormal to $\partial M$, with continuous embedding. Indeed, in adapted coordinates, the seminorms of $\dot\Ascr$ are the $L^\infty$-norms of $(x\partial_x)^k\partial^\alpha_y \phi$, which are clearly controlled by a finite number of $\Cinf_c$-seminorms. By definition, the pairing $\braket{u,\varphi}$ of $u\in\Ascr'(\bar{M})$ with $\varphi\in\Cinf_c(\bar{M})$ is then well-defined. It follows that every $u\in\Ascr'(\bar{M})$ has a \textit{canonical} supported representative which is itself dual-conormal.

	Finally, notice that $x^{-\im\lambda-1}e^{-\im y\eta}$ is a conormal distribution. Therefore, both $u\in\Ascr'(\bar{M})$ and its canonical representative $U\in\dot\Ascr'(\bar{M})$ can be paired with it. Clearly, $\hat{\chi U}=\hat{\chi u}$ then, which concludes the proof.
\end{proof}

\subsection*{Non-characteristic boundary problems}
Recall that a differential operator $P$ with smooth coefficients on a smooth manifold with boundary $\bar{M}$ is called \textit{non-characteristic} if  $\sigma(P)_{pr}|_{\nu^\ast\partial M}\neq 0$, where $\nu^\ast\partial M$ is the conormal bundle to the boundary.

Recall from \cite{hormander1994analysispseudodifferential} and \cite{melrose1982transformationofboundaryproblems} the following refinement of dual-conormality
\begin{definition}\label{def:normally regular distributions}
	A distribution $u\in\Ascr'(\bar{M})$ is \textit{normally regular}, written $u\in\Nscr(\bar{M})$, if $\WF_b(u)\subset T^\ast M\cup \Tstar\partial{M}$.
\end{definition}

The example below exemplifies the relation between $\Ascr'$ and $\Nscr$.
\begin{example}\label{example:A' vs N'}
	Consider the supported distribution $u=\sum_{j\geq 1}e^{-j^2}\delta_{1/j}$ on $[0,\infty)$. That this pairs continuously with $\phi\in\Cinf_c([0,\infty))$ can be seen directly: clearly the sum converges since we can estimate it by $\|\phi\|_{L^\infty}\sum e^{-j^2}$, moreover the pairing can be controlled by the $L^\infty$ norm of any finite number of derivatives (after Taylor expansion). Then, $\supp u=\singsupp u=\{0\}\cup\{1/j,j\in\N\}$, so $u$ is not a smooth function of $x$. On the other hand, after Mellin transforming $u$ we obtain
	\begin{equation*}
		|\braket{u,x^{-\im\lambda-1}}|\leq \sum_{j\geq 1}e^{-j^2}e^{(1-\Im\lambda)\log j},
	\end{equation*}
	which is bounded on any strip $|\Im\lambda|\leq L$. Therefore, according to \cref{theo:paley-wiener for dual-conormal}, $u\in\Ascr'([0,\infty))$.
\end{example}

The importance of $\Nscr(\bar{M})$ lies in the smoothness properties it has in conjunction with non-characteristic operators, exemplified in
\begin{theorem}[see \cite{melrose1982transformationofboundaryproblems}]
	\label{theo:smoothness and solvability in N'}
	Let $u\in\Dscr'(\bar{M})$. Then
	\begin{enumerate}
		\item we have $u\in\Nscr(\bar{M})$ if, and only if, there exists an open set $U\subset\bar{M}$, an $\eps>0$ and a boundary-defining function $x$ with $U\simeq\partial M\times[0,\eps)_x$, such that in local coordinates we have \[u\in\Cinf([0,\eps);\Dscr'(U\cap\partial M)),\]
		the latter meaning by definition that $\braket{u,\phi}$ is a smooth function of $x$ for every $\phi\in\Cinf_c(\partial M)$;
		\item if $u\in\Nscr(\bar{M})$ then for all $A\in \Psi_b(\bar{M})$ or $A\in\Diff(\bar{M})$ we have $Au\in\Nscr(\bar{M})$;
		\item if $Pu=f\in\Nscr(\bar{M})$ and $\partial M$ is non-characteristic for $P$, then $u\in\Nscr(\bar{M})$.
	\end{enumerate}
\end{theorem}
\begin{corollary}\label{cor:kernel of NHO is normally regular}
	For $P\in\Diff(\bar{M})$ with non-characteristic $\partial M$ we have $\ker(P\colon \Dscr'(\bar{M})\to\Dscr'(\bar{M}))\subset\Nscr(\bar{M})$.
\end{corollary}

If we set $\partial\WF(u)=\WF_b(u)\cap \b\Tstar_{\partial M}\bar{M}$, we thus have, for a $u\in\Nscr(\bar{M})$, a canonical decomposition
\begin{equation*}\label{eq:b- and boundary wave-front set}
	\WF_b(u)=\WF(u)|_{M}\dot\cup \partial\WF(u)\subset \Tstar M\dot\cup \Tstar\partial M.
\end{equation*}
The object $\partial\WF(u)$ is, for $u\in\Nscr(\bar{M})$, the \textit{boundary wave-front set} of \cite{chazarain_reflection_1976} and \cite{de_gosson_microlocal_1982}, and it measures the lack of \textit{strong tangential regularity} of the distribution $u$. Recall that for a distribution on $\bar{\R^n_+}$, we have that $u$ is smooth near $p$ if, and only if, $u$ is simultaneously normally smooth (i.e. $u\in\Nscr(\bar{\R^n_+})$) and \textit{strongly tangentially smooth} near $p$, where this concept means the following: there exists $\chi\in\Cinf_c(\bar{\R^n_+})$ with $\chi(p)\neq 0$, such that for every $N\in\N$ we find constants $C_{N,\chi}, \mu_{\chi}$ (the latter independent of $N$!) such that for every $(\xi,\eta)\in \R^n$ ($\xi$ the Fourier variable of $x$, $\eta$ the Fourier variables of $y$)
\begin{equation*}\label{eq:decay for strong tangential regularity}
	|\hat{\chi u}(\xi,\eta)|\leq C_{N,\chi}(1+|\xi|)^{\mu_\chi}(1+|\eta|)^{-N}.
\end{equation*}
With the notation above we also have
\begin{prop}[cf. Theorem 18.3.27 part (iv) of \cite{hormander1994analysispseudodifferential}]\label{prop:restriction and boundary WF set}
	Any $u\in\Nscr(\bar{M})$ has traces of all orders at $\partial M$, i.e. for every $k\in\N$, $\partial^k_{\nu} u$ exists as an element of $\Nscr(\bar{M})$ and restricts to $\partial M$ to give an element of $\Dcal'(\bar{M})$. Moreover, $\WF(\gamma_0 \partial^k_\nu u)\subset\bWF(u)$.
\end{prop}

As shown in \cite{hormander1994analysispseudodifferential}, Section 18.3, and \cite{de_gosson_microlocal_1982}, $\bWF(u)$ can be tested with partial Fourier transforms or with pseudo-differential operators along $\partial M$. In \cref{sect:ultrastatic} we shall use the definition of $\WF_b(u)$ instead, but it will be clear that a tangential Fourier transform would be sufficient (as it should, since our distributions will be of $\Nscr$ type).

In the class $\Nscr(\bar{M})$ we can take advantage of more specialised results concerning singularities. To phrase the next \cref{theo:b-WF with elliptic boundary conditions}, we recall the Lopatinskij--Shapiro condition associated with the boundary-value problem $(P,\gamma_0 B)u=0$, where $P$ and $B$ are differential operators on $\bar M$ with principal symbols $p$ and $b$. For a boundary point $(y,\zeta)\in T^\ast_{\partial M} \bar M$, denote by $X^+_{y,\zeta}$ the span of the set of smooth solutions of the equation $p(0,y,\zeta+D_s\nu_{(0,y)})u(s)=0$ which are bounded on $\R^+$ (in Definition 20.1 in \cite{hormander1994analysispseudodifferential}, this is denoted $M^+_{x,\zeta}$). In the terminology of \cite{benzoni-gavage_multi-dimensional_2007}, $X^+_{y,\zeta}$ is the union of the stable and central subspaces at the point $(y,\zeta)$ of the boundary value problem $(P,\gamma_0 B)$.
\begin{definition}\label{def:LS-condition and joint characteristic set}
	Let $(P,\gamma_0 B)$ be a boundary-value problem on the manifold with boundary $\bar{M}$ (that is, $P, B$ are differential operators with smooth coefficients on $\bar{M}$ and $\gamma_0$ is the restriction map to the boundary). Let $(y,\zeta)\in T^\ast_{\partial M}\bar{M}$ with $\zeta$ not proportional to the interior conormal $\nu_{(0,y)}$. We say that the boundary-value problem satisfies the Lopatinskij--Shapiro condition at $(y,\zeta)$ (or that the boundary-value problem is \textit{elliptic} at $(y,\zeta)$) if the map
	\begin{equation}\label{eq:Lopatinskij--Shapiro map}
		\LS_{y,\zeta}\colon X^+_{y,\zeta}\ni u\mapsto b(y,\zeta+D_s\nu_{(0,y)})u|_{s=0}\in \C
	\end{equation}
	is bijective. We introduce then the \textit{boundary characteristic set} of the pair $(P,\gamma_0B)$ as the subset $\Char(P;B)$ of $T^\ast \partial M$ where either \eqref{eq:Lopatinskij--Shapiro map} is not bijective or the point $(y,\zeta)$ has a preimage under $\varrho^t$ that is characteristic for $P$. Explicitly,
	\begin{equation}\label{eq:joint characteristic set}
		\Char(P;B):=\{(y,\eta)\in T^\ast \partial M\colon (\exists \zeta\in T^\ast_y \bar M, \varrho^t(\zeta)=\eta, p(y,\zeta)=0)\text{ or }(\ker\LS_{y,\zeta}\neq \{0\})\}.
	\end{equation}
\end{definition}

\begin{theorem}[cf.\ \cite{hormander1994analysispseudodifferential}, Theorem 20.1.14]
	\label{theo:b-WF with elliptic boundary conditions}
	If $P$ is a differential operator on the manifold $\bar{M}$ with non-characteristic boundary $\partial M$, $\gamma_0B$ is a local boundary condition for $P$, and $u\in \Nscr(\bar{M})$ is a solution of the homogeneous boundary problem $(P, \gamma_0B)u=0$, then we have
	\begin{equation}\label{eq:localisation of singularities with boundary conditions}
		\WF_b(u)\subset\Char (P)\dot\cup \Char(P;B).
	\end{equation}
\end{theorem}

We finish by showing that the Dirichlet boundary condition $B=\Id$ is, for the ordinary wave operator on the half-space $\R_{t}\times [0,\infty)_x\times\R^{n-1}_y$, elliptic in the directions $(\tau,0,\eta)$ with $\tau^2<|\eta|^2$ along the boundary.
\begin{theorem}\label{theo:DBC satisfies LS}
	Let $u\in \Nscr(\bar{\R^n_+})$ satisfy $\Box u\equiv(\partial^2_t+\Delta_{x,y}) u=0$ and $u|_{x=0}=0$, where $\Delta_{x,y}$ is $-\partial^2_{xx}-\sum_{j=1}^{n-1}\partial^2_{y^j y^j}$. We have $\WF_b(u)|_{x=0}=\partial\WF(u)\subset \{(t,0,y,\tau,0,\eta)\colon \tau^2\geq |\eta|^2\}$.
\end{theorem}
\begin{proof}
	By \cref{theo:b-WF with elliptic boundary conditions} we know $\WF_b(u)\subset\Char(\Box;\Id)$, so that we need to compute the joint characteristic set. To this end, we analyse in which directions the boundary-value problem $(\Box,\gamma_0)$ satisfies the Lopatinskij--Shapiro condition at a generic point $q\in \partial \bar{\R^n_+}=\{(t,0,y)\}$.

	The space $X^+_{q,(\tau,\xi,\eta)}$ is determined by finding the bounded solutions to the ODE $(\tau^2-(\xi+D_s)^2-|\eta|^2)u=0$, that is,
	\begin{equation}\label{eq:ODE for LS condition}
		\partial_s^2 u+2\im \xi \partial_su-(\xi^2+|\eta|^2-\tau^2)u=0.
	\end{equation}
	One calculates the discriminant to be $4(|\eta|^2-\tau^2)$, and analyses separately the cases:
	\begin{enumerate}
		\item $\tau^2< |\eta|^2$, in which case we have the solutions $e^{-s(\im\xi\pm\sqrt{|\eta|^2-\tau^2})}$;
		\item $\tau^2=|\eta|^2$, in which case the solutions are $e^{-\im s\xi}$ and $s e^{-\im s\xi}$;
		\item $\tau^2>|\eta|^2$, in which case we have the solutions $e^{-\im s (\xi+\sqrt{\tau^2-|\eta|^2})}$ and $e^{-\im s(\xi-\sqrt{\tau^2-|\eta|^2})}$.
	\end{enumerate}
	Therefore, the space $X^+_{q,(\tau,0,\eta)}$ is, in the three cases above, the span of
	\begin{enumerate}
		\item $e^{-s(\im \xi+\sqrt{|\eta|^2-\tau^2})}$ if $\tau^2<|\eta|^2$;
		\item $e^{-\im\xi s}$ if $\tau^2=|\eta|^2$;
		\item $e^{-\im s(\xi\pm\sqrt{\tau^2-|\eta|^2})}$ if $\tau^2>|\eta|^2$.
	\end{enumerate}
	The map $\LS_{q,(\tau,\xi,\eta)}$ is, for $B=\Id$, just evaluation at $s=0$, so that, if $(\tau,\eta)\neq 0$ (that is, the covector $(\tau,\xi,\eta)$ is not proportional to $\nu_q$), $\LS_{q,(\tau,\xi,\eta)}$ is a bijection in the first two cases. In the third case, we have a nontrivial kernel of $\LS$. However, the third region is also the image via the transpose of the anchor map $\varrho^t$ of $\{(q,(\tau,\xi,\eta))\colon -\tau^2+\xi^2+|\eta|^2=0\}$, the characteristic set of $\Box$. That is, the covectors in the third region are tangent covectors to the boundary having a characteristic preimage for $P$, lying by definition in $\Char(\Box;\Id)$. Of course, the covectors with $\tau^2=|\eta|^2$ are characteristic for $P$ too, when viewed as elements of $T^\ast \bar M$. Therefore, we have determined the joint characteristic set to be the region
	\begin{equation}
		\Char(\Box;\Id)=\{(q,(\tau,\eta))\in T^\ast \partial M\colon \tau^2\geq |\eta|^2 \},
	\end{equation}
	which concludes the proof.
\end{proof}
\begin{remark}
	After we have introduced the necessary geometric preliminaries in \cref{sect:GHSTB and IBVP}, \cref{theo:DBC satisfies LS} will say that the singularities of elements of the kernel of a normally hyperbolic operator are, at boundary points, constrained to the \textit{compressed causal cone} (see \cref{def:causal compressed b-covectors} for the definition).
\end{remark}

\section{Hilbert-space-valued distributions}
\label{sect:HS-valued distributions}

We generalise here the discussion of Hilbert-space-valued distributions and their wave-front sets of \cite{strohmaier_reeh-schlieder_2002} to the setting of manifolds with boundary and $b$-wave-front sets (notice also the related discussion of operator wave-front sets of \cite{wrochna_holographic_2017}). Let $\Hcal$ denote a complex Hilbert space with scalar product $(\cdot,\cdot)$, conjugate-linear in the first slot, and norm $\|\cdot\|$.
\begin{definition}\label{def:b-HS-valued distributions}
	Let $\bar{M}$ be a smooth manifold with boundary. A \textit{supported $\Hcal$-valued distribution}, $u\in\dot\Dscr'(\bar M;\Hcal)$ is a weakly continuous linear form $u\colon\Cinf_c(\bar{M})\to\Hcal$, that is, for every $h\in \Hcal$ the linear form $(h,u(\cdot))$ is in $\dot\Dscr'(\bar{M})$. An \textit{extensible $\Hcal$-valued distribution}, $u\in\Dscr'(\bar M;\Hcal)$, is a weakly continuous linear map $u\colon\Cdotinf_c(\bar{M})\to\Hcal$. The notation $\Fscr(\bar{M};\Hcal)$ shall be used for any other space of distributions introduced in \cref{sect:distributions on manifolds with boundary} to denote the obvious generalisation.
\end{definition}

In fact, by nuclearity of $\dot\Dscr'(\bar{M})$ and $\Dscr'(\bar{M})$, we can identify any closed subspace $\Fscr(\bar{M};\Hcal)$ of distributions (such as $\dot\Ascr(\bar{M})$ and $\Ascr'(\bar{M})$) with the completed tensor product $\Fscr(\bar{M})\hat\tens\Hcal$. This allows for a direct generalisation of many concepts from the complex to the $\Hcal$-valued case by just ``replacing absolute values with the Hilbert space norm''. See \cite{treves_topological_2006}, Chapter 50 for more details.

If $X$ is an open subset of $\bar{\R^n_+}$, Definition \cref{def:b-HS-valued distributions} reduces to a known form: $u\in\dot\Dscr'(X;\Hcal)$ (resp.\ $u\in\Dscr'(X;\Hcal)$) if, and only if, for every compact set $K\subset X$ there are constants $C>0$ and $k\in\N$ such that for all $f\in\Cinf_c(K)$ (resp.\ for all $f\in{\Cdotinf}_c(K)$) we have
\begin{equation}\label{eq:hilbert-space-valued local distributions}
	\|u(f)\|\le C\sum_{|\alpha|\le k}\sup_{q\in K}|\partial^\alpha f(q)|.
\end{equation}
As before, we use the notation $\Escr$ in place of $\Dscr$ to denote elements of compact support. Thus, $u\in\dot\Escr'(X;\Hcal)$ (resp.\ $u\in\Escr'(X;\Hcal)$) can be extended to a map on $\Cinf$ (resp.\ $\Cdotinf$), and will then satisfy \eqref{eq:hilbert-space-valued local distributions} for every compact $K$ containing the support of $u$. Since, for a given compact set $K$, any linear map satisfying the inequality above on $K$ is in fact a distribution with support in $K$, we have that the space $\dot{\Escr}'(X;\Hcal)$ (resp.\ $\Escr'(X;\Hcal)$) consists of all the continuous linear maps $\Cinf(X)\to\Hcal$ (resp.\ $\Cdotinf(X)\to\Hcal$). The topology of $\Escr'(X;\Hcal)$ is such that, if a set $B\subset\Escr'(X;\Hcal)$ is bounded, then the estimate \eqref{eq:hilbert-space-valued local distributions} holds true uniformly in $B$, that is, with constants $C,k$ not depending on $u$.

The action of $A\in \Psi_b(\bar{M})$ or $A\in\Diff(\bar{M})$ on $u\in\dot\Dscr'(\bar{M};\Hcal)$ is defined by duality: $Au$ is the $\Hcal$-valued distribution which satisfies, for all $\phi\in\Cinf_c(\bar{M})$ and all $v\in\Hcal$,
\begin{equation}\label{eq:pseudos acting on HS-distributions}
	\left(v,\braket{Au,\phi}\right)=\left(v,\braket{u,A^\ast \phi}\right).
\end{equation}
In view of \cref{lemma:continuity of b-operators on distributions}, the same definition makes sense for $u\in\Dscr'(\bar{M};\Hcal)$, provided we test $u$ only against $\phi\in\Cdotinf_c(\bar{M})$. More generally, using the identification as tensor product, we can transport all results for the action of differential or $b$-pseudo-differential operators from the complex to the $\Hcal$-valued case, letting an operator $A$ act on elements of the form $u\tens v$ by $A(u\tens v)=(Au)\tens v$ for $u\in\Dscr'(\bar{M})$, $v\in\Hcal$, and then extending by continuity. Thus, differential and $b$-pseudo-differential operators act continuously on $\dot{\Ascr}(\bar{M};\Hcal)$ and $\Ascr'(\bar{M};\Hcal)$ (cf. Appendix \cref{app:b-calculus}).

Similarly, we extend the definition of the Mellin transform for scalar $\Escr'(\bar{M})$-distributions to the $\Hcal$-valued case by setting, for all $h\in\Hcal$ and all $u\in\Escr'(\bar{M};\Hcal)$,
\begin{equation*}
	\label{eq:Mellin transform of H-valued distribution}
	(h,\Mcal(u)(\lambda))=(h,(2\pi)^{-1/2}\braket{u,x^{-\im\lambda-1}}).
\end{equation*}

We now introduce the $b$-wave-front set of distributions in $\Ascr'(\bar{M};\Hcal)$ in analogy to the scalar case. Here, as before, $\hat{u}$ stands for the combined Mellin--Fourier transform (locally, the first in the variables $x$ and the second in the variables $y$).

\begin{definition}\label{def:b-wave-front set of hilbert-valued distribution}
	Let $u\in\Ascr'(\bar M;\Hcal)$. We say that $(x_0,y_0,\xi_0,\eta_0)\in\b\Tstar\bar{M}$ does not belong to $\WF_b(u)$ (or that it is $b$-regular directed) if there exist a chart $\varkappa\colon U\to \bar{\R^n_+}$ near $(x_0,y_0)$, a cutoff $\chi\in\Cinf_c(\bar{M})$ with $\chi\equiv1$ near $(x_0,y_0)$, a conic neighbourhood $\Gamma\subset\R^{n}$ of $(\xi_0,\eta_0)$ and $L>0$ such that $\Mcal (\chi u)(\lambda)$ is holomorphic in a neighbourhood of the line $\{\Im\lambda=L \}$ and, for all $(\Re\lambda,\eta)\in \Gamma$ and all $N\in\N$,
	\begin{equation}\label{eq:rapid decay in a cone for HS-valued distribution}
		\|\widehat{\varkappa_\ast(\chi u)}(\xi+\im L,\eta)\|\lesssim \braket{(\xi,\eta)}^{-N}.
	\end{equation}
\end{definition}
We can then establish the following properties, analogous to Proposition 2.2 in \cite{strohmaier_reeh-schlieder_2002}.
\begin{prop}
	\label{prop:properties of b-WF for HS-distributions}
	Let $X$ be a smooth domain in $\bar{\R^n_+}$ and let $u\in\Ascr'( X;\Hcal)$. Then:
	\begin{enumerate}[(a)]
		\item If $u$ satisfies \eqref{eq:rapid decay in a cone for HS-valued distribution} for some $\chi\in\Cinf_c({X})$, then it also satisfies it for $f\chi$ for all $f\in\Cinf({X})$. In particular, $\WF_b(fu)\subset\WF_b(u)$.
		\item If $\supp u$ is compact, then $\hat{u}$ is smooth and polynomially bounded in the norm: there exist constants $C>0,m\in\Z$ such that
		\begin{equation}
			\label{eq:cptly suppd HS-valued has polynomially bdd FM transform}
			\|\hat{u}(\xi,\eta)\|\leq C\braket{(\xi,\eta)}^m.
		\end{equation}
		\item If $\WF_b(u)=\emptyset$, then $u\in\Cinf({X};\Hcal)$.
		\item Let $w\in \Ascr'( X\times X)$ be defined by $w(f,g)=\left({u(\bar f)}, u(g)\right)$. Then, denoting by $z=(x,y)$, $\zeta=(\xi,\eta)$ the canonical coordinates on $\b\Tstar X$,
		\begin{equation}\label{eq:wave-front set of bidistribution vs distribution}
			(z,\zeta)\in\WF_b(u)\iff ((z,-\zeta),(z,\zeta))\in \WF_b(w),
		\end{equation}
		and moreover if $\zeta_0\neq 0$ and $(z_0,\zeta_0)$ does not belong to $\WF_b(u)$, then for any $(z,\zeta)$ neither $((z_0,-\zeta_0),(z,\zeta))$ nor $((z,\zeta),(z_0,\zeta_0))$ belong to $\WF_b(w)$.
		\item $\WF_b(u)$ transforms as a subset of the $b$-cotangent bundle.
	\end{enumerate}
\end{prop}
\begin{proof}
	\begin{enumerate}[(a)]
		\item This is proven in the same way as for the ordinary wave-front set of ordinary distributions.

		\item That $\hat{u}$ is smooth follows by integration by parts using the standard argument for compactly supported distributions. For the bound, notice that, by definition, $u$ can be paired with $x^{-\im\lambda-1}e^{-\im y\eta}$. By continuity of $u$ on $\dot\Ascr(\bar{X};\Hcal)$, we find that for all compact $K\supset\supp u$ and all $s\in\R$ there exist constants $C>0$, $k\in\N$ such that
		\begin{align*}
			\|u(x^{-\im\lambda-1}e^{-\im y\eta})\|&\leq C\sum_{j=0}^k\left\|(x\partial_x)^j\sum_{|\alpha|\leq k-j}\partial^\alpha_y (x^{-\im\lambda-1}e^{-\im y\eta})\right\|_{\sob^s}\\
			&\leq C\sum_{j=0}^k\left\|(-\im\lambda-1)^j\sum_{|\alpha|\leq k-j}(-\im\eta)^\alpha e^{-\im y\eta}x^{-\im\lambda-1}\right\|_{\sob^s}\\
			&\leq \tilde C(1+|(\lambda,\eta)|)^k,
		\end{align*}
		where $\tilde C$ is a sufficiently large positive constant.

		\item If $\WF_b(u)=\emptyset$, then for every point $(x,y)\in{X}$ and every covector $(\xi,\eta)\in\bTs_{(x,y)}X$ we find $f\in\Cinf_c({X})$, $f(x,y)=1$, such that $\|\widehat{fu}(\lambda,\eta)\|$ is rapidly decreasing in a cone around $(\xi,\eta)$ in $\braket{(\Re\lambda,\eta)}$. By compactness of the $b$-cosphere bundle we find a single cutoff $\chi$, depending on $(x,y)$ but not on $(\xi,\eta)$, such that $\widehat{\chi u}$ is rapidly decreasing in all directions. Since $\chi u\in\Escr'(X)$, the Mellin--Fourier transform is holomorphic in a half-plane and satisfies there the Paley--Wiener bound
		\begin{equation}\label{eq:paley-wiener bound for HS-valued}
			\|\widehat{\chi u}(\lambda, \eta)\|\leq Ce^{A|\Im\lambda|}(1+|\Re\lambda|)^s.
		\end{equation}
		This means that $\chi u$ is actually a conormal distribution to $X\cap\partial \R^n_+$, see \cref{theo:paley-wiener for conormal}. Since $u\in\Ascr'({X};\Hcal)$ by assumption, we have $u\in\Cinf({X};\Hcal)$ in view of \cref{lemma:conormal+dual-conormal=smooth}.

		\item It is no restriction to assume that $w$ has compact support, since we can always localise $u$. That $w\in\Ascr'(X\times X)$ follows then by Mellin transform: we have that $\Mcal w(\bar{\lambda_1},\lambda_2)=\left(u(x_1^{\im\bar{\lambda_1}-1}),u(x_2^{-\im\lambda_2-1})\right)$, with values in $\Dscr'(\R^{2(n-1)})$, is holomorphic in the product $\{\Im\lambda_1>0\}\times\{\Im\lambda_2>0\}$ (recall that $(\cdot,\cdot)$ is conjugate-linear in the first slot). After being regularised by $\prod_{j=0}^{\lfloor L\rfloor}(\lambda_1+\im j)(\lambda_2+\im j)$, $w$ is polynomially bounded in $|\Im \lambda_1|,|\Im \lambda_2|\leq L$, since this is true for $u$. By \cref{theo:mellin transform of A'} applied to distributions in $X\times X$ we find $w\in\Ascr'(X\times X)$.

		Now assume $\zeta =(\xi,\eta)$ and $(z,-\zeta,z,\zeta)\notin\WF_b(w)$, so that we find $L>0$ and a function $f\in\Cinf_c({X}\times X)$ with $\widehat{fw}(-\xi+\im L,-\eta,\xi+\im L,\eta)$ rapidly decreasing in a conic neighbourhood $\Gamma\subset\R^{2n}_0$ of $(z,-\zeta,z,\zeta)$. We may choose $f=f_1\tens f_1$ for some positive function $f_1\in\Cinf_c(X)$, from which it follows
		\begin{equation*}\label{eq:rapid decay of M-F transform of u from decay of w}
			\|\widehat{f_1 u}(\xi+\im L,\eta)\|^2=|\widehat{(f_1\!\tens\! f_1)w}(-\xi+\im L,-\eta,\xi+\im L,\eta)|\lesssim \braket{\zeta}^{-N}.
		\end{equation*}
		That is, $(z,\zeta)$ is not in $\WF_b(u)$. On the other hand, suppose $(z_0,\zeta_0)\notin\WF_b(u)$, so that for an $L>0$ and an $f\in\Cinf_c({X})$ we have rapid decay of $\widehat{fu}(\xi+\im L,\eta)$ for $(\xi,\eta)$ in a conic neighbourhood $\Gamma$ of $(z_0,\zeta_0)$. By Cauchy--Schwarz
		\begin{equation*}\label{eq:rapid decay of M-F transform of w given that of u}
			|\widehat{(g\tens f)w}(\xi_1+\im L,\eta_1,\xi_2+\im L,\eta_2)|\leq \|u(\bar{g}x_1^{+\im\xi_1-1+L}e^{\im y\eta_1})\|\cdot\|u(fx_2^{-\im\xi_2-1+L}e^{-\im y\eta_2})\|,
		\end{equation*}
		where the second factor is rapidly decreasing for $(\xi_2,\eta_2)\in \Gamma$. Since, for every $g\in\Cinf_c(\bar{X})$, the first factor is polynomially bounded in the region where it is holomorphic, it follows that $(z,\zeta,z_0,\zeta_0)\notin\WF_b(w)$ for any $(z,\zeta)\in\b\Tstar {X}$. The same proof, exchanging the roles of $f$ and $g$, gives that $(z_0,-\zeta_0,z,\zeta)\notin\WF_b(w)$.

		\item This follows from the previous point.
	\end{enumerate}
\end{proof}
\begin{remark}
	After we introduce the setup of quantum field theory on curved spacetimes in \cref{sect:qft}, part (d) of the previous result will relate the $b$-wave-front set of the two-point function $\omega_2$ ($w$ from above) to that of the one-particle distribution $\Phi$ ($u$ in the Proposition).
\end{remark}
For the purpose of computation of $\WF_b$ for an $\Hcal$-valued distribution, we can sometimes use
\begin{theorem}\label{theo:weak and strong b-wave-front set}
	For any $u\in\Ascr'(\bar{M};\Hcal)$, let $u_v\in\Ascr'(\bar{M})$ be defined by $\braket{u_v,\phi}\equiv (v,\braket{u,\phi})$ for all $\phi\in\Cinf_c(\bar{M})$. Then
	\begin{equation}
		\WF_b(u)=\overline{\bigcup_{v\in \Hcal}\WF_b(u_v)}.
	\end{equation}
\end{theorem}
\begin{proof}
	It suffices to prove the statement locally. Also recall that, according to \cref{lemma:properties of b-wave-front set}, if $u\in\Ascr'(\bar{M})$ we have $\WF_b(u)=\emptyset$ if, and only if, $u$ is $\Cinf$ in the norm topology, and that $\Nscr(\bar{M})\subset\Ascr'(\bar{M})$. By nuclearity, the same is true for $\Hcal$-valued objects.

	Thus, if $(q_0,p^0)\notin\WF_b(u)$, meaning that we find an open cone $\Gamma$ and a cutoff $\chi$ with $\chi(q_0)\neq 0$ such that $\hat{\chi u}$ is rapidly decreasing in $\Gamma$, then, by Cauchy--Schwarz, each $\hat{\chi u_v}$ decays rapidly in the same cone $\Gamma$. That is, $(q_0,p^0)\notin\WF_b(u_v)$. Since both sets are closed, we find $(q_0,p^0)\notin \overline{\bigcup_{v\in \Hcal}\WF_b(u_v)}$.

	Conversely, let $(q_0,p^0)\notin \bar{\bigcup_{v\in\Hcal}\WF_b(u_v)}$. Then, there exists an open conic neighbourhood $W\subset\bTs{\bar{M}}\setminus O$ with
	\begin{equation}\label{eq:common regular region}
		W\cap\WF_b(u_v)=\emptyset\quad\text{for every }v\in\Hcal.
	\end{equation}
	Let $\chi$, $\psi$, $\Gamma$ and $L$ be as given by \cref{cor:uniform microlocal cutoff} and thus independent of $v$. Then we have
	\begin{equation}\label{eq:uniform rapid decay}
		\sup_{(\xi,\eta)\in\Gamma}\braket{(\xi,\eta)}^{N}
		\bigl|\widehat{\chi u_v}(\xi+\im L,\eta)\bigr|<\infty
		\quad \forall N\in\N,\, \forall v\in\Hcal.
	\end{equation}

	Fix then $N\in\N$. The Mellin--Fourier transform $\widehat{\chi u}(\xi+\im L,\eta)$ is an element of $\Hcal$ by definition, and for $(\xi,\eta)\in\Gamma$ we consider the conjugate-linear functionals
	\begin{equation*}
		T_{\xi,\eta}\colon\Hcal\to\C,\qquad
		T_{\xi,\eta}(v)\equiv\braket{(\xi,\eta)}^{N}\,\widehat{\chi u_v}(\xi+\im L,\eta)
		=\left(v,\braket{(\xi,\eta)}^{N}\widehat{\chi u}(\xi+\im L,\eta)\right).
	\end{equation*}
	Each $T_{\xi,\eta}$ is bounded, with $\|T_{\xi,\eta}\|=\braket{(\xi,\eta)}^{N} \|\widehat{\chi u}(\xi+\im L,\eta)\|$ as can be easily seen with the Cauchy--Schwarz inequality. That is, the family $\{T_{\xi,\eta}\}_{(\xi,\eta)\in\Gamma}$ is pointwise bounded. By Banach--Steinhaus, it is uniformly bounded as well, that is, we find $C_N>0$ with
	\begin{equation*}
		\bigl\|\widehat{\chi u}(\xi+\im L,\eta)\bigr\|\leq C_N\braket{(\xi,\eta)}^{-N},
		\qquad(\xi,\eta)\in\Gamma.
	\end{equation*}
	As $N\in\N$ was arbitrary, $(q_0,p^0)\notin\WF_b(u)$, which is the required inclusion.
\end{proof}
\begin{corollary}\label{cor:HS-valued N' is preserved}
	If $u\in\Nscr(\bar{M};\Hcal)$ and either $A\in \Psi_b(\bar{M})$ or $A\in \Diff(\bar{M})$, then $Au\in\Nscr(\bar{M};\Hcal)$. Similarly, if $\partial M$ is non-characteristic for $P\in\Diff(\bar{M})$ and $u\in \Dscr'(\bar{M};\Hcal)$ satisfies $Pu\in\Nscr(\bar{M};\Hcal)$, then $u\in\Nscr(\bar{M};\Hcal)$.
\end{corollary}
\begin{proof}
	Neither $\Psi_b(\bar{M})$ nor $\Diff(\bar{M})$ enlarge $\WF_b(u)$ for any scalar extensible distribution $u$ on $\bar{M}$. For $u\in\Nscr(\bar{M};\Hcal)$ and any $v\in\Hcal$ we have the scalar distribution $u_v=(v,u(\cdot))$, which is in $\Nscr(\bar{M})$ by definition and has then $\bWF(u_v)\subset T^\ast\partial M$. Then $Au_v=(Au)_v\in\Nscr(\bar{M})$ by \cref{theo:smoothness and solvability in N'} (2), so that by \cref{theo:weak and strong b-wave-front set} we find $\bWF(Au)\subset T^\ast\partial M$, i.e. $Au\in\Nscr(\bar{M};\Hcal)$. For the second statement, the proof is the same using \cref{theo:smoothness and solvability in N'} (3).
\end{proof}

\section{Timelike boundaries and normally hyperbolic operators}
\label{sect:GHSTB and IBVP}

\subsection*{Spacetimes with timelike boundaries} We summarise here the relevant results of \cite{ake_hau_structure_boundary_2020} on spacetimes with timelike boundary and on the conditions under which such a Lorentzian manifold can be called globally hyperbolic. Subsequently we discuss boundary value problems for normally hyperbolic operators on them.

\begin{definition}\label{def:GHSTB}
	\begin{enumerate}
		\item A \textit{spacetime with timelike boundary} is an oriented and time-oriented Lorentzian manifold $\bar{M}$, having a boundary $\partial M$, embedded into $\bar{M}$ as a timelike hypersurface (that is, the normal vector field to $\partial M$ is everywhere spacelike). In particular, the boundary is itself an oriented time-oriented Lorentzian manifold.
		\item If $\bar M$ is a spacetime with timelike boundary and $p, q\in \bar{M}$, one says that $q$ belongs to the \textit{chronological} (resp.\ \textit{causal}) future of $p$, $I^+(p)$ (resp.\ $J^+(p)$), if $p$ can be joined to $q$ by a future-directed timelike (resp.\ causal) $\sob^1$-curve (see \cref{rem:causality and H1-curves}). One denotes by $I^+(K)$ (resp.\ $J^+(K)$) the union of the chronological (resp.\ causal) futures of all points in the set $K\subset \bar{M}$. The same definitions with roles reversed define the chronological and causal past of $q$, namely the sets $I^-(q)$ and $J^-(q)$.
		\item One calls a spacetime with timelike boundary \textit{causal} if there are no closed future-directed causal curves and \textit{globally hyperbolic} if it is causal and for any $p,q\in \bar M$ the sets $J^+(p)\cap J^-(q)$ are compact.
	\end{enumerate}
\end{definition}

\begin{remark}\label{rem:causality and H1-curves}
	\begin{enumerate}
		\item The standard definitions of timelike/lightlike/spacelike/causal vectors and submanifolds carry over to this setting without change.
		\item The definition of $J^{\pm} (p)$ uses the notion of $\sob^1$-curves, namely curves whose local parametrisations are in the Sobolev space $\sob^1_{\locs}(\R)$. On the one hand, this ensures that some desirable generalisations of properties of the boundaryless case, like the existence of Cauchy surfaces, remain true. On the other hand, a maximally extended causal curve on $\bar{M}$ will not be smooth in general, since it can hit the timelike boundary and be reflected. To define $I^+(p)$, we can as well use smooth curves.
	\end{enumerate}
\end{remark}

\begin{prop}[cf. \cite{ake_hau_structure_boundary_2020}]\label{prop:properties of ghstb}
	Let $\bar M$ be a globally hyperbolic spacetime with timelike boundary. Then
	\begin{enumerate}
		\item $\bar{M}$ admits a \textit{Cauchy temporal function}, namely a function $t\in\Cinf(\bar{M})$ with future-directed timelike gradient whose level sets are spacelike hypersurfaces which are intersected exactly once by any inextensible causal $\sob^1$-curve (i.e. Cauchy surfaces). In addition, $t$ can be chosen such that $\nabla t$ is tangent to $\partial M$.
		\item $\partial M$ is a globally hyperbolic spacetime (without boundary), and we can choose the restriction of $t$ as a Cauchy temporal function on it.
		\item $M$ is a smooth causal spacetime, and we can choose the restriction of $t$ as a temporal function on it.
		\item $\bar M$ is isometric to $\R\times \bar\Sigma$ with a metric $g=\beta\de t^2-g_t$, where $\bar \Sigma$ is a Riemannian manifold with boundary, $\beta\colon\R\times\bar \Sigma\rightarrow\R$ is a positive function and $g_t$ is a smooth family of Riemannian metrics on $\bar \Sigma$. Any slice $\bar \Sigma_a=t^{-1}(a)$ is a Cauchy surface. Correspondingly, $\partial M\cong \R\times\partial \Sigma$ and $M\cong\R\times\Sigma$ with the restricted metrics.
		\item The previous property and the existence of a Cauchy surface (with boundary) are both equivalent to global hyperbolicity.
	\end{enumerate}
\end{prop}
\begin{definition}
	\label{def:past and future compact}
	Let $\bar{M}$ be a globally hyperbolic spacetime with timelike boundary. For $K\subset\bar{M}$ we write $J(K):=J^+(K)\cup J^-(K)$.
	A subset $A\subset\bar{M}$ is called
	\begin{enumerate}
		\item \textit{spatially compact} if there exists a compact set $K\subset\bar{M}$
		such that $A\subset J(K)$;
		\item \textit{past-compact} resp.\ \textit{future-compact} if $A\cap J^-(K)$ resp.\
		$A\cap J^+(K)$ is compact for every compact set $K\subset\bar{M}$;
		\item \textit{strictly past-compact} resp.\ \textit{strictly future-compact} if there
		exists a compact set $K\subset\bar{M}$ such that $A\subset J^+(K)$ resp.\
		$A\subset J^-(K)$.
	\end{enumerate}
	A distribution $u\in\dot{\Dscr}'(\bar{M})$ is called spatially compact, (strictly)
	past-compact, or (strictly) future-compact if its support $\supp u$ has the
	corresponding property. If $\Fscr$ is a space of distributions on $\bar{M}$, we
	denote by $\Fscr_{\scs}$, $\Fscr_{\pc/\fc}$ and
	$\Fscr_{\spc/\sfc}$ the subspaces of those elements of $\Fscr$ which are
	spatially compact, past-/future-compact, and strictly past-/future-compact.
\end{definition}

\subsection*{The Klein--Gordon equation}
Let $P$ be a scalar, formally self-adjoint, normally hyperbolic operator on $\bar{M}$, i.e. $P=\Box_g+V$ for $V\in\Cinf(\bar{M},\R)$, with principal symbol $p(z,\zeta)=g_z(\zeta,\zeta)$ for $z\in\bar{M}$ and $\zeta\in T^\ast_z\bar{M}$. With the regularities to be specified shortly and denoting by $\n$ the unit normal vector field to $\bar\Sigma$, we analyse the Cauchy--Dirichlet problem
\begin{equation}\label{eq:IBVP}
	\left\{\begin{aligned}
		Pu&=f\quad \text{on }M,\\
		u&=h\quad \text{on }\partial M,\\
		u&=u_1,\quad \partial_\n u=u_2\quad\text{on }\bar\Sigma,
	\end{aligned}\right.
\end{equation}
and its retarded variant (i.e. $f,v$ are past-compact and we look for a past-compact solution $u$)
\begin{equation}\label{eq:retarded BVP}
	\left\{\begin{aligned}
		Pu&=f\quad \text{on }M,\\
		u&=v\quad \text{on }\partial M.
	\end{aligned}\right.
\end{equation}
We will thus, in what follows, analyse the boundary-value problem $(P,\gamma_0)$ as defined by $\Cinf(\bar{M})\to\Cinf(\bar{M})\oplus\Cinf(\partial M)$, $u\mapsto (Pu,u|_{\partial M})$ and its extensions.

The retarded problem, in the case of spatially compact $\bar{M}$, is the subject of Theorem 24.1.1 in \cite{hormander1994analysispseudodifferential}, which we quote:
\begin{theorem}
	\label{theo:existence and uniqueness}
	Assume $\bar{M}$ is spatially compact. If $f\in \sob^s_{\locs,\pc}(\bar{M})$ and $v\in\sob^{s+1}_{\locs,\pc}(\partial M)$, there exists a unique $u\in\sob^{s+1}_{\locs,\pc}(\bar{M})$ that solves \eqref{eq:retarded BVP}.
\end{theorem}
\begin{corollary}\label{cor:smooth existence}
	If $f\in \Cinf_{\pc}(\bar{M})$ and $v\in \Cinf_{\pc}(\partial M)$, then there exists a unique $u\in\Cinf_{\pc}(\bar{M})$ satisfying $Pu=f$ in ${M}$ and $u|_{\partial M}=v$. Moreover, the assignment $(f,v)\mapsto u$ is continuous.
\end{corollary}
\begin{proof}
	The existence is an immediate consequence of \cref{theo:existence and uniqueness}, using that $\bigcap_s\sob^s_{\locs}(\bar{M})=\Cinf(\bar{M})$ and $\bigcap_s\sob^s_{\locs}(\partial M)=\Cinf(\partial M)$. The continuity follows since the topology of $\Cinf(\bar{M})$ can be defined using $\sob^s$-norms for large $s$.
\end{proof}
We now discuss an extension of this to the non-spatially compact case. The result we aim at is
\begin{theorem}\label{theo:propagators}
	Let $\bar{M}$ be a globally hyperbolic spacetime with timelike boundary. For every $f\in\Cinf_\pc(\bar{M}), v\in\Cinf_\pc(\partial M)$ we find a unique $u\in\Cinf_{\pc}(\bar{M})$ such that $Pu=f$ and $u|_{\partial M}=v$. Moreover, we have
	\begin{equation}\label{eq:support of retarded solution}
		\supp u\subset J^+(\supp f\cup \supp v).
	\end{equation}
\end{theorem}
The local existence of a solution is by now classical, by using energy estimates. The global statement will follow from finite propagation speed by piecing together local solutions given by \cref{theo:existence and uniqueness}. We will follow the main idea in Chapter XXIV of \textcite{hormander1994analysispseudodifferential}, with a small difference concerning the domain in which we integrate, to establish an energy estimate that will provide the support properties of solutions. A more general statement is given by \textcite{ginoux_cauchy_friedrichs_2022}, but their notion of weak solution (Definition~4.1 in the reference) does not depend on the Cauchy datum, since admissible test sections are required to vanish at both $\bar\Sigma_0$ and $\bar{\Sigma}_T$; we give an independent proof below.

We translate first the discussion in Chapter XXIV in \cite{hormander1994analysispseudodifferential} into an invariant framework and give a few supplementary facts. For a point $z\in\bar{M}$, a complex covector $\zeta\in T^\ast_z\bar{M}$ and real covectors $\alpha,\beta\in T^\ast_z\bar{M}$, define the quadratic form in $\zeta$ (the \textit{stress-energy tensor})
\begin{equation}\label{eq:stress-energy tensor}
	T_{\zeta}(\alpha,\beta)\equiv 2\Re\left(p(\zeta,\alpha)\bar{p(\zeta,\beta)} \right)-p(\bar \zeta,\zeta)p(\alpha,\beta).
\end{equation}
Notice that $T$ only depends on the principal symbol $p$ and not on lower order terms. Moreover, by polarisation we have $T_{a+\im b}=T_a+T_b$, so it mostly suffices to consider the case of real $\zeta$.
\begin{lemma}[cf. \cite{hormander1994analysispseudodifferential}, Lemma 24.1.2]
	We have
	\begin{enumerate}
		\item If $\beta$ is future-directed timelike and $\alpha$ is future-directed causal, then $T_\zeta(\alpha,\beta)$ is positive semidefinite, and positive definite exactly if $\alpha$ is timelike.
		\item Same as (1) but with $\alpha,\beta$ past-directed.
		\item If $T_\cdot(\alpha,\beta)$ is positive definite, then $\alpha, \beta$ are necessarily timelike and in the same connected component of the causal cone.
	\end{enumerate}
\end{lemma}
\begin{proof}
	The only addition with respect to the reference is the statement about the positive semidefinite case. This is a quick consequence: since $T_\zeta(\cdot,\beta)$ is linear, the set
	\begin{equation*}
		\{\alpha\in T^\ast_z\bar{M}\colon T_\zeta(\alpha,\beta)\ge 0\}
	\end{equation*}
	is a closed cone, containing the future-directed timelike directions in view of the statement for future-directed timelike covectors. Since the future causal cone is the closure of the future chronological cone, the claim follows.
\end{proof}
\begin{lemma}\label{lemma:stress-energy tensor for conormal direction}
	Let $z\in\partial M$, $U$ a neighbourhood of $z$ in $\bar{M}$ and $\nu_z$a (spacelike) exterior conormal covector to $\partial M$. Let $u\in C^1(U)$ be such that $u|_{\partial M\cap U}=0$ so that $\de u_z=\sigma\nu_z$ for some $\sigma\in\C$. Then, if $L=\beta^\sharp$ (with respect to $g$),
	\begin{equation}\label{eq:stress-energy tensor for conormal direction}
		T_{\de u}(\nu_z,\beta)=|\sigma|^2p(\nu_z,\nu_z)\braket{\nu_z,L}.
	\end{equation}
	In particular, if $\sigma\neq 0$, then $T_{\de u}(\nu_z,\beta)\ge 0$ if, and only if, $L$ does not point outward from $\bar{M}$, with equality precisely for $L$ tangent to $\partial M$.
\end{lemma}
\begin{proof}
	For any $\sigma\in\C$, a direct computation with $\zeta=\sigma \nu_z$ shows that
	\begin{equation*}
		T_{\sigma\nu_z}(\nu_z,\beta)=|\sigma|^2p(\nu_z,\nu_z)p(\nu_z,\beta)=|\sigma|^2p(\nu_z,\nu_z)\braket{\nu_z,L},
	\end{equation*}
	where $L=\beta^\sharp$. With this formula, the rest of the proof is quick: if $L$ is tangent to $\partial M$, then $\braket{\nu_z,L}=0$, and since $p(\nu_z,\nu_z)<0$ we have that the sign of $T$ is positive exactly when $L$ points inward. This finishes the proof.
\end{proof}
For $u\in C^1(\bar{M})$ and $L$ a tangent vector with $L=\beta^\sharp$ we define the current $\Jcal[u]$ associated with $T$ and $L$ by
\begin{equation}\label{eq:current of SE tensor}
	\Jcal[u]\equiv 2\Re\left(\bar{\braket{\de u,L}}\nabla u\right)-p(\de \bar u,\de u)L,
\end{equation}
or equivalently, for every real covector $\alpha$, by $\braket{\alpha,\Jcal[u]}=T_{\de u}(\alpha,\beta)$. The next lemma is proven in \cite{hormander1994analysispseudodifferential} in the discussion preceding Lemma 24.1.2.
\begin{lemma}\label{lemma:divergence of SE current}
	For every $u\in C^2(\bar{M})$ there exists a complex quadratic form $R$ in $(u,\de u)$, with coefficients depending only on $\nabla L$ and on the potential $V$, such that
	\begin{equation}\label{eq:divergence of SE tensor}
		\diverg \Jcal[u]=2\Re\left(\bar{\braket{\de u,L}}Pu\right)+R(u,\de u).
	\end{equation}
	Moreover, $R$ is continuous and for any compact $K\subset\bar{M}$ there is $C_K>0$ such that
	\begin{equation}\label{eq:bounded quadratic remainder}
		|R(u,\de u)|\leq C_K(|u|^2+|\de u|^2).
	\end{equation}
\end{lemma}
In the next \cref{lemma:slab-cone is a Lipschitz domain}, a quick consequence of Proposition 3.12 and Corollary 3.13 of \cite{ake_hau_structure_boundary_2020}, we recall that the intersection of a past cone with a time slab is (the closure of) a Lipschitz domain.
\begin{lemma}\label{lemma:slab-cone is a Lipschitz domain}
	Consider the metric splitting $\bar{M}\simeq\R_t\times\bar{\Sigma}$, let $z\in\bar{M}$ and let $a<b$ be real numbers. Then $J^-(z)\cap t^{-1}[a,b]$ is a Lipschitz domain in $\bar{M}$, irrespective of whether it intersects $\partial M$ and of the value of $t(z)$.
\end{lemma}
We now prove a more precise support estimate for solutions of the retarded Dirichlet problem for $P$.
\begin{theorem}[Finite propagation speed]\label{theo:support of solutions}
	Let $f\in\Cinf_\pc(\bar{M})$, $v\in\Cinf_\pc(\partial M)$. Assume that $u\in\Cinf_\pc(\bar{M})$ solves $Pu=f$ on $M$ and $\gamma_0u=v$ on $\partial M$. Then,
	\begin{equation}
		\supp u\subset J^+(\supp f\cup\supp v).
	\end{equation}
\end{theorem}
\begin{proof}	
	Let $K\equiv \supp f\cup \supp v$ and fix $z\notin J^+(K)$, i.e. $J^-(z)\cap K=\emptyset$. The claim is that $u(z)=0$. If $J^-(z)\cap \supp u=\emptyset$, then $u(z)=0$ so there is nothing to prove. Assume therefore that the intersection is nonempty. Choose a time function $t$ as in \cref{prop:properties of ghstb}, with $\supp f, \supp v\subset t^{-1}([a,\infty))$ for some $a\in\R$ and assume without loss of generality that $t(z)\ge a$, otherwise the claim is trivial.

	Since $\supp u\cap J^-(z)$ is compact, we may choose $a'<\min_{q\in\supp u\cap J^-(z)}t(q)$. Thus, the source and boundary datum vanish in $J^-(z)$ and $u$ vanishes with all its derivatives near $t=a'$ (inside $J^-(z)$). Then, for every $\tau\in[a',t(z)]$ consider the compact set
	\begin{equation}\label{eq:Lipschitz domain for energy estimate}
		\Omega_\tau\equiv J^-(z)\cap t^{-1}[a',\tau],
	\end{equation}
	and notice as well that, by \cref{lemma:slab-cone is a Lipschitz domain}, $\Omega_\tau$ is a Lipschitz region. Its boundary is given by the union of the four pieces
	\begin{equation}\label{eq:boundary of Lipschitz domain}
		\begin{aligned}
			T&=J^-(z)\cap \bar\Sigma_\tau,\quad B=J^-(z)\cap \bar\Sigma_{a'},\\
			W&=\Omega_\tau\cap \partial M,\quad C=(\partial J^-(z)\cap t^{-1}[a',\tau])\setminus W.
		\end{aligned}
	\end{equation}
	Depending on the relative position of the objects involved, it can happen that $W$ is a set of zero measure in $\partial \Omega_\tau$. In this case, the boundary condition is inessential and the result holds by the theory on boundaryless spacetimes. Otherwise, since $f$ and $v$ vanish on $J^-(z)$, we know that $u$ satisfies $Pu=0$ on $\Omega_\tau$ and $\gamma_0 u=0$ on $W$. Moreover, since $f, v$ and $u$ vanish on $t^{-1}(-\infty,a)$, we find an open neighbourhood $U$ of $\bar\Sigma_{a'}$ on which also $\de u$ vanishes.

	Now we integrate \eqref{eq:current of SE tensor} over $\Omega_\tau$ and apply the divergence theorem (see \cite{mclean_strongly_2000}, Theorem 3.34, for a version for Lipschitz domains). Since $Pu=0$ in $\Omega_\tau$ we have (with $\de S$ denoting the surface measure)
	\begin{equation}\label{eq:divergence theorem for SE current}
		\begin{aligned}
			\int_{\Omega_\tau} R(u,\de u)\de\vol_g&=\int_{\Omega_\tau}\diverg \Jcal[u]\de\vol_g=\int_{\partial\Omega_\tau}\braket{\nu,\Jcal[u]}\de S\\
			&=\int_{T}\braket{\nu,\Jcal[u]}\de S+\int_B\braket{\nu,\Jcal[u]}\de S+\int_C\braket{\nu,\Jcal[u]}\de S+\int_W\braket{\nu,\Jcal[u]}\de S.
		\end{aligned}
	\end{equation}
	Since $u$ and $\de u$ vanish in a neighbourhood of $\bar\Sigma_{a'}$, the contribution from $B$ is zero. Similarly, by \cref{lemma:stress-energy tensor for conormal direction}, the integral over $W$ is also zero. The very same result guarantees that the integral over $C$ is nonnegative, so we can neglect it to obtain, for the energy $E(\tau)=\int_T \Jcal[u]\de S$, the inequality
	\begin{equation}\label{eq:energy inequality 1}
		0\leq E(\tau)\leq \int_{\Omega_\tau} R(u,\de u)\de\vol_g\leq c\int_{\Omega_\tau}|u|^2+|\de u|^2\de\vol_g,
	\end{equation}
	where we also used the continuity of $R$ to find the constant $c$.

	Now, since the exterior conormal to $T$ is simply $\de t$, we have $E(\tau)=\int_T T_{\de u}(\de t,L^\flat)$ with $T_{\de u}(\de t,L^\flat)$ positive-definite. Then we find a uniform constant $c'$ such that $|T_{\de u}(\de t,L^\flat)|\ge c'|\de u|^2$. Consider then
	\begin{equation}\label{eq:energy at time tau}
		e(\tau)=\int_{\bar\Sigma_\tau\cap J^-(z)}|u|^2+|\de u|^2\de S,
	\end{equation}
	which is an integrable function in view of Fubini. For a constant $c''$ we then have
	\begin{equation}\label{eq:estimate for du in the energy}
		\int_{\bar\Sigma_\tau\cap J^-(z)}|\de u|^2\de S\leq c''\int_{a'}^\tau e(s)\de s.
	\end{equation}
	A similar argument controls $\int_{\Omega_\tau}|u|^2\de\vol_g$: since $u$ vanishes on $\bar\Sigma_{a'}$, the fundamental theorem of calculus applied along the interval $[a',\tau]$ along an integral curve of $\nabla t$ gives, for every $x\in \bar{\Sigma}_\tau\cap J^-(z)$,
	\begin{equation}
		u(\tau,x)=\int_{a'}^\tau\partial_s u(s,x)\de s,
	\end{equation}
	whence by Cauchy--Schwarz we deduce that, for a constant $c^{(3)}$ independent of $x\in\bar\Sigma_\tau\cap J^-(z)$,
	\begin{equation}\label{eq:pointwise estimate of u on Sigma_tau}
		|u(\tau,x)|^2\le c^{(3)}\int_{a'}^\tau|\de u_{(s,x)}|^2\de s.
	\end{equation}
	Integrating this estimate with respect to the Riemannian volume over $\bar\Sigma_\tau\cap J^-(z)$ leads to
	\begin{equation}
		\int_{\bar\Sigma_\tau\cap J^-(z)}|u|^2\de S_{\tau}\le c^{(4)}\int_{a'}^\tau \int_{\Sigma_s\cap J^-(z)}|\de u_{(s,x)}|^2\de x\de s,
	\end{equation}
	which is again controlled by the integral of $e$. Thus, with a new constant $\tilde c$, we have
	\begin{equation}
		e(\tau)\le \tilde{c}\int_{a'}^\tau e(s)\de s,
	\end{equation}
	with $e(a')=0$. By Gronwall's lemma (applicable since $e$ is integrable and nonnegative) we deduce $e=0$, which in turn means that $u=0$ on $\bar\Sigma_\tau\cap J^-(z)$. Since $\tau$ was arbitrary in $[a',t(z)]$, this implies $u(z)=0$ and completes the proof.
\end{proof}
\begin{proof}[Proof of \cref{theo:propagators}]
	Since the spatially compact case is covered by \cref{cor:smooth existence}, assume that $\bar{M}$ is spatially noncompact.

	\textit{Step 1:} We begin by showing that to every $f\in\Cinf_c(\bar{M})$ and every $v\in\Cinf_c(\partial{M})$ there is a retarded solution $u$ of $Pu=f$ on $M$ and $\gamma_0 u=v$. Denote $C\equiv \supp f\cup\supp v$, choose a splitting $\bar{M}=\R_t\times\bar{\Sigma}$ with noncompact Cauchy surface $\bar\Sigma$ and with a time function $t$ such that $C\subset\{t>a\}$ for some $a$. For any $T>a$ consider the subset (compact by global hyperbolicity)
	\begin{equation*}
		K_T\equiv J^+(C)\cap t^{-1}[a,T+1]
	\end{equation*}
	and denote by $S_T$ the projection of $K_T$ onto $\bar\Sigma$, also compact. Now we choose a compact smooth submanifold with boundary $\bar{U}_T$ of $\bar{\Sigma}$ such that
	\begin{enumerate}[i)]
		\item $S_T\subset\bar U_T$;
		\item $\partial \bar U_T$ is smooth and there is a neighbourhood of $S_T\cap\partial \Sigma$ in $\partial \Sigma$ on which $\partial U_T$ coincides with $\partial \Sigma$;
		\item the remaining portion of $\partial \bar U_T$ has positive distance from $S_T$.
	\end{enumerate}
	That is, $\partial\bar U_T$ is the union of a ``genuine'' part $\partial U^{gen}_T\subset \partial \Sigma$ and of an ``artificial'' part $\partial U^{art}_T\subset \Sigma$. Such a submanifold can be constructed by first taking an $\epsilon$-neighbourhood of $S_T$ in $\bar{\Sigma}$ and then smoothing out the corners that this creates away from $S_T\cap\partial\Sigma$.

	Let now $\bar{M}_T\equiv \R\times \bar{U}_T$ equipped with the restriction of the metric $g$. This is a spatially compact spacetime with timelike boundary (the normal to the boundary is clearly spacelike), with the same Cauchy time function $t$. For $v_T\equiv v$ on $\partial \bar U^{gen}_T$ and zero elsewhere on $\partial\bar U_T$, consider then the auxiliary problem
	\begin{equation}
		\left\{\begin{aligned}
			Pu&=f \quad \text{on }M,\\
			\gamma_0 u&=v_T.
		\end{aligned}\right.
	\end{equation}
	Being set up on a spatially compact globally hyperbolic spacetime with timelike boundary, this problem has a unique retarded solution $u_T\in\Cinf_{\pc}(\bar{M}_T)$, with support controlled by
	\begin{equation}\label{eq:support control in auxiliary spacetime}
	\supp u_T\subset J^+_{\bar M_T}(C)\subset J^+_{\bar{M}}(C).
	\end{equation}
	Indeed, every causal curve in $\bar{M}_T$ is also causal in $\bar{M}$. Moreover, by construction, $\partial\bar U_T^{art}$ is disjoint from $J^+_{\bar{M}}(C)\cap t^{-1}(-\infty,T+1]$, so that \eqref{eq:support control in auxiliary spacetime} gives that $u_T$ vanishes in a neighbourhood of $(-\infty,T]\times \partial U^{art}_T$, so it can be extended by zero across it. This shows that $u_T$ is a solution of the original problem on $(-\infty,T]\times \bar{\Sigma}$.

	If we were to choose a different neighbourhood $\bar{U}_T'$ or a different $T'>T$, the constructed solution would not change on the smaller domains. Indeed, for example (the case of a different $T'$ is completely analogous), if $\bar{U}'_T$ is a (say larger) submanifold of $\bar{\Sigma}$ satisfying the same properties as $\bar{U}_T$, and $u_T'$ is the solution constructed by the above procedure on it, then the difference $w=u_T-u_T'$ satisfies $Pw=0$ and $\gamma_0 w=0$ on the smaller domain $\bar U_T$. By \cref{theo:support of solutions}, $w=0$ on $\bar U_T$ which shows that the solution is independent of the choice of $\bar{U}_T$. Then, choose a monotone sequence $(T_j)\to\infty$ and define
	\begin{equation}\label{eq:solution to noncompact problem}
		u(q)\equiv u_{T_j}(q) \quad \text{for any }j\text{ with }T_j>t(q).
	\end{equation}
	By construction, $u$ is a well-defined smooth function, solving $Pu=f$ and $\gamma_0 u=v$ on $\bar{M}$. By the support estimate of \cref{theo:support of solutions} we have $\supp u\subset J^+(C)$, so that $u$ is past-compact (in fact strictly so in this case).

	\textit{Step 2:} We now let $f\in\Cinf_{\pc}(\bar{M})$ and $v\in\Cinf_\pc(\partial M)$ be general. Denote again by $C$ the union of the supports of $f$ and $v$. If $K$ is any compact set in $\bar{M}$, consider the compact set $C_K\equiv C\cap J^-(K)$ and choose a cutoff $\chi_K\in\Cinf_c(\bar{M})$, equal to one in a neighbourhood of $C_K$. Then the problem with compactly supported data
	\begin{equation}
		\left\{\begin{aligned}
			Pu&=\chi_K f,\\
			\gamma_0 u&=(\gamma_0 \chi_K)v,
		\end{aligned}\right.
	\end{equation}
	has, by Step 1, a unique retarded solution $u_K$, whose restriction to $K$ is independent of the choice of $\chi_K$. Indeed, if $\chi_K'$ were another cutoff with the same properties, the difference of data would vanish on $J^-(K)$, so that by uniqueness/finite propagation speed the solutions would coincide on $K$.

	Take now a compact exhaustion $\bigcup_{j}K_j=\bar{M}$ and define
	\begin{equation}
		u(q)\equiv u_{K_j}(q),\quad \text{if }q\in K_j.
	\end{equation}
	The function $u$ is well-defined by the argument above, solves the required problem on $\bar{M}$, and has past-compact support by construction. This finishes the proof.
\end{proof}
\begin{remark}
	The proof of \cref{theo:propagators} above works as well for compactly supported distributional sources $f\in\dot\Escr'(\bar{M})$ and produces distributional solutions $u\in\dot{\Dscr}'_{\pc}(\bar{M})$. Indeed, local existence and uniqueness are obtained by duality from \cref{theo:existence and uniqueness}, using that $P$ is formally self-adjoint. For the global statement, notice that a distribution can be restricted to open subsets of $\bar{M}$ and that this is all that is required to apply the exhaustion argument.
\end{remark}
Clearly, the time-orientation can be reversed in all arguments above, so that Theorems \cref{theo:existence and uniqueness} and \cref{theo:propagators} have analogues for future-compact data and produce future-compact solutions. Therefore, we may assert
\begin{corollary}
	\label{cor:fundamental solution for retarded BVP}
	There exist bounded operators $\tilde G^\pm_D\colon \Cinf_{\pc/\fc}(\bar{M})\oplus\Cinf_{\pc/\fc}(\partial M)\to\Cinf_{\pc/\fc}(\bar{M})$ that are inverses to $(P,\gamma_0)$. More precisely,
	\begin{align*}
		(P,\gamma_0)\tilde G^+_D(f,v)&=(f,v),\qquad \tilde G^+_D(P,\gamma_0)u=u,\qquad u,f,v\in\Cinf_{\pc},\\
		(P,\gamma_0)\tilde G^-_D(f,v)&=(f,v),\qquad\tilde G^-_D(P,\gamma_0)u=u,\qquad u,f,v\in\Cinf_{\fc}.
	\end{align*}
\end{corollary}

Notice that, by finite propagation speed and the theory of normally hyperbolic operators on spacetimes, we have that $\tilde G^+_D$ actually maps $\Cinf_{c}(\bar{M})\oplus \Cinf_c(\partial M)$ to $\Cinf_{\spc}(\bar{M})$, and similarly $\tilde G^-_D\colon\Cinf_c(\bar{M})\oplus \Cinf_c(\partial M)\to\Cinf_{\sfc}(\bar{M})$. We call $\tilde G^+_D$ and $\tilde G^-_D$ the retarded and advanced inverses to $(P,\gamma_0)$. We write
\begin{equation*}
	\tilde G_D\equiv \tilde G^+_D-\tilde G^-_D
\end{equation*}
for the difference of the two inverses. If we let $\Cinf_{D,\scs}(\bar{M})$ denote the set of smooth, spatially compact functions satisfying the homogeneous Dirichlet boundary condition, we can interpret $\tilde G_D$ as a map $\Cinf_c(\bar{M})\oplus \Cinf_c(\partial M)\to \Cinf_{D,\scs}(\bar{M})$:
\begin{itemize}
	\item the support of $\tilde G_D(f,h)$ is spatially compact by the argument in \cref{theo:propagators};
	\item $u=\tilde G_D(f,h)$ satisfies $\gamma_0u=0$ automatically and independently of $h$, indeed
	\begin{align*}
		\gamma_0u&=\gamma_0\left(\tilde G^+_D(f,h)- \tilde G^-_D(f,h)\right)=\gamma_0\tilde G^{+}_D(f,h)-\gamma_0\tilde G^-_D(f,h)\\
		&=h-h=0.
	\end{align*}
\end{itemize}
With these preliminaries, we have (cf. \cite{dappiaggi_maxwell_timelike_2020} for the analogous statement for differential forms, and \cite{baer_wave_lorentzian_2007} for the classical boundaryless version)

\begin{lemma}\label{lemma:exact sequence}
	The following sequence is exact:
	\begin{equation}
		\label{eq:exact sequence}
		0\to \Cinf_c(\bar{M})\xrightarrow{(P,\gamma_0)}
		\begin{array}{c}
			\Cinf_c(\bar{M}) \\
			\oplus \\
			\Cinf_c(\partial M)
		\end{array}
		\xrightarrow{\tilde G_D}\Cinf_{D,\scs}(\bar{M})\xrightarrow{P}\Cinf_{\scs}(\bar{M}).
	\end{equation}
\end{lemma}
\begin{proof}
	The proof is a quick adaptation of the proof for the boundaryless case by employing the propagator $\tilde G_D$.
	\begin{enumerate}
		\item Exactness at $\Cinf_c(\bar{M})$ amounts to injectivity of $(P,\gamma_0)$ on this domain. However, $(P,\gamma_0)$ is invertible on the larger domain $\Cinf_{\pc}(\bar{M})$ with inverse $\tilde G^+_D$, therefore its restriction to $\Cinf_c(\bar{M})$ is certainly injective.

		\item Exactness at $\Cinf_c(\bar{M})\oplus\Cinf_c(\partial M)$ corresponds to having $\ran((P,\gamma_0)|_{\Cinf_c(\bar{M})})=\ker \tilde G_D$.
		To see this, first let $u\in\Cinf_c(\bar{M})$ and consider $(f,v)=(Pu,\gamma_0u)\in\Cinf_c(\bar{M})\oplus\Cinf_c(\partial M)$. Then $\tilde G_D(f,v)=\tilde G_D^+(f,v)-\tilde G^-_D(f,v)=u-u=0$, having used that $\Cinf_c(\bar{M})\subset\Cinf_{\pc}\cap\Cinf_{\fc}$.
		On the other hand, assume that for some $(f,v)\in\Cinf_c(\bar{M})\oplus\Cinf_c(\partial M)$ we have $\tilde G_D(f,v)=0$, that is $\tilde G_D^+(f,v)=\tilde G_D^-(f,v)$. The right-hand side being strictly future-compact and the left-hand side strictly past-compact, we have that $u\equiv \tilde G^+_D(f,v)$ is compactly supported, and by definition it satisfies $Pu=f,\gamma_0u=v$. This proves the claim.

		\item We need to show $\ran(\tilde G_D)=\ker P|_{\Cinf_{D,\scs}}$. First, if $u=\tilde G_D(f,v)$, we already remarked that $(P,\gamma_0)\tilde G_D^\pm(f,v)=(f,v)$, so that the difference of the two vanishes. For the opposite inclusion, let $u\in\Cinf_{D,\scs}(\bar{M})$ satisfy $Pu=0$. We fix a past cutoff function $\chi\in\Cinf_{\pc}(\bar{M})$, with $1-\chi\in\Cinf_{\fc}(\bar{M})$, and consider $Pu=P(\chi u)+P((1-\chi)u)=f^++f^-$, where $f^+\in\Cinf_{\pc}$, $f^-\in\Cinf_{\fc}$. Since $Pu=0$, we have $f^+=-f^-$, moreover $\chi u=\tilde G^+_D f^+$ and $(1-\chi)u=\tilde G^-_Df^-$. Then, $f\equiv f^+$ satisfies
		\begin{align*}
			\tilde G_Df&=\tilde G^+_Df^+-\tilde G^-_Df^+=\tilde G^+_Df^++\tilde G^-_Df^-\\
			&=\chi u+(1-\chi)u=u,
		\end{align*}
		proving the claim.
	\end{enumerate}
\end{proof}

\begin{remark}\label{rem:supports and exact sequence}
	In \eqref{eq:exact sequence} it is important to consider functions whose compact supports \textit{can} intersect the boundary. Indeed, the formulation above must be invariant under the time evolution, so that the function spaces themselves have to be. However, as remarked in the Introduction, $\Cinf_c(M)$ is not invariant, as under time evolution the supports propagate and eventually reach the boundary. Notice as well that exactness at $\Cinf_c(\bar{M})$ corresponds to the classical fact that the only solution to the wave equation with vanishing source and compact support is the zero function. Here, of course, the boundary condition plays the role of restoring this injectivity.
\end{remark}

In what follows, we fix the notation $G^\pm_D(f)\equiv \tilde{G}^\pm_D(f,0)$ and refer to $G_D\equiv G_D^+- G_D^-$, seen as a map $\Cinf_c(\bar{M})\to\Cinf_{D,\scs}(\bar{M})$, as the \textit{Dirichlet Pauli--Jordan propagator} or simply the \textit{Pauli--Jordan propagator} of $P$.

We can reformulate the exact sequence as follows
\begin{equation}
	\label{eq:homogeneous exact sequence}
	0\to \Cinf_{D,c}(\bar{M})\xrightarrow{P}
	\Cinf_c(\bar M)
	\xrightarrow{G_D}\Cinf_{D,\scs}(\bar{M})\xrightarrow{P}\Cinf_{\scs}(\bar{M}).
\end{equation}
As a corollary, we find the usual quotient characterisation of the solution space (with the notation $\ker_D P=\ker(P|_{\Cinf_{D,\scs}})$)
\begin{align*}
	\label{eq:solution space as a quotient}
	\ker_D P&=\ran G_D\cong \setquotient{\Cinf_c(\bar{M})}{\ker  G_D}\\
	&=\setquotient{\Cinf_c(\bar{M})}{P\Cinf_{D,c}(\bar{M})}
\end{align*}
For the construction of states we also need
\begin{lemma}
	\label{lemma:G_D is a symplectic form}
	For $[u],[v]\in \Cinf_c(\bar M)/P\Cinf_{D,c}(\bar M)$ define the bilinear form
	\begin{equation}\label{eq:symplectic form on solutions}
		\omega([u],[v])=\int_{\bar{M}}u G_Dv \de\vol_g.
	\end{equation}
	Then $\omega$ is symplectic, i.e. skew-symmetric and nondegenerate. It induces therefore a symplectic form on $\ker_D P$.
\end{lemma}
\begin{proof}
	Since the operator $P$ is formally self-adjoint and solutions satisfy the homogeneous Dirichlet boundary condition, the inverses of $P$ satisfy $(G^\pm_D)^t= G^\mp_D$ on $\Cinf_c(\bar{M})$. Hence, on its domain, $G_D$ is skew-symmetric, $G_D^t=-G_D$. On the other hand, let $[u]\in\Cinf_{c}(\bar{M})/P\Cinf_{D,c}(\bar{M})$ satisfy $\omega([u],[v])=(u,G_D v)=0$ for all $v\in \Cinf_{c}(\bar{M})$. Then, after transposing, we must have $G_Du=0$ by nondegeneracy of the inner product of $L^2(M)$. By \cref{lemma:exact sequence} we find $w\in \Cinf_{D,c}(\bar{M})$ with $u=Pw$, so $[u]=0$. This finishes the proof.
\end{proof}

\subsection*{Singularities}
The hypothesis that $\partial M$ is a timelike hypersurface implies that it is non-characteristic for any normally hyperbolic operator $P$ on $\bar{M}$. Therefore, the analysis of \cref{sect:distributions on manifolds with boundary} applies. We begin however with some remarks.

In the specific case of globally hyperbolic spacetimes with timelike boundaries, notice that it does not seem possible to canonically define the causal relation on the whole of $\b\Tstar\bar{M}$. However (as already remarked by \textcite{vasy_propagation_singularities_corners_2008} and exploited by \textcite{gannot_propagation_2022} on asymptotically AdS spacetimes) there is a natural definition of causality on the compressed $b$-cotangent bundle ${\tilde T^\ast \bar{M}}\equiv\varrho^t(T^\ast \bar{M})\subset \b\Tstar \bar{M}$. Indeed, the Lorentzian metric on $\bar{M}$ and the induced one on $\partial M$ allow us to give the following
\begin{definition}
	\label{def:causal compressed b-covectors}
	Fix $q\in\partial M$. A covector $\alpha\in{\tilde T_q^\ast}\bar M$ is called causal if it is the image of a causal covector $\beta\in T_q^\ast \bar{M}$ under $\varrho^t$. A covector $\alpha\in\tilde T^\ast_q\bar{M}$ is called spacelike if it is not causal. Equivalently, $\alpha$ is spacelike if all its preimages via $\varrho^t$ are spacelike covectors in $T^\ast_q\bar{M}$.

	For $\alpha$ a causal covector in ${\tilde T_q^\ast}\bar{M}$, we say that it is \textit{future-directed} (resp.\ past-directed) if it is the image of a future-directed (resp.\ past-directed) causal covector in $\Tstar_q\bar{M}$ under $\varrho^t$. Equivalently, for some Cauchy temporal function $t$ chosen near $q$ and some causal covector $\beta\in T^\ast_q\bar{M}$ with $\varrho^t(\beta)=\alpha$, we have $g^{-1}_q(\beta,\de t)>0$. The set of compressed causal covectors over all points of $\bar{M}$, $\tilde{V}\bar{M}$, is a cone bundle, which splits, after removing the zero section, into two disjoint connected components $\tilde V^+\bar{M}$ and $\tilde V^-\bar{M}$, comprising respectively all future-directed and past-directed compressed causal covectors over points of $\bar{M}$.

	The above definitions are invariant along rays in the cotangent directions in $\b\Tstar \bar{M}$, thus the sets of future- and past-directed causal covectors are conic. The same terminology will therefore be used to speak about oriented directions in the \textit{compressed cosphere bundle} $\compcosphere{\bar{M}}$, i.e. the quotient by the scaling action.
\end{definition}
\begin{remark}\label{rem:compressed causal cone}
	The reason why we do not distinguish between time- and lightlike covectors in ${\tilde T^\ast}\bar{M}$ is that, under the anchor map, both some lightlike and some timelike covectors have the same image in $\b\Tstar\bar{M}$. This will however not cause any difficulty, since we are only interested in the fact that the singularities of a solution $u$ are constrained to keep their time-orientation. In fact, over a boundary point $q$, the image via $\varrho^t$ of the light cone $\{(\tau,\xi,\eta)\in T^ \ast_q\bar{M}\colon \tau^2=\xi^2+|\eta|^2\}$ is exactly the compressed causal cone $\tilde{V}_q\bar{M}=\{(\tau,0,\eta)\in\tilde T_q^\ast \bar{M}\colon \tau^2\geq |\eta|^2\}$.
\end{remark}

If it exists, a solution of the boundary-value problem $(P,\gamma_0)$ is unique in the class $\Nscr_{\pc}(\bar{M})$ (and in $\Nscr_{\fc}(\bar{M})$ too, with a completely analogous argument).
\begin{lemma}\label{lemma:distributional uniqueness for the retarded BVP}
	Let $u\in\Nscr(\bar M)$ be past-compact and satisfy $Pu=0$, $\gamma_0u=0$.
	Then $u=0$.
\end{lemma}
\begin{proof}
	By \cref{theo:smoothness and solvability in N'} (1), $u$ has well-defined traces
	$\gamma_ju=(\partial_x^ju)|_{x=0}\in\Dscr'(\partial M)$, and by
	\cref{lemma:extensible=supported for dual-conormal} it has a unique supported
	representative, so $\braket{u,\phi}$ is defined for all $\phi\in\Cinf_c(\bar M)$.

	Let $\chi\in\Cinf_c(\bar M)$ and set $w=G^-_D\chi\in\Cinf_{\sfc}(\bar M)$, so that
	$Pw=\chi$, $\gamma_0w=0$, $\supp w\subset J^-(\supp\chi)$. As $\supp u$ is
	past-compact and $\supp w$ strictly future-compact, $\supp u\cap\supp w$ is
	compact; pick $\theta\in\Cinf_c(\bar M)$ with $\theta\equiv1$ near it and put
	$W=\theta w\in\Cinf_c(\bar M)$, which still satisfies $\gamma_0W=0$. Since
	$PW=\theta\chi+[P,\theta]w$ and $\supp([P,\theta]w)\cap\supp u=\emptyset$, we get
	$\braket{u,\chi}=\braket{u,PW}$. Green's identity for the formally self-adjoint $P$,
	\[
	\braket{u,PW}-\braket{Pu,W}
	=\braket{\gamma_1u,\gamma_0W}_{\partial M}-\braket{\gamma_0u,\gamma_1W}_{\partial M},
	\]
	holds for $u\in\Nscr(\bar M)$ and $W\in\Cinf_c(\bar M)$: integrate by parts on
	$\{x\ge\delta\}$ and let $\delta\downarrow0$, the boundary terms converging because
	$u$ is a smooth function of $x$ with values in $\Dscr'(\partial M)$. Both right-hand
	terms vanish and $Pu=0$, so $\braket{u,\chi}=0$ for all $\chi$, i.e. $u=0$.
\end{proof}

The singularities of solutions of the boundary-value problem for $P$ are constrained to the compressed light cone bundle:
\begin{theorem}\label{theo:singularities of the BVP}
	Let $P$ be scalar and normally hyperbolic, and assume we have $u\in \Nscr(\bar{M})$ solving $Pu=0, \gamma_0 u=0$. Then $\WF_b(u)\subset\tilde{V}\bar M$.
\end{theorem}
\begin{proof}
	In the interior, this is part of the standard theory. At the boundary, it suffices to argue locally. Moreover, the conditions only depend on the principal symbols. Therefore we are, after localisation, exactly in the situation of \cref{theo:DBC satisfies LS}. The proof is finished since the set $\Char(P;B)$ computed there is, in local adapted coordinates, precisely the compressed causal cone bundle, see \cref{rem:compressed causal cone}.
\end{proof}

The final result of this section concerns the preservation of positive energy for solutions, and will be key in establishing the existence of Hadamard states.
\begin{theorem}
	\label{theo:propagation of future singularities}
	Let $\bar{M}$ be a globally hyperbolic spacetime with timelike boundary, $\bar\Sigma$ any Cauchy surface in $\bar{M}$. Assume that $u\in\Nscr(\bar{M})$ is a solution to
	\begin{equation}
		\label{eq:vanishing mixed problem}
		\left\{\begin{aligned}
			P u&=0\\
			u|_{\partial M}&=0.
		\end{aligned}\right.
	\end{equation}
	Assume furthermore that there is an open neighbourhood $\mathcal U\simeq (-1,1)\times \bar{\Sigma}$ of $\bar\Sigma$ such that $\WF_b(u|_{\mathcal U})\subset \tilde V^+{\mathcal U}$. Then, $\WF_b(u)\subset \tilde V^+\bar{M}$. Moreover, the same is true in the Hilbert-space-valued case.
\end{theorem}
\begin{proof}
	We divide the proof in several steps.
	
	\emph{Step 1: Microlocal splitting.} We decompose $u$ according to the time-orientation of its singularities. To this end, we cover $\tilde\Sstar{\bar{M}}\subset\bSs\bar{M}$ with open sets $U_+,U_-,U_s$, as follows:
	\begin{enumerate}
		\item $U_+$ is a small open neighbourhood (in $\bSs\bar{M}$) of the set $K_+=\{[\varsigma]\in\compcosphere{\bar{M}}\colon \varsigma\in \tilde V^+\bar{M}\}$ with $U_+\cap K_-=\emptyset$;
		\item $U_-$ is an open neighbourhood of $K_-$, which is the same as $K_+$ but with $\varsigma\in \tilde V^-\bar{M}$, and $U_-$ does not intersect $U_+$;
		\item $U_s$ is any open set in $\bSs{\bar{M}}$, not intersecting $K_+\cup K_-$ and such that $\{U_+,U_-,U_s\}$ covers $\bSs{\bar{M}}$.
	\end{enumerate}
	Sets as above can be constructed by choosing a Riemannian metric on $\bTs\bar{M}$ and taking $\epsilon$-neighbourhoods of $K_\pm$. Notice as well that $K_+$ and $K_-$ are closed and disjoint subsets of $\bSs{\bar{M}}$. Figure \cref{fig:microlocal partition} depicts, for $q\in\partial M$, the compressed causal cone in $\bTs_q\bar{M}$ and the neighbourhoods $U_\bullet$ that we use.
	
	
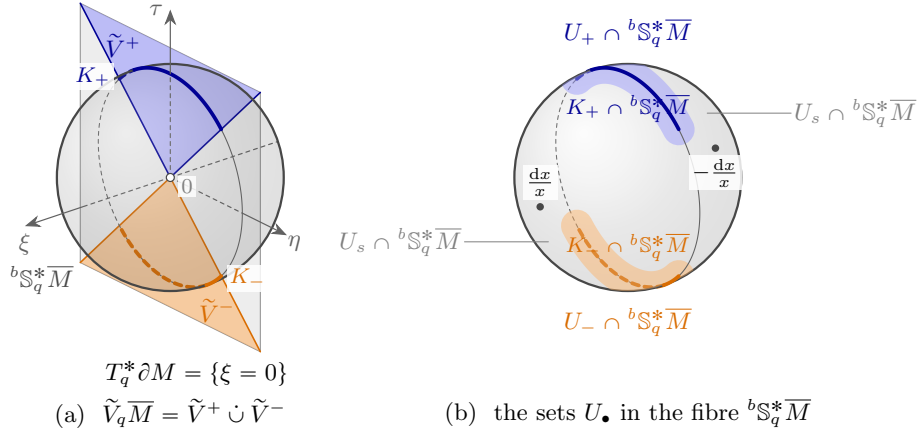
\begin{figure}[ht]\label{fig:microlocal partition}
		\centering
		\begin{tikzpicture}[
			line cap=round, line join=round,
			bproj/.style={x={(-1.1589cm,-0.3873cm)},   
				y={( 0.9524cm,-0.4713cm)},   
				z={( 0cm      , 1.3703cm)}}, 
			hidden/.style={dash pattern=on 1.5pt off 1.5pt},
			fatdash/.style={dash pattern=on 2.6pt off 2.2pt},
			axs/.style={-{Stealth[length=2.1mm,width=1.6mm]},cax,line width=.45pt},
			axsh/.style={cax,line width=.45pt,hidden},
			eq/.style={cax,line width=.4pt},
			lbl/.style={font=\footnotesize,inner sep=1.5pt},
			slbl/.style={font=\scriptsize,inner sep=1.5pt},
			lead/.style={cs,line width=.35pt},
			dot/.style={circle,inner sep=0pt,minimum size=2.8pt,fill=csph},
			]
			
			\colorlet{cwall}{black!7}
			\colorlet{cwalle}{black!42}
			\colorlet{csph}{black!72}
			\colorlet{cax}{black!62}
			\colorlet{cfut}{blue!58!black}
			\colorlet{cfutf}{blue!42}
			\colorlet{cpast}{orange!85!black}
			\colorlet{cpastf}{orange!62}
			\colorlet{cs}{black!50}
			
			\begin{scope}[bproj]
				\fill[cwall] (0,-1.25,-1.25) -- (0,1.25,-1.25) -- (0,1.25,1.25) -- (0,-1.25,1.25) -- cycle;
				\draw[cwalle,line width=.4pt]
				(0,-1.25,-1.25) -- (0,1.25,-1.25) -- (0,1.25,1.25) -- (0,-1.25,1.25) -- cycle;
				
				\fill[cfutf,opacity=.55]  (0,0,0) -- (0,-1.25,1.25) -- (0,1.25,1.25) -- cycle;
				\fill[cpastf,opacity=.55] (0,0,0) -- (0,-1.25,-1.25) -- (0,1.25,-1.25) -- cycle;
				\draw[cfut,line width=.6pt]  (0,-1.25,1.25) -- (0,0,0) -- (0,1.25,1.25);
				\draw[cpast,line width=.6pt] (0,-1.25,-1.25) -- (0,0,0) -- (0,1.25,-1.25);
				
				\draw[eq,hidden]
				plot[variable=\th,domain=150:315,samples=90,smooth] (0,{sin(\th)},{cos(\th)});
				\draw[cfut,line width=1.4pt,fatdash]
				plot[variable=\th,domain=-45:-30,samples=20,smooth] (0,{sin(\th)},{cos(\th)});
				\draw[cpast,line width=1.4pt,fatdash]
				plot[variable=\th,domain=150:225,samples=50,smooth] (0,{sin(\th)},{cos(\th)});
				\draw[axsh] (0,0,0) -- (1,0,0);
				\draw[axsh] (0,0,0) -- (-1.228,0,0);
				\draw[axsh] (0,0,0) -- (0,1,0);
				\draw[axsh] (0,0,0) -- (0,0,1);
			\end{scope}
			
			\shade[ball color=black!35,opacity=.16] (0,0) circle[radius=1.5cm];
			\draw[csph,line width=.85pt]            (0,0) circle[radius=1.5cm];
			
			\begin{scope}[bproj]
				\draw[eq]
				plot[variable=\th,domain=-30:150,samples=95,smooth] (0,{sin(\th)},{cos(\th)});
				\draw[cfut,line width=1.4pt]
				plot[variable=\th,domain=-30:45,samples=45,smooth] (0,{sin(\th)},{cos(\th)});
				\draw[cpast,line width=1.4pt]
				plot[variable=\th,domain=135:150,samples=20,smooth] (0,{sin(\th)},{cos(\th)});
				
				\draw[axs] (1,0,0) -- (1.66,0,0) node[lbl,pos=1,below=1pt] {$\xi$};
				\draw[axs] (0,1,0) -- (0,1.62,0) node[lbl,pos=1,below right=-2pt] {$\eta$};
				\draw[axs] (0,0,1) -- (0,0,1.62) node[lbl,pos=1,left=1pt] {$\tau$};
				
				\node[circle,inner sep=0pt,minimum size=3pt,draw=csph,line width=.45pt,
				fill=white] at (0,0,0) {};
			\end{scope}
			
			\node[lbl,text=cfut]  at (-0.62, 1.75) {$\widetilde V^+$};
			\node[lbl,text=cpast] at ( 0.62,-1.75) {$\widetilde V^-$};
			\node[slbl,text=cfut,anchor=east,fill=white,fill opacity=.85,text opacity=1,
			inner sep=1.2pt] at (-0.74, 1.33) {$K_+$};
			\node[slbl,text=cpast,anchor=west,fill=white,fill opacity=.85,text opacity=1,
			inner sep=1.2pt] at ( 0.74,-1.33) {$K_-$};
			\node[slbl,text=cs,anchor=west,fill=white,fill opacity=.85,text opacity=1,
			inner sep=1.2pt] at ( 0.10,-0.12) {$0$};
			\node[lbl,text=csph,anchor=north east,fill=white,fill opacity=.85,
			text opacity=1] at (-1.22,-1.10) {${}^{b}\mathbb{S}^{*}_{q}\overline{M}$};
			\node[lbl,anchor=north] at (0.35,-2.36) {$T^{*}_{q}\partial M=\{\xi=0\}$};
			\node[lbl,anchor=north] at (0,-2.86)
			{(a)\ \ $\widetilde V_q\overline{M}=\widetilde V^+\,\dot\cup\,\widetilde V^-$};
			
			\begin{scope}[shift={(6.05cm,0)}]
				\begin{scope}[bproj]
					\draw[eq,hidden]
					plot[variable=\th,domain=150:315,samples=90,smooth] (0,{sin(\th)},{cos(\th)});
					\draw[cfut,line width=1.4pt,fatdash]
					plot[variable=\th,domain=-45:-30,samples=20,smooth] (0,{sin(\th)},{cos(\th)});
					\draw[cpast,line width=1.4pt,fatdash]
					plot[variable=\th,domain=150:225,samples=50,smooth] (0,{sin(\th)},{cos(\th)});
				\end{scope}
				
				\shade[ball color=black!35,opacity=.16] (0,0) circle[radius=1.5cm];
				\draw[csph,line width=.85pt]            (0,0) circle[radius=1.5cm];
				
				\begin{scope}
					\clip (0,0) circle[radius=1.5cm];
					\begin{scope}[bproj]
						\draw[cfutf,opacity=.45,line width=12pt,line cap=round]
						plot[variable=\th,domain=-45:45,samples=55,smooth] (0,{sin(\th)},{cos(\th)});
						\draw[cpastf,opacity=.45,line width=12pt,line cap=round]
						plot[variable=\th,domain=135:225,samples=55,smooth] (0,{sin(\th)},{cos(\th)});
					\end{scope}
				\end{scope}
				
				\begin{scope}[bproj]
					\draw[eq]
					plot[variable=\th,domain=-30:150,samples=95,smooth] (0,{sin(\th)},{cos(\th)});
					\draw[cfut,line width=1.4pt]
					plot[variable=\th,domain=-30:45,samples=45,smooth] (0,{sin(\th)},{cos(\th)});
					\draw[cpast,line width=1.4pt]
					plot[variable=\th,domain=135:150,samples=20,smooth] (0,{sin(\th)},{cos(\th)});
					
					\node[dot] at ( 1,0,0) {};
					\node[dot] at (-1,0,0) {};
				\end{scope}
				
				\node[lbl,text=cfut,anchor=south]  at (0, 1.66) {$U_+\cap \bSs_q\bar{M}$};
				\node[lbl,text=cpast,anchor=north] at (0,-1.66) {$U_-\cap \bSs_q\bar{M}$};
				\node[slbl,text=cfut,anchor=north]  at (0, 1.16) {$K_+\cap \bSs_q\bar{M}$};
				\node[slbl,text=cpast,anchor=south] at (0,-1.16) {$K_-\cap \bSs_q\bar{M}$};
				
				\draw[lead] (-2.16,-0.85) -- (-1.05,-0.85);
				\node[lbl,text=cs,anchor=east] at (-2.14,-0.85) {$U_s\cap \bSs_q\bar{M}$};
				\draw[lead] ( 2.16, 0.85) -- ( 1.05, 0.85);
				\node[lbl,text=cs,anchor=west] at ( 2.14, 0.85) {$U_s\cap \bSs_q\bar{M}$};
				
				\node[slbl,anchor=south,fill=white,fill opacity=.85,text opacity=1,
				inner sep=1.2pt] at (-1.16,-0.30) {$\tfrac{\mathrm{d}x}{x}$};
				\node[slbl,anchor=north,fill=white,fill opacity=.85,text opacity=1,
				inner sep=1.2pt] at ( 1.16, 0.30) {$-\tfrac{\mathrm{d}x}{x}$};
				
				\node[lbl,anchor=north] at (0,-2.86)
				{(b)\ \ the sets $U_\bullet$ in the fibre ${}^{b}\mathbb{S}^{*}_{q}\overline{M}$};
			\end{scope}
		\end{tikzpicture}
	\end{figure}

	Pick a $b$-microlocal partition of unity subordinate to this cover, $\{B_+,B_-,B_s\}$, and write $I=B_++B_-+B_s$ up to an operator in $\Psi_b^{-\infty}(\bar{M})$. Since each $B_\bullet$ preserves $\WF_b(u)$, we obtain distributions $u_\bullet=B_\bullet u\in\Nscr(\bar{M})$ for $\bullet\in\{+,-,s\}$, microlocalised in the corresponding set and satisfying $u=u_++u_-+u_s\mod\Ascr(\bar{M})$.

	\emph{Step 2: Smoothness of the spacelike part.} We now claim that $u_s$ is smooth: on the one hand, by definition of $B_s$, the $b$-wave-front set of $u_s$ does not intersect $\tilde V^\pm\bar{M}$, on the other we have $\WF_b(u_s)\subset\WF_b(u)$ by \cref{lemma:properties of b-wave-front set}, so that $\WF_b(u_s)\subset \tilde{V}\bar{M}$. Therefore $\WF_b(u_s)$ is empty, so that $u_s\in\Ascr(\bar{M})$.

	\emph{Step 3: Uniqueness of the splitting.} The splitting $u=u_++u_-$ is then unique up to $\Cinf(\bar{M})$. Indeed, if $\tilde{u}_++\tilde{u}_-$ is another such splitting, then
	\[
	u_+-\tilde{u}_+=\tilde{u}_--u_-\mod \Ascr(\bar{M}).
	\]
	Since the distributions $u_+-\tilde u_+$ and $u_--\tilde{u}_-$ are microlocalised, respectively, in $\tilde V^+\bar{M}$ and $\tilde V^-\bar{M}$, and $\tilde V^-\bar{M}\cap \tilde V^+\bar{M}$ is empty, the $b$-wave-front sets of the two sides both have to be empty. Thus,
	\[
	\tilde{u}_+-u_+\in{\Ascr}(\bar{M})\cap\Nscr(\bar{M})=\Cinf(\bar{M}),
	\]
	proving essential uniqueness of the splitting.

	\emph{Step 4: Regularity of $Pu_\pm$.} The equation $Pu=0$ implies that $Pu_+=-P u_-\mod\Cinf(\bar{M})$ and $\WF_b(Pu_\bullet)\subset \WF_b(u_\bullet)$ by \cref{lemma:properties of b-wave-front set}, so that the $b$-wave-front set of either side has to be empty. Thus, by construction, $P u_{\pm}\in\Cinf(\bar{M})$.

	\emph{Step 5: Smoothness of $u_-$ in the future.} Choose now a time function $t\colon \bar{M}\rightarrow\R$ for which $\mathcal{U}\supset t^{-1}(-1,1)$ ($\mathcal U$ being the neighbourhood in the assumption) and a cutoff $\psi\in\Cinf(\bar{M},[0,1])$ such that
	\[\begin{aligned}
		\psi(q)&\equiv 0 \quad\text{if}\quad t(q)<-1, \\
		\psi(q)&\equiv 1 \quad\text{if}\quad t(q)>1.
	\end{aligned}\]
	Then $v_-\equiv\psi u_-\in\Nscr(\bar{M})$ satisfies $v_-\in\Cinf(\mathcal{U})$ and $v_-\equiv0$ on $t^{-1}(-\infty,-1)$. Moreover, $P v_-=f_-\in\Cinf_{\pc}(\bar{M})$. Indeed,
	\begin{equation*}
		Pv_-=\psi P u_-+[P,\psi]u_-,
	\end{equation*}
	where the first term is smooth by the previous paragraph, while $[P,\psi]$ has nonzero coefficients only in $\{-1\le t\le1\}\subset\mathcal U$, where $u_-$ is smooth.

	An analogous argument shows that $\gamma_0u_\pm\in\Cinf(\partial M)$. Recall first that since $u_\bullet\in\Nscr(\bar{M})$, the
	restrictions $\gamma_0u_\bullet\in\Dscr'(\partial M)$ are well defined by
	\cref{prop:restriction and boundary WF set}. Applying $\gamma_0$ to
	$u=u_++u_-\mod\Cinf(\bar{M})$ and using $\gamma_0u=0$ gives
	\begin{equation}\label{eq:traces of the two halves}
		\gamma_0u_+=-\gamma_0u_-\mod\Cinf(\partial M).
	\end{equation}
	On the other hand $\WF_b(u_\pm)\subset\tilde V^\pm\bar{M}$ by definition of the microlocal partition of unity, so that
	\cref{prop:restriction and boundary WF set} yields
	\[
	\WF(\gamma_0u_\pm)\subset\partial\WF(u_\pm)
	\subset\tilde V^\pm\bar{M}.
	\]
	Since the sets $\tilde{V}^+\bar{M}$ and $\tilde{V}^-\bar{M}$ are disjoint we must have $\WF(\gamma_0u_+)=\WF(\gamma_0u_-)$, so both are empty and
	$\gamma_0u_\pm\in\Cinf(\partial M)$.

	Consequently $\gamma_0v_-=(\psi|_{\partial M})\,\gamma_0u_-$ is smooth and vanishes for $t<-1$, that is, $\gamma_0v_-\in\Cinf_{\pc}(\partial M)$. Let $w\equiv\tilde G^+_D(f_-,\gamma_0v_-)\in\Cinf_{\pc}(\bar{M})$ be the smooth solution given by \cref{cor:fundamental solution for retarded BVP}. Then $w-v_-$ is past-compact, lies in $\Nscr(\bar{M})$ and satisfies $P(w-v_-)=0$, $\gamma_0(w-v_-)=0$, so that by \cref{lemma:distributional uniqueness for the retarded BVP} it vanishes. That is, $v_-=w\in\Cinf_{\pc}(\bar{M})$, which means in particular that $u_-$ is smooth for $t>1$, since $\psi=1$ there.

	\emph{Step 6: End of the proof.} The analogous argument for the advanced problem applies to $(1-\psi)u_-$ and shows that $u_-\in\Cinf(t^{-1}(-\infty,-1))$, and thus everywhere, away from $t^{-1}(\{-1,1\})$. However, we may always choose another cutoff $\tilde{\psi}$ with $\tilde{\psi}=0$ on $t^{-1}(-\infty,-1+\eps)$ and $\tilde{\psi}=1$ on $t^{-1}(1-\eps,\infty)$, and the same argument shows that $u_-$ is smooth away from $\{t=-1+\eps,1-\eps\}$. Thus, $u_-$ is in fact smooth everywhere. As a consequence, we have $u=u_+\mod\Ascr(\bar{M})$ and $\WF_b(u)=\WF_b(u_+)\subset\tilde{V}^+\bar{M}$, which is the claim for the scalar case. If $u$ takes values in a Hilbert space $\Hcal$, the same computation shows that, for any $v\in\Hcal$, $\WF_b((v,u))\subset \tilde{V}^+\bar{M}$. By \cref{theo:weak and strong b-wave-front set}, $\WF_b(u)=\bar{\cup_{v\in\Hcal}\WF_b((v,u))}$, and since $\tilde{V}^+\bar{M}$ is closed we find that $u$ only has future-directed singularities. This concludes the proof.
\end{proof}

\section{Quantum field theory for Dirichlet fields}
\label{sect:qft}

In this section we recall the formalism of quantum field theory on curved spacetimes, specialised to the Klein--Gordon operator with homogeneous Dirichlet boundary conditions. 

The field algebra $\mathfrak{U}(\bar{M})$ with Dirichlet boundary conditions is the abstract unital $\ast$-algebra, with unit $\Id$, generated by the symbols $\Phi(f)$ for all $f\in\Cinf_{c}(\bar{M})$, under the following relations
\begin{align*}
	f\mapsto\Phi(f)&\quad\text{is complex linear,}\\
	\Phi(f)\Phi(h)-\Phi(h)\Phi(f)&=-\im \braket{ f,G_D(h)}\mathrm{Id},\\
	\Phi(Pf)&=0,\quad f\in\Cinf_{D,c}(\bar{M}),\\
	\Phi(f)^\ast&=\Phi(\bar{f}).
\end{align*}
We discuss briefly our choice of test function space. A fundamental principle of quantum field theory is the so-called \textit{time-slice axiom}. This is the principle that asserts that the whole content of a QFT can be recovered by testing with functions supported near a Cauchy surface. More precisely, assume $\bar\Sigma_1$ and $\bar\Sigma_2$ are disjoint Cauchy surfaces in $\bar M$ and assume we have open neighbourhoods $U_j$ of $\bar\Sigma_j$, with $U_2$ lying entirely to the future of $U_1$. We may do so by choosing a global time function $t$ and considering $U_1=t^{-1}((-1,0))$ and $U_2=t^{-1}((1,2))$. If $f\in\Cinf_c(U_2)$, the time-slice axiom is the requirement that we are able to find an $\tilde f\in\Cinf_c(U_1)$ such that $f-\tilde{f}=Pu$ for some $u\in\Cinf_{D,c}(\bar{M})$. This is constructed by taking a past-cutoff $\psi$, identically equal to one in $[0,\infty)$ and identically vanishing in $(-\infty,-1]$, and defining $\tilde{f}=f-P(\psi G^-_D f)$. By definition, $\tilde{f}=0$ outside of $t^{-1}(-1,0)$, moreover the support of $\tilde{f}$ is contained in $J^-(\supp f)$. Thus, $\tilde{f}\in\Cinf_c(U_1)$ by global hyperbolicity. Furthermore, $G^-_Df$ satisfies homogeneous Dirichlet boundary conditions, which proves the time-slice axiom.

Given the above discussion, the setup of quantum field theory on curved spacetime, as presented for example in Section 4 of \cite{strohmaier_analytic_2026} (notice as well the discussion of the algebras of observables in \cite{dappiaggi_maxwell_timelike_2020}), can be implemented. Thus, the causal propagator identifies the algebra with the abstract CCR algebra of the symplectic vector space $(\ker_DP,\omega)$. The passage to the $C^\ast$-algebra version can then be obtained via exponentiation, since $\Phi(f)$ is represented by essentially self-adjoint operators on the (soon to be defined) Fock space.

The construction of an irreducible representation of this algebra can be achieved using a \textit{one-particle structure}. Explicitly, this is an embedding (i.e. an injection, continuous with respect to the quotient topology induced by $G_D$) $\kappa$ of $\ker_DP$ into a real Hilbert space $(\Hcal_\kappa,\left(\cdot,\cdot\right))$, endowed with a compatible complex structure $J\colon\Hcal_\kappa\to\Hcal_\kappa$ (that is, the tuple $(\left(\cdot,\cdot\right),\omega, J)$ is a Kähler structure), such that
\begin{itemize}
	\item $\kappa(\ker_D P)$ is dense;
	\item $\omega(f,h)=\left(\kappa(f),J\kappa(h)\right)$.
\end{itemize}
Thus, $\Hcal_\kappa$ is the completion of $\ker_DP$ with respect to the real inner product $\left(\cdot,\cdot\right)$. Via $J$, we can also view $\Hcal_\kappa$ as a complex Hilbert space $\Hcal_\kappa^\C$, where the inner product is given by
\begin{equation}\label{eq:complex inner product}
	\left({v_1,v_2}\right)_\C=\left({v_1,v_2}\right)+\im\left(J v_1,v_2\right).
\end{equation}
More concretely, we can identify $\Hcal^\C_\kappa$ with a subspace of the complexification $\Hcal_\kappa\tens\C$ as the $+\im$-eigenspace $H_+$ of the $\C$-linear extension of $J$. More precisely, $J\colon\Hcal_\kappa\tens\C\to\Hcal_\kappa\tens \C$ is given by $J(v\tens z)=(Jv)\tens z$ for any $z\in\C$ and induces a splitting $\Hcal_\kappa\tens \C=H_+\oplus H_-$, $H_-=\bar{H_+}$. The projection is $\frac{1}{2}(\Id-\im J)$, so $H_+=\ran(\Id-\im J)$ and the subspace $\Hcal_\kappa\tens 1$ of $\Hcal_\kappa\tens\C$ consists of all vectors of the form $v+\bar{v}$ in this identification. Due to our choice of normalisation for \eqref{eq:complex inner product} we have that $\|\frac{1}{2}(\Id-\im J)(v\tens 1)\|=\frac{1}{\sqrt{2}}\|v\|$.

The projection gives then a real-linear map $v\mapsto \frac{1}{2}(\Id-\im J)(v\tens 1)$, defined on the kernel $\ker_D P$. We let $p\colon\ker_D P\tens\C\to H_+$ be the complex-linear extension of this map, and see that the composition
\begin{equation}\label{eq:quantum field as HS-valued distribution}
	\Phi(\cdot)\equiv p\circ G_D(\cdot)
\end{equation}
is an $H_+$-valued supported distribution on $\bar{M}$ (this is the \textit{classical field} or \textit{one-particle distribution}).
\begin{theorem}\label{theo:propagator is a solution}
	$\Phi$ is a distributional solution of
	\begin{equation}\label{eq:BVP for causal propagator}
		\left\{\begin{aligned}
			Pu&=0\quad\text{on }M\\
			\gamma_0 u&=0.
		\end{aligned}\right.
	\end{equation}
	That is, it is an element of $\Nscr(\bar{M})$, it satisfies $\Phi(P\phi)=0$ for all $\phi\in\Cinf_{D,c}(\bar{M})$ and $\gamma_0\Phi(\psi)=0$ for all $\psi\in\Cinf_c(\partial M)$, both equations holding between elements of $H_+$.
\end{theorem}
\begin{proof}
	That $\Phi(P\phi)=0$ for all $\phi\in\Cinf_{D,c}(\bar{M})$ is a direct consequence of the above discussion. Then, by \cref{cor:kernel of NHO is normally regular}, $\Phi\in \Nscr(\bar{M};H_+)$, so that, for any choice of boundary defining function $x$, $\Phi$ is a smooth family of distributions on $\partial M$, parametrised by $x$. Thus, it restricts strongly to $x=0$ to give a distribution $u_0\in\Dscr'(\partial M;H_+)$. More precisely, $\Phi$ has a Taylor expansion in $x$ with coefficients $u_k=\frac{1}{k!}\partial^k_x (p G_D)|_{x=0}\in\Dscr'(\partial M;H_+)$:
	\begin{equation}
		\Phi=u_0+xu_1+\dots.
	\end{equation}
	The restriction to $\partial M$ is simply the $H_+$-valued distribution $u_0$. For $\psi\in\Cinf_c(\partial M)$, the pairing $\braket{u_0,\psi}$ is then $pG_D(\psi\tens \delta)$ (notice that this is well-defined since $\psi\tens\delta_{\partial M}$ is conormal to $\partial M$ and $\Nscr\subset\Ascr'$), and the claim is that this vanishes. Notice that the operator $G_D$ has a distributional kernel $K_D\in\dot\Dscr'(\bar{M}\times\bar{M})$, which satisfies the Dirichlet boundary condition in the first component, i.e. $\gamma_0^{(1)}K_D=0$ (the restriction being well-defined in view of normal regularity). Indeed, the range of $G_D$ consist precisely of functions solving the homogeneous Klein--Gordon equation with vanishing Dirichlet trace. Since $G_D$ is skew-adjoint as in \cref{lemma:G_D is a symplectic form}, $K_D$ satisfies the Dirichlet condition in the second component as well, i.e. $\gamma_0^{(2)}K_D=0$.

	For any $\psi\in\Cinf_c(\partial M)$, we have thus
	\begin{equation*}
		\braket{pG_D(\psi\tens\delta_{\partial M}),f}=\braket{K_D,f\tens \gamma_0^\ast\psi}=(\gamma_0^{(2)}K_D)(f,\psi)=0.
	\end{equation*}
	Since $p$ is a linear and continuous map it follows that $\gamma_0\Phi$ itself must vanish. This proves the assertion.
\end{proof}

Now, the field algebra acts on the symmetric Fock space $\Fock(H_+)$, defined to be the Hilbert space completion of $\bigoplus_{k=0}^\infty\odot^k H_+$, $\odot$ being the symmetrised tensor product. One calls the algebraic direct sum of the completed tensor products, $\bigoplus_{k=0}^\infty\hat\odot^k H_+$, the \textit{finite particle subspace} and denotes it by $\Fock_{\mathsf{fp}}(H_+)$.

Recall now (for example from \cite{reed_methods_II_1975}) the creation and annihilation operators associated with a vector $v\in H_+$, $a^\ast(v)$ and $a(v)$, defined on $\Fock_{\mathsf{fp}}(H_+)$. They satisfy the canonical commutation relations and give rise to the Segal field operator
\begin{equation}\label{eq:segal field operator}
	\Phi_S(f)=a(\bar{(\Id-p)G_Df})+a^\ast(pG_D f),\quad f\in\Cinf_c(\bar{M}).
\end{equation}
For real-valued test functions, this is just $\Phi_S(f)=a(p G_D(f))+a^\ast (p G_D(f))$. The operators $\Phi_S(f)$ then satisfy, for $f,h\in \Cinf_c(\bar{M})$,
\begin{equation}
	[\Phi_S(f),\Phi_S(h)]=-\im \braket{f,G_Dh}\Id,
\end{equation}
so that they define a representation of the CCR algebra on $\Fock(H_+)$ by unbounded, densely defined operators on the invariant domain $\Fock_{\mathrm{fp}}(H_+)$.
The Fock vacuum $\Omega$ defines a pure centred ($\omega_1=0$)quasi-free state
whose two-point function is
\begin{equation}\label{eq:two point function}
	\omega_2(f,h)
	=
	\left(\Omega,\Phi_S(f)\Phi_S(h)\Omega\right)
	=
	\left(p\circ G_D(\bar f),p\circ G_D(h)\right),
	\qquad f,h\in\Cinf_c(\bar M).
\end{equation}
Recall that a quasi-free state $\omega\colon \mathfrak{U}(P;\bar{M})\to\C$ is determined by its two-point function.
We will not investigate the properties of $\omega_2$ directly. Instead, we rely on \cref{prop:properties of b-WF for HS-distributions} (d) to define the Hadamard condition. 

For a general centred quasi-free state $\omega$   whose two-point function is a distribution, let
$(\Hcal_\omega,\pi_\omega,\Omega_\omega)$ be its GNS representation and
define its one-particle Hilbert space by
\[
	H_\omega
	\equiv
	\overline{\operatorname{span}_{\C}
	\left\{
		\pi_\omega(\Phi(f))\Omega_\omega:
		f\in\Cinf_c(\bar M)
	\right\}}
	\subset\Hcal_\omega .
\]
The associated one-particle distribution is
\[
	\Phi_\omega(f)
	\equiv
	\pi_\omega(\Phi(f))\Omega_\omega
	\in H_\omega .
\]
It is an $H_\omega$-valued distribution and, with our convention that
the Hilbert-space scalar product is conjugate-linear in its first
argument, satisfies
\[
	\omega_2(f,h)
	=
	\left(\Phi_\omega(\bar f),\Phi_\omega(h)\right)_{H_\omega}.
\]
For the pure Fock states constructed above, this distribution agrees,
up to the canonical unitary identification of the one-particle
spaces, with $p\circ G_D$.

\begin{definition}\label{def:hadamard states for Dirichlet}
	We say that a centred quasi-free state $\omega$ on
	$\mathfrak{U}(P,\bar M)$ is a \textit{Dirichlet--Hadamard state}
	if $\Phi_\omega\in\Nscr(\bar M;H_\omega)$ and we have
	\begin{equation}\label{eq:hadamard states}
		\WF_b(\Phi_\omega)\subset\tilde V^+\bar M .
	\end{equation}
\end{definition}

Notice that for states coming from the Dirichlet field algebra of the Klein--Gordon equation, the condition of normal regularity of $\Phi$ is automatically fulfilled. Moreover, the explicit state that we will construct in \cref{sect:hadamard} comes from a one-particle structure, hence is automatically pure.
\begin{remark}\label{rem:hadamard condition}
	In view of the aforementioned \cref{prop:properties of b-WF for HS-distributions}, the proposed Dirichlet--Hadamard condition is the natural analogue of the microlocal Hadamard condition of \textcite{radzikowski_micro-local_1996}, as reformulated in \cite{strohmaier_reeh-schlieder_2002}. Importantly, this reformulation allows us to do away with considerations of corners. 
\end{remark}

We finish by proving that our condition the Dirichlet--Hadamard condition completely characterises the singularities of the two-point function. In keeping with the philosophy of the paper, the argument is carried out on the one-particle distribution: two states are compared by assembling their one-particle distributions into a single Hilbert-space-valued distribution on $\bar{M}$, and the whole of the microlocal input is the salience of $\tilde{V}^+\bar{M}$.

\begin{theorem}
	\label{theo:uniqueness of hadamard states}
	Let $\omega$ and $\omega'$ be quasi-free Dirichlet--Hadamard states on
	$\mathfrak{U}(\bar{M})$, with one-particle distributions $\Phi$, $\Phi'$ and
	two-point functions $\omega_2$, $\omega_2'$. Then
	\begin{equation}
		\label{eq:uniqueness of hadamard states}
		\omega_2-\omega_2'\in\Cinf(\bar{M}\times\bar{M}).
	\end{equation}
\end{theorem}

\begin{proof}
	Recall, in the proof, the convention that we use coordinates $\zeta=(\lambda,\eta)$ with $\lambda=\xi+\im L$, and analyse the behaviour of the Mellin--Fourier transform for fixed $L$ and $|(\xi,\eta)|\to\infty$.
	
	Denote by $H_+$ and $H_+'$ the two one-particle Hilbert spaces and set
	\begin{equation}\label{eq:doubled one-particle distribution}
		\Hcal=H_+\oplus H_+',\qquad
		\Psi=\Phi\oplus\Phi',\qquad
		K=\Id_{H_+}\oplus\left(-\Id_{H_+'}\right).
	\end{equation}
	Then $\Psi$ is an $\Hcal$-valued distribution on $\bar{M}$, $K$ is a bounded self-adjoint involution on $\Hcal$ with $\|K\|=1$, and the difference of $\omega_2$ and $\omega_2'$ is the sesquilinear kernel of $\Psi$ twisted by $K$:
	\begin{equation}\label{eq:difference as a K-kernel}
		w\equiv\omega_2-\omega_2'
		=\left(\Psi(\bar{\,\cdot\,}),K\Psi(\,\cdot\,)\right).
	\end{equation}
	Clearly $\Psi$ is dual-conormal, hence $w$ is as well. Similar to Corollary 3.11 in \cite{vasy_propagation_singularities_corners_2008}, where an empty Sobolev $b$-wave-front set means that the distribution is in $\sob^1_\locs$, and thanks to \cref{lemma:ordinary and sobolev wave-front set}, it suffices to show that $\WF_b(w)$ is empty, for then $w$ is conormal to the full boundary of $\bar{M}\times\bar M$, and by dual-conormality it follows that $w$ is smooth on $\bar{M}\times \bar M$ (cf. Theorem 4.6.1 in \cite{melrose_differential_nodate} too). Indeed, if Vasy's $\WF_b^{m,\infty}(u)=\emptyset$ (with the obvious extended definition from the cases $m=\pm 1$) for all $m$, this means that the distributions $u$ is $H^m$-conormal to all boundary faces (i.e. in $\sob^{m,\infty}_b$ for all $m$), for all $m$. This means in particular that $u$ is in $\Ascr(\bar{M}\times\bar M)$, so $\WF_b(u)=\emptyset$ in view of the relation between the Sobolev and ordinary $b$-wave-front set we established.
	
	For any $\chi\in\Cinf_c(\bar{M})$ we have $\|\widehat{\chi\Psi}\|^2= \|\widehat{\chi\Phi}\|^2+\|\widehat{\chi\Phi'}\|^2$, which implies that $\widehat{\chi\Psi}$ is rapidly decreasing in a cone precisely when both	summands are. In particular, thanks to \cref{lemma:properties of b-wave-front set} (e),
	\begin{equation}\label{eq:b-WF of doubled distribution}
		\WF_b(\Psi)=\WF_b(\Phi)\cup\WF_b(\Phi')\subset\tilde{V}^+\bar{M}
	\end{equation}
	since both two-point functions are Dirichlet--Hadamard. Therefore, for any given covector $z,\zeta$, at most one of $(z,\zeta)$ and $(z,-\zeta)$ is in $\WF_b(\Psi)$.
	
	On the other hand, $w$ is symmetric. Indeed, since $\omega_2,\omega_2'$ are two-point functions on the same field algebra, we have by definition
	\begin{equation*}
		\omega_2(f,h)-\omega_2(h,f)=-\im\braket{f,G_Dh}=\omega_2'(f,h)-\omega_2'(h,f),
	\end{equation*}
	the middle term depending on the Dirichlet propagator alone and not on the
	choice of one-particle structure. Hence, denoting by $\check w$ the flipped distribution $\check{w}(f,h)=w(h,f)$, we obtain $w=\check{w}$ after rearranging.
	
	Notice now that the proof of \cref{prop:properties of b-WF for HS-distributions} (d) works just as well to describe the $b$-wave-front set of $\Psi$. Indeed, the Cauchy--Schwarz upper bound for the kernel easily extends to cover the case at hand for a bounded operator $K$ (the only difference is that an extra $\|K\|$ appears in the estimate, and in our case we even have $\|K\|=1$). Thus, with a completely analogous argument we find, for any real-valued cutoffs $\chi_1,\chi_2\in\Cinf_c(\bar{M})$ and any covectors $\zeta_1,\zeta_2$, an $L>0$ giving us the estimate
	\begin{equation}\label{eq:CS bound for w}
		\left|\widehat{(\chi_1\tens \chi_2)w}(\xi_1+\im L,\eta_1,\xi_2+\im L,\eta_2)\right|
		\le\left\|\widehat{\chi_1\Psi}(-\xi_1+\im L,-\eta_1)\right\|\cdot
		\left\|\widehat{\chi_2\Psi}(\xi_2+\im L,\eta_2)\right\|.
	\end{equation}
	
	We now show that the $b$-wave-front set of $w$ is empty by employing the characterisation of \cref{prop:properties of b-WF for HS-distributions} (c). First, for any $z_0\in\bar M$ and any $\zeta_0=(\xi_0,\eta_0)\neq 0$ we find $L>0$ such that for $\lambda=\xi+\im L$ and all $(\xi,\eta)$ in a cone $\Gamma$ around $\zeta_0$ with
	\begin{equation}\label{eq:opposite pair bracket}
		\left\|\widehat{\chi\Psi}(\xi+\im L,\eta)\right\|\cdot
		\left\|\widehat{\chi\Psi}(-\xi+\im L,-\eta)\right\|\le C_N\braket{(\xi,\eta)}^{-N},
		\qquad\zeta\in\Gamma,\quad N\in\N.
	\end{equation}
	Indeed, we know that at least one of $(z_0,\zeta_0)$ or $(z_0,-\zeta_0)$ does not belong to $\WF_b(\Psi)$, so that at least one Mellin--Fourier transform is rapidly decaying while the other stays polynomially bounded.
	
	Consider now the case of different points $z_1,z_2\in\bar{M}$ and covectors $(\zeta_1,\zeta_2)$ with $(\zeta_1,\zeta_2)\neq(0,0)$. Since $\check w=w$, \eqref{eq:CS bound for w} gives also the same estimate with exchanged $\chi_1$ and $\chi_2$ and arguments $\zeta_1$ and $\zeta_2$. Thus, if both $\zeta_1$ and $\zeta_2$ do not vanish, we find cutoffs $\chi_1$ and $\chi_2$, a complex line (using \cref{lemma:independence of the line}) $\Im\lambda_1=\Im\lambda_2=L$ and open cones $\Gamma_1$ and $\Gamma_2$ around $\zeta_1$ and $\zeta_2$ such that, multiplying these two versions of \eqref{eq:CS bound for w} gives for $\zeta_1'\in\Gamma_1$, $\zeta_2'\in\Gamma_2$
	\begin{equation}\label{eq:product bound}
		\begin{aligned}
			&\left|\Mcal\FT[(\chi_1\tens \chi_2)w](\xi_1'+\im L,\eta_1',\xi_2'+\im L,\eta_2')\right|^2\\ &\le\left[\left\|\widehat{\chi_1\Psi}(\xi_1'+\im L,\eta_1')\right\|\left\|\widehat{\chi_1\Psi}(-\xi_1'+\im L,-\eta_1')\right\|\right]\cdot\left[\left\|\widehat{\chi_2\Psi}(\xi_2'+\im L,\eta_2')\right\|\left\|\widehat{\chi_2\Psi}(-\xi_2'+\im L,-\eta_2')\right\|\right]\\
			&\le C_N\braket{(\xi_1',\eta_1')}^{-2N}\braket{\xi_2',\eta_2'}^{-2N}.
		\end{aligned}
	\end{equation}
	Since $\braket{(\xi_1',\eta_1'),(\xi_2',\eta_2')} \le\braket{(\xi_1',\eta_1')}\braket{(\xi_2',\eta_2')}$ and $N$ is arbitrary, we find that $\Mcal\FT[(\chi_1\tens\chi_2)w]$ is rapidly decreasing in $\Gamma_1\times\Gamma_2$. The same estimate \eqref{eq:product bound} gives also the rapid decay if either covector vanishes, the other being necessarily nonzero. In fact, if for example $\zeta_1=0$ and $\zeta_2\neq 0$, then we find a line $\Im\lambda_1=L$ and an open cone $\Gamma_2$ around $\zeta_2$ on which either $\hat{\chi_2\Psi}(\xi_2'+\im L,\eta_2')$ or $\hat{\chi_2\Psi}(-\xi_2'+\im L,-\eta_2')$ (hence their product) is rapidly decreasing. We now choose a cutoff $\chi_1$ equal to one near $z_1$, the very same line $\Im \lambda_1=L$ and an open cone $\Gamma$ around $(0,\zeta_2)$, small enough so that $(\zeta_1',\zeta_2')\in\Gamma$ implies $\zeta_2'\in\Gamma_2$ and $|(\xi_2',\eta_2')|\ge c|(\xi_1',\eta_1',\xi_2',\eta_2')|$ for some positive constant $c$. Then the second bracket in \eqref{eq:product bound} is rapidly decreasing on $\Im\lambda_2=L$, with the first staying polynomially bounded there in view of \cref{prop:properties of b-WF for HS-distributions}. Thus, the Mellin--Fourier transform is rapidly decaying in $\Gamma$, and the case $\zeta_2=0$ is obtained by swapping the roles of $\zeta_1$ and $\zeta_2$. It follows that $\WF_b(w)=\emptyset$, and the proof is complete.
\end{proof}

\section{The ultrastatic case}
\label{sect:ultrastatic}

In this section we assume $\bar{M}=\R\times\bar\Sigma$ with the product metric $g=\de t^2-g_{\bar{\Sigma}}$, where $g_{\bar{\Sigma}}$ is a time-independent Riemannian metric on $\bar\Sigma$. We analyse the operator $P=\Box_g+m^2$ on $\bar{M}$, which splits then as
\begin{equation}\label{eq:wave op in ultrastatic}
	P=\partial^2_{tt}+\Delta_D+m^2,
\end{equation}
where $\Delta_D$ is the Dirichlet extension of the Laplace--Beltrami operator on $\bar{\Sigma}$, with domain $\Dcal_2$. We fix the notation $A=\Delta_D+m^2$ and recall that, by Kato--Rellich, the powers $A^s$ in the range $s\in[0,1/2]$ are then defined on $\Dcal_{2s}$.

\subsection*{Preliminaries} We will consider the following problem for $u\colon\R\to \Dscr'(\bar{\Sigma})$:
\begin{equation}\label{eq:mixed problem in ultrastatic setting}
	\left\{\begin{aligned}
		&\ddot{u}(t)=-(\Delta_D+m^2) u(t)\quad \text{in }M=\R\times\Sigma,\\
		&u(t)|_{\partial \Sigma}=0,\\
		&u(0)=f\in\Dscr'(\bar\Sigma),\quad \dot u(0)=h\in\Dscr'(\bar\Sigma),
	\end{aligned}\right.
\end{equation}
with the regularity assumptions in $t$ to be specified later. Notice that we are only asking that $u(t)$ is a solution of the wave equation in the interior $M$ of $\bar{M}$, and that we have to be more precise with regard to the meaning of the boundary condition for $u$.

For the initial-boundary-value problem \eqref{eq:mixed problem in ultrastatic setting} (and in general for the nonhomogeneous versions), a necessary condition for the solvability in $\Cinf(\bar{M})$ is that the initial data satisfy some compatibility conditions. They are determined as follows: let us assume that we have a $u\in\Cinf(\R\times\bar{\Sigma})$ such that $\partial^2_{tt} u=-(\Delta+m^2)u$ and $u|_{\R\times\partial\Sigma}=0$. Then, denoting $u_k\equiv \partial^k_t u|_{t=0}\in\Cinf(\bar{\Sigma})$, we must have that $u_k|_{\partial \Sigma}=0$, since the solution $u$ vanishes identically on $\R\times\partial\Sigma$. On the other hand, $\partial^2_{tt} u=-(\Delta+m^2)u$, implying that $u_{k+2}=-(\Delta+m^2)u_k$. Thus necessarily we have:
\begin{equation}\label{eq:compatibility conditions for homogeneous Dirichlet}
	(\Delta+m^2)^k u_0|_{\partial \Sigma}=0,\quad (\Delta+m^2)^k u_1|_{\partial \Sigma}=0, \quad k\geq 0.
\end{equation}

Denote by $E\subset\Cinf_c(\bar{\Sigma})\oplus\Cinf_c(\bar{\Sigma})$ the (closed) subspace consisting of all pairs $(f,h)$ satisfying \eqref{eq:compatibility conditions for homogeneous Dirichlet}. It is a known fact (see Ivrii's exposition in \cite{egorov_encyclopedia_4_1993}) that, provided a weak solution $u$ exists and the initial data are smooth, these conditions are also sufficient to ensure smoothness of $u$ on $\R\times\bar\Sigma$ as well as vanishing at $\R\times\partial\Sigma$.
\begin{prop}\label{prop:flat initial data}
	$E$ is dense in $\Dcal_{1/2}\oplus\Dcal_{-1/2}$.
\end{prop}
\begin{proof}
	The space $\Cinf_c(\Sigma)$ is dense in $\Dcal_1$ since this is the domain of the Dirichlet extension. Since $\Dcal_s$ is dense in $\Dcal_r$ for $r<s$, we have that $\Cinf_c(\Sigma)$ is dense in $\Dcal_{1/2}$ and $\Dcal_{-1/2}$, too. On the other hand, $\Cinf_c(\Sigma)\oplus\Cinf_c(\Sigma)\subset E$, since a compactly supported function in the interior certainly satisfies the compatibility conditions \eqref{eq:compatibility conditions for homogeneous Dirichlet}. Since $E$ is a subspace of $\Dcal_1\oplus L^2(\Sigma)$, we find the chain of inclusions
	\begin{equation}\label{eq:inclusions of spaces with compatibility conditions}
		\Cinf_c(\Sigma)\oplus\Cinf_c(\Sigma)\hookrightarrow E\hookrightarrow \Dcal_{1/2}\oplus \Dcal_{-1/2}.
	\end{equation}
	Since $\Cinf_c(\Sigma)$ is dense in $\Dcal_{1/2}$ and in $\Dcal_{-1/2}$, we conclude that the closure of $E$ is the whole space. This finishes the proof.
\end{proof}

In order to study the properties of solutions with rough data, consider the explicit formula obtained via the spectral theorem: for $f\in\Dcal_1$ and $h\in \Dcal_0=L^2(\Sigma)$ we have
\begin{equation}\label{eq:solution to ultrastatic}
	u(t)=\left(\cos (t\sqrt{A})f+\frac{\sin(t\sqrt{A})}{\sqrt{A}}h\right),
\end{equation}
which is defined by functional calculus and is an element of $C(\R;\Dcal_1)\cap C^1(\R;\Dcal_0)$. It solves \eqref{eq:mixed problem in ultrastatic setting} by spectral calculus.

By interpolation and duality as in \cref{sect:distributions on manifolds with boundary}, the same formula is true for initial data in $\Dcal_s\oplus \Dcal_{s-1}$ for any $s\in[0,1]$ (in fact for $s\in\R$). Indeed for every fixed $t\in\R$ the operators $\cos(t\sqrt{A})$ and $A^{-1/2}\sin(t\sqrt{A})$ are bounded on $L^2(\Sigma)$, and in this case $u(t)$ is an element of $C(\R;\Dcal_s)\cap C^1(\R;\Dcal_{s-1})$. In the special case $s=1/2$, corresponding to the one-particle Hilbert space $\Dcal_{1/2}\oplus\Dcal_{-1/2}$, we have a natural symplectic structure, see \cref{cor:symplectic domain of Laplacian}. This allows us to give a direct argument.
\begin{lemma}
	\label{lemma:H^1/2-initial data gives H^1/2 solutions}
	Consider $(f,h)\in \Dcal_{1/2}\oplus \Dcal_{-1/2}$. There is a unique  $\R\ni t\mapsto u(t)\in \Dcal_{1/2}$ solving $\partial^2_{t} u(t)=-Au(t)$ on $M$, $u(t)|_{\partial\Sigma}=0$ for all $t$, and $u(0)=f$, $\dot u(0)=h$. Moreover, $u\in\Nscr(\R\times\bar{\Sigma})$.
\end{lemma}
\begin{proof}
	That we have a solution for initial data in $\Dcal_{1}\oplus \Dcal_0$ is classical, see \cite{evans_partial_2010}. By duality with respect to the symplectic form
	\begin{equation*}
		\omega_{\bar{\Sigma}}((f_1,h_1),(f_2,h_2))=\int_{\bar\Sigma}f_2h_1-f_1h_2 \de\vol_{g_{\bar\Sigma}},
	\end{equation*}
	we obtain that initial data in $\Dcal_0\oplus\Dcal_{-1}$ also give unique distributional solutions. By reflexivity, the complex interpolation method is compatible with duality (see \cite{bergh_interpolation_1976}) and hence the two endpoints interpolate to give unique solvability in $C(\R;\Dcal_s)$ for initial data in any $\Dcal_s\oplus\Dcal_{s-1}$, $s\in [0,1]$.

	Given the above solution, we clearly have $u\in L^1_\locs(\R;\Dcal_s)$ so that $u$ is an (extensible) distribution on $\R\times\Sigma$. By \cref{theo:smoothness and solvability in N'} we find $u\in\Nscr(\bar{M})$. The proof is complete.
\end{proof}

The inverses $G^\pm_D$ can also be written explicitly, again via functional calculus. By Duhamel's principle, they are given by
\begin{equation}
	\label{eq:spectral form of ultrastatic propagators}
	G^+_D=\vartheta(t)\frac{\sin(t\sqrt{A})}{\sqrt{A}},\quad
	G^-_D=-\vartheta(-t)\frac{\sin(t\sqrt{A})}{\sqrt{A}},
\end{equation}
where $\vartheta$ is the Heaviside function. The Pauli--Jordan propagator, in this product case, is then obtained as the one-parameter family
\begin{equation}
	\label{eq:spectral form of ultrastatic causal propagator}
	G_D=S(t)=\frac{\sin(t\sqrt{ A})}{\sqrt{A}}.
\end{equation}
Notice that, indeed, $S(t)$ is a bounded operator (acting by convolution in $t$), $Sf(t)$ solves the homogeneous wave equation by construction, and, for any $v\in \Dcal_{-1/2}$,
\begin{equation*}
	\label{eq:spectral expansion of causal propagator}
	\braket{v,(Sf)(t)}=\int_{[0,\infty)}\frac{1}{\omega}\sin((t-s)\omega)\de s\de\mu_{v,f(t)}.
\end{equation*}
For every $t\in\R$ and any given $f\in\dot\sob^{1/2}(\bar{\Sigma})$, the above spectral integral gives an element of $\Dcal_{1/2}$. Since, moreover, $Sf$ is an element of $\Nscr(\bar{M})$, the weak Dirichlet boundary condition enforced by $\Dcal_{1/2}\subset\dot\sob^{1/2}_\locs(\bar{\Sigma})$ is in fact the strong one $Sf(t)|_{\partial\Sigma}=0$.

Thus, for every fixed $t\in\R$ and every $f\in\Cinf_c(\bar M)$, $Sf(t)$ is a solution and its Cauchy data $(Sf(0),(\partial_t(Sf))(0))$ lie in $\Hcal_\kappa=\Dcal_{1/2}\oplus\Dcal_{-1/2}$. The map $\kappa$ is in this case exactly the Cauchy data map into the Hilbert space $\Hcal_\kappa=\Dcal_{1/2}\oplus\Dcal_{-1/2}$, namely
\begin{equation}
	\kappa\colon \ker_DP\to \Dcal_{1/2}\oplus\Dcal_{-1/2},\quad \kappa(u)=(u|_{t=0},\partial_\n u|_{t=0}).
\end{equation}
That $\kappa$ has dense range is exactly the content of \cref{prop:flat initial data}. The symplectic form is given, in terms of Cauchy data, by the formula
\begin{equation}\label{eq:ultrastatic symplectic form}
	\omega_{\bar{\Sigma}}((f_1,h_1),(f_2,h_2))=\int_{\bar\Sigma}f_2h_1-f_1h_2 \de\vol_{g_{\bar\Sigma}},
\end{equation}
as can be promptly verified by performing one integration by parts. The one-particle Hilbert space $H_+$ is then constructed with the help of the complex structure
\begin{equation*}
	J(f,h)=\left(-A^{-1/2}h, A^{1/2}f\right).
\end{equation*}
The scalar product obtained from $\omega$ and $J$, $(\cdot,\cdot)=-\omega(\cdot,J\cdot)$, is positive-definite:
\begin{equation}\label{eq:ultrastatic Hilbert product}
	\begin{aligned}
	((f,h),(f,h))&=-\omega((f,h),J(f,h))=(A^{1/2}f,f)_{L^2}+(A^{-1/2}h,h)_{L^2}\\
	&=\|A^{1/4}f\|^2+\|A^{-1/4}h\|^2\ge 0,
	\end{aligned}
\end{equation}
which reconfirms that $\Hcal_\kappa \cong \Dcal_{1/2}\oplus \Dcal_{-1/2}$. The positive-energy projection onto $H_+$ is $p=\frac{1}{2}(\Id-\im J)$, i.e.
\begin{equation*}
	p(f,h)=(w,-\im \sqrt{A}w), \quad w\in\Dcal_{1/2}.
\end{equation*}
That is, $H_+$ is the graph of $-\im\sqrt{A}$, which we identify with $\Dcal_{1/2}$ itself via $(w,-\im\sqrt{A}w)\mapsto \sqrt{2}w$ (the normalisation makes this map an isometry). The inner product in $H_+$ is therefore given by
\begin{equation}
	(v,w)_{H_+}=(A^{1/4}v,A^{1/4}w)_{L^2}.
\end{equation}
In this representation, the field $\Phi$ is thus the $H_+$-valued supported distribution on $\bar{M}$ given by
\begin{equation}\label{eq:ultrastatic field}
	\Phi(\varphi)=pG_D(\varphi) =\frac{\im}{\sqrt2}A^{-1/2}\int_\R e^{\im t\sqrt{A}}\varphi(t,\cdot)\de t.
\end{equation}
By the general theory of \cref{sect:qft}, the distribution $\Phi$ solves the system \eqref{eq:mixed problem in ultrastatic setting}, so that it is normally regular and its $b$-wave-front set is a subset of the compressed causal cone $\tilde{V}\bar{M}$. The goal is to show that the past-directed covectors in $\bTs\bar{M}$ are absent from $\WF_b(\Phi)$. As in \cref{theo:weak and strong b-wave-front set}, we have
\begin{equation*}
	\label{eq:HS-valued WF via Banach-Steinhaus}
	\WF_b(\Phi)=\overline{\bigcup_{v\in H_+}\WF_b\left(\left(v,p\circ G_D(\cdot)\right)\right)}.
\end{equation*}
Hence in order to prove that $\WF_b(\Phi)\subset {\tilde V}^+\bar{M}$ it suffices to show this for the scalar distributions $u_v\equiv \left(v,\Phi(\cdot)\right)$, for all $v\in H_+$. Notice that $P u_v=0$, so that $u_v\in\Nscr(\bar{M})$ according to \cref{cor:kernel of NHO is normally regular}. Moreover $u_v|_{\partial M}=0$. Thus, in the interior the standard theory gives $\WF(u_v|_{M})\subset VM$. By \cref{theo:b-WF with elliptic boundary conditions} we see that $\bWF(u_v)\subset \Char(P;\Id)\subset{\tilde V\bar M}\cap T^\ast \partial M$, that is, the $b$-wave-front set of $u_v$ at a point $(0,q)\in\R\times\partial\Sigma$ is confined to the set of covectors $(\tau,0,\eta)\in\b{T^\ast_q\bar M}$ such that $\tau^2\geq|\eta|^2$ in canonical coordinates.

Summarising the above discussion:
\begin{theorem}
	Let $\bar{M}$ be an ultrastatic globally hyperbolic spacetime with timelike boundary, let $H_+\cong\Dcal_{1/2}$ be the one-particle Hilbert space as constructed in \cref{sect:qft}, and let $\Phi\in\Nscr(\bar{M};H_+)$ be the one-particle distribution. For any $v\in H_+$, $u_v\equiv(v,\Phi(\cdot))_{H_+}$ is a scalar, supported, normally regular distribution on $\bar{M}$ and is given by
	\begin{equation}
		u_v(\phi)=\left(v,\Phi(\phi)\right)_{H_+}=\frac{\im}{\sqrt2}\int_\R\left(v,e^{\im t\sqrt A}\phi(t,\cdot)\right)_{L^2(\bar\Sigma)}\de t ,\quad v\in H_+\cong\Dcal_{1/2}\subset L^2(\bar\Sigma).
	\end{equation}
\end{theorem}
\begin{remark}\label{rem:ultrastatic field is local}
	It is well known that $\WF(u_v)\subset V^+M$ in the interior. Moreover, the computation is always local. We will therefore prove that, after localisation at a generic point $q\in\partial M$, the past-directed directions in $\tilde V_q\bar{M}$ cannot occur.
\end{remark}

\subsection*{The $b$-wave-front set estimate}
We isolate here some facts that will be used in the proof the \cref{theo:Hadamard form of quantum field}.
\begin{lemma}\label{lemma:fourier transform in time}
	Let $\rho\in\Cinf_c(\R)$ and consider its Fourier transform $\hat{\rho}\in\Scal(\R)$. Then, for every $w\in L^2(\bar\Sigma)$ and every $\tau\in\R$, the $L^2(\bar\Sigma)$-valued integral below converges absolutely and
	\begin{equation}\label{eq:time transform as functional calculus}
		(2\pi)^{-1/2}\int_\R e^{-\im t\tau}\rho(t)\,e^{\im t\sqrt A}w\de t
		=\hat\rho\left(\tau-\sqrt A\right)w ,
	\end{equation}
	the right-hand side being defined by the functional calculus of $\sqrt A$.
\end{lemma}
\begin{proof}
	The integrand is continuous in $t$ with values in $L^2(\bar\Sigma)$ and is bounded in norm by $|\rho(t)|\,\|w\|_{L^2}$, so the integral converges absolutely as a Bochner integral. For $v\in L^2(\bar\Sigma)$, functional calculus gives $\left(v,e^{\im t\sqrt A}w\right)=\int_{[0,\infty)}e^{\im t\omega}\de\mu_{v,w}(\omega)$, and the complex measure $\mu_{v,w}$ has finite total variation, bounded by $\|v\|_{L^2}\|w\|_{L^2}$. Hence Fubini's theorem applies to
	\begin{equation*}
		(2\pi)^{-1/2}\int_\R e^{-\im t\tau}\rho(t)\left(v,e^{\im t\sqrt A}w\right)\de t
		=\int_{[0,\infty)}\left((2\pi)^{-1/2}\int_\R e^{-\im t(\tau-\omega)}\rho(t)\de t\right)\de\mu_{v,w}(\omega).
	\end{equation*}
	The inner integral being $\hat\rho(\tau-\omega)$, we see that the right-hand side is $\left(v,\hat\rho(\tau-\sqrt A)w\right)$ by functional calculus. Since $v$ was arbitrary, the claim follows.
\end{proof}
\begin{lemma}\label{lemma:rapid decay of shifted cutoff}
	Let $\rho\in\Cinf_c(\R)$ and $\hat\rho=\FT\rho$.
	\begin{enumerate}
		\item For all $\alpha,\omega\ge0$ and all integers $0\le\theta\le N$ we have $\braket{\alpha+\omega}^{-N}\le\braket{\alpha}^{-(N-\theta)}\braket{\omega}^{-\theta}$. Consequently there are constants $C_{N,\theta,\rho}>0$ with \begin{equation}\label{eq:bound for Fourier transform of time-cutoff}
			\left|\hat\rho(\alpha+\omega)\right| \le C_{N,\theta,\rho}\braket{\alpha}^{-(N-\theta)}\braket{\omega}^{-\theta}, \qquad\alpha,\omega\ge0.
		\end{equation}
		\item Let $T$ be a self-adjoint operator on a Hilbert space with
		$\spec T\subset[0,\infty)$. Then for every $N\in\N$ there is a constant
		$C_{N,\rho}>0$, depending only on $N$ and $\rho$ and not on $T$, with
		\begin{equation}\label{eq:one-sided operator bound}
			\left\|\hat\rho\left(\tau-T\right)\right\|\le C_{N,\rho}\braket{\tau}^{-N}
			\qquad\text{for every }\tau\le0 .
		\end{equation}
	\end{enumerate}
\end{lemma}
\begin{proof}
	\begin{enumerate}
		\item Notice first that for positive $\alpha,\omega$ we have $\braket{\alpha+\omega}\geq \max\{\braket{\alpha},\braket{\omega}\}$. For every $N,\theta\in\N$ with $0\leq\theta\le N$ we then have, multiplying $N$ factors of $\braket{\alpha+\omega}^{-1}$ and using the above estimate, that
		\begin{equation}\label{eq:control on jap bracket of the sum}
			\braket{\alpha+\omega}^{-N}\le \braket{\alpha}^{-(N-\theta)}\braket{\omega}^{-\theta}.
		\end{equation}
		We have $\hat\rho\in\Sscr(\R)$ as well, so for every $N$ we find a positive constant $B_{N,\rho}$ with $|\hat\rho(s)|\le B_{N,\rho}\braket{s}^{-N}$. Together with the above estimate for $\braket{\alpha+\omega}$ this gives a constant $C_{N,\rho}>0$ with
		\begin{equation}
			|\hat\rho(\alpha+\omega)|\le C_{N,\rho}\braket{\alpha}^{-(N-\theta)}\braket{\omega}^{-\theta}.
		\end{equation}
	
		\item Since $\hat{\rho}$ is continuous and bounded, the spectral theorem gives $\left\|\hat\rho(\tau-T)\right\|=\sup_{\omega\in\spec T}\left|\hat\rho(\tau-\omega)\right|$. If $\tau\le0$ and $\omega\ge0$ then $\tau-\omega=-(|\tau|+\omega)$, so \eqref{eq:bound for Fourier transform of time-cutoff} with $\theta=0$ and $\alpha=|\tau|$ bounds the supremum by $C_{N,\rho}\braket{\tau}^{-N}$.
	\end{enumerate}
\end{proof}

Let now $q\in\partial\Sigma$ and choose adapted coordinates $(x,y)\in[0,\eps)\times\R^{n-2}$ near $q$, so that $(t,x,y)$ are coordinates on $\bar M$ near $q$ with dual $b$-coordinates $(\tau,\xi,\eta)$. Choose cutoffs $\chi\in\Cinf_c(\bar\Sigma)$ with $0\le\chi\le1$ and $\chi\equiv1$ near $q$, and $\rho\in\Cinf_c(\R)$ with $\rho\equiv1$ near $0$. For $\lambda\in\C$ and $\eta\in\R^{n-2}$ put
\begin{equation}\label{eq:test factor for MF transform}
	\psi_{\lambda,\eta}(x,y)\equiv\chi(x,y)\,x^{-\im\lambda-1}e^{-\im y\eta}.
\end{equation}
Then we have
\begin{lemma}\label{nc:lem:psi}
	Let $\lambda=\xi+\im L$ with $L>1/2$. Then $\psi_{\lambda,\eta}\in L^2(\Sigma)$	and $\|\psi_{\lambda,\eta}\|_{L^2}$ depends only on $L$ and $\chi$; in particular it is independent of $(\xi,\eta)$. In addition $\lambda\mapsto\psi_{\lambda,\eta}$ is a holomorphic $L^2(\Sigma)$-valued map on $\{\Im\lambda>1/2\}$.
\end{lemma}
\begin{proof}
	We have $|x^{-\im\lambda-1}|=x^{L-1}$ and $|e^{-\im y\eta}|=1$, so $\|\psi_{\lambda,\eta}\|^2_{L^2}=\int|\chi|^2x^{2L-2}\de\vol_{g_{\bar\Sigma}}$,	which is finite for $L>1/2$ and manifestly independent of $\xi$ and $\eta$. The difference quotients in $\lambda$ converge pointwise to $-\im\psi_{\lambda,\eta}\log x$, which lies in $L^2(\Sigma)$ for $\Im\lambda>1/2$; dominated convergence upgrades this to convergence in $L^2(\Sigma)$, so $\lambda\mapsto\psi_{\lambda,\eta}$ is holomorphic there.
\end{proof}
We can now prove the main estimate for the Mellin--Fourier transform of the localised field.
\begin{lemma}\label{lemma:main estimate for WF_b}
	Fix $L>\max\{L_0,1/2\}$ with $L_0$ such that $\hat{\rho\chi\Phi}$ is holomorphic for $\lambda>L_0$. Then for every $N\in\N$ there is a constant $C_N=C_N(N,\rho,\chi,L,m_0)$, independent of $\tau$, $\xi$ and $\eta$, such that
	\begin{equation}\label{nc:eq:main estimate}
		\left\|\widehat{\rho\chi\Phi}(\tau,\xi+\im L,\eta)\right\|_{H_+}\le C_N\braket{\tau}^{-N}\qquad\text{for all }\tau\le0,\ \xi\in\R,\ \eta\in\R^{n-2}.
	\end{equation}
	Moreover $\lambda\mapsto\widehat{\rho\chi\Phi}(\tau,\lambda,\eta)$ is holomorphic in a neighbourhood of the line $\{\Im\lambda=L\}$.
\end{lemma}
\begin{proof}
	Let $\lambda=\xi+\im L$. The Mellin--Fourier transform of the compactly supported distribution $\rho\chi\Phi$ is, by definition,
	\begin{equation}
		\widehat{\rho\chi\Phi}(\tau,\lambda,\eta)
		=(2\pi)^{-n/2}\,\frac{\im}{\sqrt2}A^{-1/2}\int_\R e^{\im t\sqrt A}\rho(t)e^{-\im t\tau}\psi_{\lambda,\eta}\de t
	\end{equation}
	According to \cref{lemma:fourier transform in time}, this is the same as
	\begin{equation}\label{eq:MF transform in closed form}		
		\hat{\rho\chi\Phi}=\frac{\im}{\sqrt2}(2\pi)^{\frac{1-n}{2}}A^{-1/2}\hat\rho\left(\tau-\sqrt A\right)\psi_{\lambda,\eta},
	\end{equation}
	all operators being bounded on $L^2(\Sigma)$ and $\psi_{\lambda,\eta}$ being square-integrable on $\bar{\Sigma}$ in view of \cref{nc:lem:psi}. Since $\|A^{-1/2}u\|_{H_+}=\|A^{-1/4}u\|_{L^2}$ we obtain a constant $C>0$ such that
	\begin{equation*}
		\left\|\widehat{\rho\chi\Phi}(\tau,\lambda,\eta)\right\|_{H_+} \le C \left\|\hat\rho\left(\tau-\sqrt A\right)\right\|_{\mathcal BL^2}\left\|\psi_{\lambda,\eta}\right\|_{L^2}.
	\end{equation*}
	Taking into account \cref{lemma:rapid decay of shifted cutoff} (2), the operator norm of $\hat\rho(\tau-\sqrt{A})$ can be estimated by $C_{N,\rho}\braket{\tau}^{-N}$ for $\tau\le0$. Since the norm of $\psi_{\lambda,\eta}$ depends only on $L$ and $\chi$ in view of \cref{nc:lem:psi}, we have shown the required estimate. The holomorphy of $\hat{\rho\chi\Phi}$ is clear, since $\psi$ depends holomorphically on $\lambda$ and the operator applied to it is independent of the complex parameter.
\end{proof}
\begin{remark}\label{rem:interior positive energy}
	At an interior point the same computation applies verbatim with $\psi_{\lambda,\eta}$ replaced by $\chi(z)e^{-\im z\zeta}$, whose $L^2$-norm is again independent of $\zeta$; \cref{lemma:main estimate for WF_b} then yields	rapid decay of $\widehat{\rho\chi\Phi}$ for $\tau\le0$, uniformly in $\zeta$, and therefore $\WF(\Phi|_M)\subset\{\tau>0\}$. Thus, up to knowing that the singularities of a solution lie in the characteristic set of Klein--Gordon operator (a standard fact in ordinary microlocal analysis), the proof of \cref{theo:Hadamard form of quantum field} below also covers the interior singularities.
\end{remark}
We can now prove the main theorem of this section.
\begin{theorem}\label{theo:Hadamard form of quantum field}
	Let $\bar\Sigma$ be a Riemannian manifold with boundary with metric $g_{\bar\Sigma}$, and let $\bar{M}=\R\times\bar\Sigma$ be the globally hyperbolic spacetime with timelike boundary with the ultrastatic metric $\de t^2-g_{\bar\Sigma}$. Let $\Phi$ be the one-particle distribution constructed in \cref{sect:qft} for $P=\partial^2_{tt}+A$, where $A=\Delta_D+m^2$ for a positive constant $m$. Then $\WF_b(\Phi)\subset\tilde{V}^+\bar{M}$.
\end{theorem}
\begin{proof}
	By \cref{theo:propagator is a solution} and \cref{cor:HS-valued N' is preserved}, $\Phi\in\Nscr(\bar M;H_+)$ solves $P\Phi=0$, $\gamma_0\Phi=0$, so \cref{theo:singularities of the BVP} gives
	\begin{equation}\label{eq:field is in the causal cone}
		\WF_b(\Phi)\subset\tilde V\bar M.
	\end{equation}
	Since $\tilde V\bar M$ is the disjoint union of $\tilde V^+\bar M$ and $\tilde V^-\bar M$, it suffices to show that no past-directed compressed covector lies in $\WF_b(\Phi)$. The statement is local, so we fix a boundary point $(0,q)$, $q\in\partial\Sigma$, and coordinates $(x,y)$ on a precompact open neighbourhood $U$ of $q$ such that the causal cone is given by $\tilde{V}U\simeq U\times \{(\tau,0,\eta)\colon \tau^2\ge |\eta|^2_g\}$. Choose then any open conic neighbourhood $U'\times\Gamma\subset \bTs U$ of $q$ with $\Gamma\subset \{(\tau,\xi,\eta)\colon \tau<0\}$ and $\tilde V^- U'\subset U'\times\Gamma$. This is achieved by simply enlarging the opening angle of $\tilde V^-U$ by some positive $\epsilon$ in every fibre. Then we have $\braket{(\tau,\xi,\eta)}\lesssim \braket\tau$ in $\Gamma$.
	
	Choose $L>\max\{L_0,1/2\}$. By \cref{lemma:main estimate for WF_b}, the Mellin--Fourier transform $\lambda\mapsto\widehat{\rho\chi\Phi}(\tau,\lambda,\eta)$ is holomorphic near $\{\Im\lambda=L\}$ and satisfies
	\begin{equation*}
		\left\|\widehat{\rho\chi\Phi}(\tau,\xi+\im L,\eta)\right\|_{H_+}
		\le C_N\braket{\tau}^{-N}
	\end{equation*}
	for each $N$ with constants $C_N>0$. Since $\braket{(\tau,\xi,\eta)}\lesssim \braket{\tau}$, we deduce
	\begin{equation*}
		\left\|\widehat{\rho\chi\Phi}(\tau,\xi+\im L,\eta)\right\|_{H_+}
		\le C_N\braket{(\tau,\xi,\eta)}^{-N},
	\end{equation*}
	which proves the rapid decay along past-pointing directions. Since the only possibility is then $\WF_b(\Phi)\subset\tilde{V}^+\bar{M}$, the proof is complete.
\end{proof}

\subsection*{Example: the spatially compact case}
	We sketch here a similar, more precise computation in the spatially compact setting. Thus, assume $\bar\Sigma$ is compact, so that $\Dcal_{1/2}=\dot\sob^{1/2}(\bar{\Sigma})$, $\Dcal_{-1/2}=\sob^{-1/2}(\bar{\Sigma})$, and for every $s>0$ the operator $A^s$ has discrete spectrum, organised in a monotone non-decreasing sequence $\{\mu_j^{2s}\}_{j\in\N}$. If $V=0$ then the $\mu_j$ are just the square roots of the Laplace eigenvalues $\mathfrak{l}_j^2$; if $V=m^2$ and $s=1/2$ we find $\mu_j=\sqrt{\mathfrak l_j^2+m^2}$. With each $\mu_j$ is associated the normalised eigenfunction $\phi_j\in \Cinf(\bar{\Sigma})$, $\phi_j|_{\partial\Sigma}=0$, and the spectral measure takes the form
	\begin{equation}\label{eq:spectral measure for spatially compact}
		\mu_{f_1,f_2}=\sum_{j\ge0}\braket{f_1,\varpi_{\mu_j} f_2}\delta_{\mu_j},
	\end{equation}
	where $\varpi_\mu$ is the projection onto the eigenspace with eigenvalue $\mu$.

	The initial data $f$ and $h$ have then generalised Fourier expansions given by
	\begin{equation*}\label{eq:mode expansions of initial data}
		f=\sum_{j\ge 0} a_j \phi_j,\quad	h=\sum_{j\ge 0} b_j \phi_j,
	\end{equation*}
	where $(\sqrt{\mu_j}a_j)\in \ell^2$, $(b_j/\sqrt{\mu_j})\in\ell^2$. The solution of the wave equation with these initial data is obtained by a mode-by-mode computation and is
	\begin{equation}\label{eq:spatially compact solution}
		u(t)=\sum_{j\ge 0}\left(a_j\cos(\mu_j t)+\frac{b_j}{\mu_j}\sin(\mu_j t)\right)\phi_j\in\dot\sob^{1/2}(\bar{\Sigma}).
	\end{equation}
	The field is similarly obtained from \eqref{eq:ultrastatic field} using the discrete spectral measure above, i.e.
	\begin{equation}\label{eq:ultrastatic compact field}
		\Phi(\varphi)=\frac{\im}{\sqrt 2}\sum_{j\ge 0}\frac{1}{\mu_j}\int_\R e^{\im \mu_j t}(\phi_j,\varphi(t,\cdot))\de t \,\phi_j.
	\end{equation}

	Below is a sketch of how to derive the main estimate in the spatially compact case. Keeping in mind \cref{rem:ultrastatic field is local}, we choose cutoffs $\chi\in\Cinf_c(\bar\Sigma)$, $\chi= 1$ near $q$, and $\rho\in \Cinf_c(\R)$, $\rho=1$ near 0 and even, and perform a Mellin--Fourier transform. By definition,
	\begin{equation}\label{eq:Fourier transform of quantum field}
		\widehat{\chi\rho u_v}(\tau,\lambda,\eta)=(2\pi)^{-n/2}u_v(\chi(x,y)\rho(t) x^{-\im \lambda-1}e^{-\im(\tau t+y\eta)})
	\end{equation}
	where in principle $\Im\lambda>0$. Since, if $v=\sum_{j\geq 0} \bar{a_j}\phi_j$,
	\begin{equation}
		u_v(\varphi)=(v,\Phi(\varphi))=\sum_{j\ge 0}\frac{\im a_j}{\sqrt 2}\int_\R e^{\im t\mu_j}\braket{\varphi(t,\cdot),\phi_j}\de t
	\end{equation}
	we see, after integrating in $t$, that
	\begin{equation}\label{eq:MF transform of localised scalar field}
		\widehat{\chi\rho u_v}(\tau,\lambda,\eta)=(2\pi)^{-n/2}\sum_{j\geq 0}\im a_j\hat\rho(\mu_j-\tau)\int_0^\infty\int_{\R^{n-1}}x^{-\im\lambda}e^{-\im y\eta}\chi(x,y)\phi_j(x,y)\frac{\de x}{x}\de y.
	\end{equation}
	The Mellin integral is absolutely convergent for every $\lambda$ with $\Im \lambda>-1$ (in particular for real $\lambda$), since the Dirichlet eigenfunctions $\phi_j$ vanish at $\partial \Sigma$ to first order. As in the general case, it suffices to show that the whole integral is polynomially bounded in $\eta$ and $j$. In fact, this is here supplied by a bound on the $L^1$-norm of $\phi_j(x,y)/x$ on $\bar{\Sigma}$. By the fundamental theorem of calculus we have $\frac{1}{x}\phi_j(x,y)=\int_0^1\partial_x\phi_j(sx,y)\de s$. Integrating over $\bar\Sigma$ and applying Fubini--Tonelli and the Sobolev embedding leads to
	\begin{equation*}
		\label{eq:L1 norm of Dirichlet egf}
		\begin{aligned}
			\left\|\frac{\phi_j(x,y)}{x}\right\|_{L^1}&\leq \int_0^1\int_{\bar\Sigma}\left|\partial_x\phi_j(sx,y)\right|\de x\de y \de s\\
			&\lesssim\|\phi_j\|_{C^1}\lesssim\|\phi_j\|_{\sob^k}.
		\end{aligned}
	\end{equation*}
	As is well known, $\|\phi_j\|_{\sob^k}\lesssim\mu_j^{k}$, polynomially bounded in $j$ by the Weyl law. Moreover, $\|u\|_{\sob^k}$ is equivalent to $\|A^{k/2}u\|_{L^2}$, therefore we have
		\begin{equation}
			\label{eq:L1 norm of dirichlet egf in terms of eigenvalue}
			\|\phi_j\|^2_{\sob^k}\lesssim\sum_{l=0}^k\|A^{l/2}\phi_j\|^2_{L^2} \lesssim_k \mu_j^{2k}.
		\end{equation}
	Putting everything together we have a polynomial bound
	\begin{equation}
		\label{eq:polynomial bound for Dirichlet eigenfunctions}
		\left\|\frac{\phi_j(x,y)}{x}\right\|_{L^1}\lesssim \mu_j^{k}\lesssim j^{k'},
	\end{equation}
	where $k$ is chosen as above and $k'$ is some finite integer.

	On the other hand $(\sqrt{\mu_j}a_j)\in \ell^2$ since $v\in H_+$, so that by Cauchy--Schwarz it suffices to estimate
	\begin{equation}\label{eq:sum to be estimated for WF_b}
		\left(\sum_{j\geq 0}\mu_j^{2k-1}|\hat\rho(\mu_j-\tau)|^2\right)^{1/2}
	\end{equation}
	in terms of $\tau$. With $\alpha=-\tau$, this is exactly the situation of \cref{lemma:rapid decay of shifted cutoff} after reformulation with the spectral measure \eqref{eq:spectral measure for spatially compact}. Consequently, \eqref{eq:MF transform of localised scalar field} is rapidly decreasing in the region $\tau<0$ inside $\tilde{V}\bar{M}$; equivalently, the $b$-wave-front set of $u_v$ is contained in the compressed future causal cone. Notice that the strategy of the preceding subsection avoid the use of the Weyl law or of any local version.

\section{Hadamard states}
\label{sect:hadamard}

In this section we combine the previous results and prove \cref{theo:existence of Hadamard states} (Theorem (A) of the Introduction).
\begin{theorem}\label{theo:existence of Hadamard states}
	Let $\bar{M}$ be a globally hyperbolic spacetime with timelike boundary, $P=\Box+V$ a scalar, formally self-adjoint, normally hyperbolic operator and $V$ a real smooth function. Let $\mathfrak{U}(\bar{M})$ be the field algebra of $P$ as defined in \cref{sect:qft}. Then, there exists a pure, quasi-free Dirichlet--Hadamard state on $\mathfrak{U}(\bar{M})$.
\end{theorem}
\begin{proof}
	We employ the deformation argument of \cite{fulling_singularity_1981}, see also \cite{drago_moller_2022} for a more modern presentation. That is, choose an isometric splitting $\bar{M}\cong\R_t\times\bar{\Sigma}$ with metric $g=\beta\de t^2-h_{t}$, $\beta>0$ smooth on $\bar{M}$. We will construct below two other globally hyperbolic spacetime with timelike boundary, the ultrastatic model $\bar{M}_{ult}$ and the interpolating spacetime $\bar{M}_{int}$. For the chosen time coordinate $t$, consider a past-cutoff $\rho\in\Cinf(\R\times\bar{\Sigma};[0,1])$ with $\rho=0$ for $t\leq -1$ and $\rho=1$ for $t\ge 1$. For a fixed Riemannian metric $\tilde{h}$ on $\bar{\Sigma}$, let $\beta'=\rho\beta+(1-\rho)$ and $h'_{t}=\rho \beta^{-1}h_{t}+(1-\rho)\tilde h$ and define the new metric
	\begin{equation}\label{eq:interpolating metric}
		g'=\beta'(\de t^2-h'_{t}).
	\end{equation}
	For any choice of $\rho$ and $\tilde{h}$, the metric $g'$ is Lorentzian, coincides with $g$ for $t\ge 1$, and is ultrastatic for $t\le -1$. Moreover, the covector field $\de t$ is timelike and tangent to $\R\times\partial \Sigma$. We may choose $\tilde h$ to satisfy in addition
	\begin{itemize}
		\item $\tilde{h}\ge \frac{1}{\beta}h_{t}$ for all $t\in [-1,1]$, and
		\item the Riemannian metric induced by $\tilde{h}$ on $\bar{\Sigma}$ is complete.
	\end{itemize}
	Indeed, if $\Theta(x)\equiv \max_{t\in[-1,1]}\sup_{v\in T_x\bar{\Sigma}\setminus \{0\}} \frac{h_t(v,v)}{\beta(t,x)h_{1}(v,v)}$, choose any smooth $\tilde{\Theta}\in\Cinf(\bar{\Sigma})$ with $\tilde{\Theta}\ge \Theta$. Then we may construct $\tilde{h}=\tilde{\Theta}h_{1}+\alpha$ where $\alpha$ is any smooth Riemannian metric inducing a complete metric on $\bar{\Sigma}$. Clearly $\tilde{h}\ge \frac{1}{\beta}h_t$ for $t\in [-1,1]$, and since $\tilde h\ge \alpha$ with $\alpha$ metrically complete we see that $\tilde{h}$ is also complete.

	We claim that the spacetime with timelike boundary $\R\times\bar{\Sigma}$ with metric $g'$ is globally hyperbolic, in that $\bar\Sigma_{-1}$ is a Cauchy surface. It is clearly everywhere spacelike, and we claim it is maximal. Indeed, let $c$ be an inextensible $H^1$-causal curve with respect to $g'$ and suppose, by contradiction, that $c$ does not intersect $\bar{\Sigma}_{-1}$. Since $t$ is a time function, we must have $c\subset t^{-1}(-1,\infty)$ or $c\subset t^{-1}(-\infty,-1)$. The first possibility cannot be, since $c$ is also causal with respect to $g$ and inextensible, and since $\Sigma_{-1}$ is Cauchy for $g$ it must intersect $c$. On the other hand, in $t^{-1}(-\infty,-1)$ $g'=\de t^2-\tilde h$ so that the condition that $c$ is future-directed and causal is expressed, in the frame $(t\circ c,x=\pi_{\bar\Sigma}\circ c)$, as
	\begin{equation}\label{eq:future-directed causal curve}
		\tilde{h}(\dot x(s),\dot x(s))\le |\dot t(s)|^2.
	\end{equation}
	Notice that the above inequality holds true only almost everywhere since $c$ is locally Lipschitz. If $c\colon [0,b)\to\bar{M}$ is future-directed causal, inextensible towards the future, and $c\subset t^{-1}(-\infty,-1]$, then we have $\dot t>0$ and for any $0\le s_1\le s_2<b$ we have
	\begin{equation}\label{eq:distance and time-separation}
		\dist_{\tilde{h}}(x(s_1),x(s_2))\le \int_{s_1}^{s_2}|\dot x(s)|\de s\le\int_{s_1}^{s_2} \dot t(s)\de s=t(s_2)-t(s_1),
	\end{equation}
	that is, the Riemannian distance of $x(s_1)$ and $x(s_2)$ is controlled by their time-separation. Since $t$ is monotone increasing and $t\le -1$ in the considered region, we must have $t(s)\to t_\infty\le -1$, and \eqref{eq:distance and time-separation} implies that, for any $(s_j)\subset [0,b)$ with $s_j\to b$, $x(s_j)$ is a Cauchy sequence in $\bar{\Sigma}$. By the assumed metric completeness of $\tilde h$, there is $x_\infty$ with $x(s_j)\to x_\infty$, so that $c(s_j)\to (t_\infty,x_\infty)\in\R\times\bar{\Sigma}$. That is, $c$ is not inextensible, a contradiction. It follows that $c$ must intersect $\bar{\Sigma}_{-1}$, that is, $\bar{\Sigma}_{-1}$ is a Cauchy surface for $g'$. By \cref{prop:properties of ghstb}, $\R\times\bar{\Sigma}$ is globally hyperbolic with $g'$.

	Let then $\bar M_{int}$ be the globally hyperbolic spacetime with timelike boundary $\R\times\bar{\Sigma}$ with metric $g'$ and let $\bar{M}_{ult}$ be the manifold $\R\times\bar\Sigma$ with the product metric $\de t^2-\tilde h$. That the latter is also a globally hyperbolic spacetime with timelike boundary follows since $\bar\Sigma$ is a Cauchy surface, by the same argument as for the region $t^{-1}(-\infty,-1]$ above. In the same interpolation region $[-1,1]$ we can also deform the potential by the convex combination $V'=\rho V+(1-\rho)m^2$. Therefore, on $\bar M_{int}$ we have the operator $P'=\Box_{g'}+V'$ which reduces to $\partial^2_{tt}+(\Delta+m^2)$ in $t^{-1}(-\infty,-1)$. On $\bar M_{ult}$, this form holds globally, so the analysis of \cref{sect:ultrastatic} applies. That is, the one-particle distribution $\Phi_{ult}$, satisfying $\WF_b(\Phi_{ult})\subset \tilde V^+\bar{M}_{ult}$ in view of \cref{theo:Hadamard form of quantum field}, defines a pure, quasi-free Dirichlet--Hadamard state on $\bar M_{ult}$.

	The globally hyperbolic spacetime with timelike boundary $\bar{M}_{int}$ coincides with $\bar M_{ult}$ in a neighbourhood $U_1$ of a Cauchy surface sitting at a time $t=T<-1$ and not intersecting $\bar\Sigma_{-1}$. By the time-slice axiom, the one-particle distribution $\Phi_{int}$ constructed from $P'$ has there the same microlocal form as on $\bar M_{ult}$. That is, $\WF_b(\Phi_{int}|_{U_1})\subset\tilde{V}^+U_1$. By \cref{theo:propagation of future singularities}, it satisfies
	\begin{equation}
		\WF_b(\Phi_{int})\subset \tilde V^+{\bar M}_{int}.
	\end{equation}
	In particular, this is the case on a neighbourhood $U_2$ of a Cauchy surface sitting at time $t=T'>1$, not intersecting $\bar\Sigma_1$. Again by the time-slice axiom, the one-particle distribution $\Phi$ on $\bar M$ has, in $U_2$, the same microlocal form as $\Phi_{int}$. Another application of \cref{theo:propagation of future singularities} shows that $\WF_b(\Phi)$ only contains future-pointing covectors on the whole of $\bar M$. This proves that the pure, quasi-free state obtained from the one-particle construction on $\bar{M}$ satisfies the Dirichlet--Hadamard condition, and concludes the proof.
\end{proof}
As a final consequence, with an immediate proof by comparing \cref{theo:existence of Hadamard states} and \cref{theo:uniqueness of hadamard states}, we have
\begin{corollary}
	If $\omega_2$ is the two-point function of the pure, quasi-free state constructed in \cref{theo:existence of Hadamard states} and $\omega_2'$ is the two-point function of any other quasi-free Dirichlet--Hadamard state on the Klein--Gordon algebra of $\bar{M}$, then $\omega_2-\omega_2'\in \Cinf(\bar{M}\times\bar{M})$, with the difference solving the Dirichlet problem in both variables.
\end{corollary}

\appendix

\section{The small $b$-calculus}\label{app:b-calculus}
We give here a brief account of the properties of the $b$-calculus of Melrose, following Section 2 in \cite{vasy_propagation_singularities_corners_2008}. A very nice expository reference is \cite{grieser_b-calculus_2000}, while the most complete geometric picture, with application to index theory, is \cite{melrose_atiyah-patodi-singer_1993}.

Let $\bar{M}$ be a manifold with boundary. The $b$-calculus is a family of operators $\Psi_b(\bar{M})$, acting initially from $\Cinf_c(\bar{M})$ to $\Cinf(\bar{M})$. While a symbolic construction proceeding along the lines of the classical theory on manifolds without boundary is possible, a much clearer geometric picture has in fact been elaborated over the years. Thus, a $b$-operator $A$ is described in terms of its Schwartz kernel in the double space $\bar{M}^2$. One immediate remark is that the diagonal does not meet the corner $(\partial M)^2$ cleanly, and the concept of conormal distribution, central to Hörmander's invariant definition of a pseudo-differential operator, is ill-defined at this level. The remedy, at the most elementary level, is the introduction of polar coordinates in $[0,\infty)^2$.

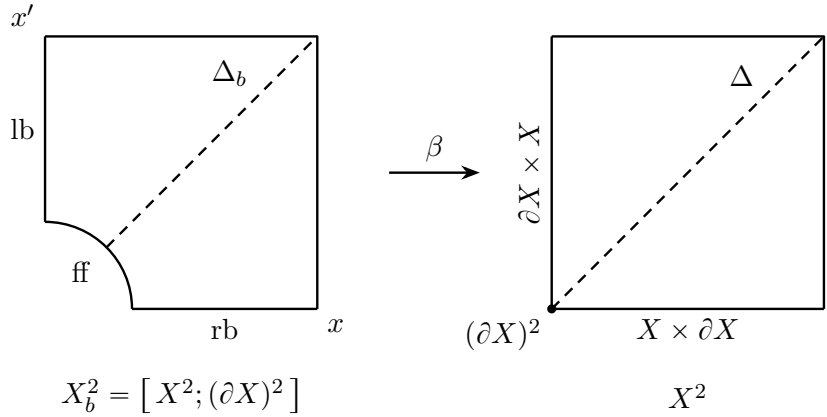
\begin{center}
\begin{tikzpicture}[
	>=Stealth, line join=round,
	lw/.style={line width=.9pt},
	dg/.style={line width=.9pt,dash pattern=on 4pt off 3pt},
	]
	
	\def\Side{3.6}   
	\def\Rad{1.15}   
	\def\Gap{3.1}    
	
	\begin{scope}
		\draw[lw] (\Rad,0) -- (\Side,0);                                
		\draw[lw] (\Rad,0) arc[start angle=0,end angle=90,radius=\Rad]; 
		\draw[lw] (0,\Rad) -- (0,\Side);                                
		\draw[lw] (0,\Side) -- (\Side,\Side) -- (\Side,0);              
		
		\draw[dg] ({\Rad/sqrt(2)},{\Rad/sqrt(2)}) -- (\Side,\Side);
		
		\node[left]  at (0,{(\Rad+\Side)/2}) {$\mathrm{lb}$};
		\node[below] at ({(\Rad+\Side)/2},0) {$\mathrm{rb}$};
		\node at ({0.42*\Rad},{0.42*\Rad})   {$\mathrm{ff}$};
		\node[above left] at (0.78*\Side,0.78*\Side) {$\Delta_b$};
		\node[below right] at (\Side,0) {$x$};
		\node[above left]  at (0,\Side)  {$x'$};
		\node at ({\Side/2},-1.15) {$X^2_b=\bigl[\,X^2;(\partial X)^2\,\bigr]$};
	\end{scope}
	
	\draw[->,lw] ({\Side+0.5*\Gap-0.6},{\Side/2}) -- ({\Side+0.5*\Gap+0.6},{\Side/2})
	node[midway,above] {$\beta$};
	
	\begin{scope}[shift={({\Side+\Gap},0)}]
		\draw[lw] (0,0) -- (\Side,0);
		\draw[lw] (0,0) -- (0,\Side);
		\draw[lw] (0,\Side) -- (\Side,\Side) -- (\Side,0);
		\draw[dg] (0,0) -- (\Side,\Side);          
		\fill (0,0) circle (1.7pt);                
		
		\node[rotate=90,above] at (0,{\Side/2}) {$\partial X\times X$};
		\node[below] at ({\Side/2},0) {$X\times\partial X$};
		\node[below left] at (0,0) {$(\partial X)^2$};
		\node[above left] at (0.78*\Side,0.78*\Side) {$\Delta$};
		\node at ({\Side/2},-1.15) {$X^2$};
	\end{scope}
\end{tikzpicture}
\end{center}

Explicitly, the interior of this manifold can be parametrised as $(r\cos \theta,r\sin\theta)$ for $r\in(0,\infty)$ and $\theta\in (0,\pi/2)$. Allowing for $r=0$ as well, we lose injectivity but gain surjectivity onto $(0,\infty)^2\cup \{(0,0)\}$, with the preimage of the origin being $\mathrm{ff}=\{(0,\theta),\theta\in(0,\pi/2)\}$. This is the ``front face''  of the blown-up space $[0,\infty)^2_b\equiv[[0,\infty)^2,\{(0,0)\}]$. Notice that the boundary of this manifold also contains the ``lateral faces'' $\mathrm{rb}=\R^+_r\times\{\theta=0\}$ and $\mathrm{lb}=\R^+_r\times\{\theta=\pi/2\}$, as well as two corners where $\mathrm{rb}$ and $\mathrm{lb}$ meet $\mathrm{ff}$. The blow-up space comes equipped with a smooth surjection $\beta\colon[0,\infty)^2_b\rightarrow [0,\infty)^2$, the ``blow-down'' map, given just by $\beta(r,\theta)=(r\cos\theta,r\sin\theta)$. Intuitively, this map ``shrinks'' the front face down to a single point and allows us to recover information on $[0,\infty)^2$. Working in $[0,\infty)^2_b$, there is a canonical lift of the diagonal $\diag=\{(x,x)\colon x\geq 0\}\subset [0,\infty)^2$ to a submanifold $\diag_b=[0,\infty)\times \{\theta=\pi/4\}$, namely $\diag_b$ satisfies $\beta(\diag_b)=\diag$. The upshot is that $\diag_b$ intersects the front face transversally, permitting us to define the notion of conormal distributions to $\diag_b$, all the way up to $\mathrm{ff}$. The Schwartz kernel of a $b$-operator on $[0,\infty)$ is precisely one such distribution $u\in\Dscr'([0,\infty)^2_b)$, together with the requirement that it has to vanish to infinite order on the lateral faces (since our manifold is not compact, we could add the requirement that the kernel be properly supported). The price is of course that $\beta$ is not invertible: smoothness of objects has to be checked after lifting, in coordinates that obfuscate the ``real'' geometry.

In the general case, for any manifold $\bar{M}$ with smooth boundary we can define a space $\bar{M}^2_b=[\bar{M}^2;(\partial \bar{M})^2]$ and a smooth surjection $\beta\colon \bar{M}^2_b\rightarrow \bar{M}^2$. Inside $\bar{M}^2_b$ there is a submanifold $\diag_b$, the inverse image of the diagonal in $\bar{M}^2$ under $\beta$, intersecting the front face $\mathrm{ff}=\beta^{-1}\left((\partial M)^2\right)$ transversally. Thus, the concept of conormality to $\diag_b$ makes sense all the way to the front face. This is the main requirement for the Schwartz kernel of a $b$-operator $A$. In addition, one imposes that $A$ is smooth at $\mathrm{ff}\setminus \diag_b$, that it vanishes to infinite order on the lateral faces $\mathrm{lb}, \mathrm{rb}$, and that it is properly supported whenever the manifold is not compact. These objects can be described, in local coordinates, in many different ways. We point the reader to the discussion following Theorem 18.3.5 in \textcite{hormander1994analysispseudodifferential} for one possibility, and to Sections 4.1 and 4.2 in \textcite{melrose_atiyah-patodi-singer_1993} for a more comprehensive account.

In order to describe the $b$-calculus symbolically, one introduces the $b$-tangent bundle and its dual, the $b$-cotangent bundle. We shall give various descriptions to emphasise the naturality of these objects. The reader should keep in mind that the only difference to the usual tangent and cotangent bundles is at the boundary. Let $\VF_b(\bar{M})$ denote the Lie algebra of vector fields tangent to the boundary. Every $V\in\VF_b(\bar{M})$ can be uniquely written, in a boundary chart $U$ with local coordinates $(x,y)$, as
\begin{equation}\label{eq:b-VF in local coordinates}
	V|_U=a^j(x,y)\partial_{y^j}+b(x,y)x\partial_x,
\end{equation}
where $a^j, b\in\Cinf(U)$. Let also $\Diff_b(\bar{M})$ denote the algebra of differential operators generated over $\Cinf(\bar{M})$ by $\VF_b(\bar{M})$.

The first description is via equivalence relations: $\bT\bar{M}$ is the disjoint union of the $b$-tangent spaces at each $q\in \bar{M}$, where the $b$-tangent space at $q\in \partial M$ is defined as the set of equivalence classes of $\VF_b(\bar{M})$ under the relation
\begin{equation}
	\label{eq:equivalence of b-vector fields}
	V\sim_q V'\iff\left\{\begin{aligned}
		\forall f\in\Cinf(\bar{M})\quad (V-V')f(q)&=0;\\
		\forall f\in\Cinf(\bar{M}), f|_{\partial M}=0,\implies \de\left((V-V')f\right)|_q&=0.
	\end{aligned}\right.
\end{equation}
A quick computation in boundary coordinates shows that the equivalence class at $q\in\partial M$ of $V\in\VF_b(\bar{M})$, $V=a^j(x,y)\partial_{y^j}+b(x,y)x\partial_x$ in local coordinates near $q$, can be represented uniquely by the vector $(a^1(q),\dots,a^{n-1}(q),b(q))$. Effectively, given two vector fields $V,V'\in\VF_b(\bar{M})$, we are saying that they are equivalent at $q\in\partial M$ if their normal component, evaluated at $q$, is the same \textit{after division by $x$}. Thus, in the $b$-tangent bundle, we retain some knowledge of the direction inside $\bar M$ from which we are approaching $\partial M$ (see the related construction of the $b$-groupoid of \cite{monthubert_2003_groupoids_corners}). In particular, for $q\in\partial M$, $\bT_q\bar{M}$ is \textit{not} $T_q\partial M$ but has the same dimension as the fibres inside.

This informal discussion can be made precise by saying that $\bT\bar{M}$ is a \textit{Lie algebroid}. This is a vector bundle $E$ over $\bar{M}$, together with a \textit{Lie bracket on sections} $[\cdot,\cdot]\colon\Gamma(E)^{\tens 2}\to\Gamma(E)$ (skew-symmetric+Jacobi identity), and an \textit{anchor}, that is, a vector bundle map $\varrho\colon E\to T\bar M$. This datum has to satisfy the ``anchored Leibniz rule'': for all $V,W\in\Gamma(E)$, and all $f\in\Cinf(\bar{M})$ we must have
\begin{equation}\label{eq:anchored leibniz rule}
	[V,fW]=(\varrho(V)f)W+f[V,W],
\end{equation}
in particular $\varrho$ is a Lie algebra homomorphism. For $\bT\bar{M}$, the sections are the sub-Lie algebra $\VF_b(\bar{M})$ of $\VF(\bar{M})$, however the anchor is not trivial: $\varrho$ is the evaluation map in $T\bar M$. In a chart $U$ with local coordinates $(x,y)$ and for a vector field $X$ with components $(a,b^1,\dots b^{n-1})$ with respect to the frame $(x\partial_x,\partial_{y^1},\dots,\partial_{y^{n-1}})$ of $\bT U$, we have $\varrho(X)=(xa,b^{1},\dots,b^{n-1})$ with respect to the ordinary frame $(\partial_x,\partial_{y^1},\dots,\partial_{y^{n-1}})$ of $T U$. Thus, at interior points $\varrho$ is an isomorphism of vector bundles, while at boundary points the anchor map has nontrivial kernel and cokernel, both of dimension $1$. Geometrically, one sees that $\varrho(\bT \bar{M})=TM\dot\cup T\partial M.$

The $b$-cotangent bundle is $\b{\Tstar}\bar{M}$, the dual bundle (also in the sense of Lie algebroids) to $\bT \bar{M}$. Its sections are so-called $b$-one-forms, $\b{\,\Omega^1(\bar{M})}$, and are written locally as smooth linear combinations of $\de y^j$ and $\frac{\de x}{x}$. Clearly, the notation $\frac{\de x}{x}$ is partially formal, since these objects are not smooth on the whole of $\bar{M}$ (unless the coefficient of $\frac{\de x}{x}$ vanishes at the boundary). Thus $b$-1-forms are not ordinary 1-forms. However, they can always be paired with a $b$-vector field, since the $x$ in the denominator ``cancels out'' with $x\partial_x$. On the other hand, any 1-form is a $b$-one-form.

Notice that, by duality, the ordinary cotangent bundle inherits a map into $\b{\Tstar \bar{M}}$. More specifically, the transpose of the anchor map, $\varrho^t$, is a smooth vector bundle map $\Tstar \bar{M}\rightarrow \b{\Tstar \bar{M}}$ whose image can be identified with $\Tstar M\dot\cup \Tstar \partial M$. In canonical coordinates $(x,y,\zeta,\eta)$ on $T^\ast \bar M$ (that is, $\zeta$ is the dual variable to $x$ and $\eta$ those to $y$), this is given by $\varrho^t(x,y,\zeta,\eta)=(x,y,\xi=x\zeta,\eta)$, so that the image of a section of $\Tstar \bar{M}$ in $\b{\Tstar \bar{M}}$ is obtained by ``multiplying and dividing by $x$''. For $A\subset \Tstar\bar{M}$, we denote by $\tilde{A}=\varrho^t(A)\subset\b\Tstar\bar{M}$, in particular $\tilde{\Tstar}\bar{M}$ is the image of $\varrho^t$. We call $\tilde{A}$ the \textit{compressed} $A$, for example $\tilde\Tstar\bar{M}=\Tstar{M}\cup T^\ast\partial M$ is the compressed cotangent bundle of $\bar{M}$, as it appears in the main body of the article.

As a second description of $\bT\bar{M}$ and $\b\Tstar\bar{M}$, recall that, for a boundaryless manifold $X$, there is a natural equivalence of vector bundles $\nu^\ast\diag\simeq T^\ast X$, the first as a bundle over the diagonal in $X\times X$ and the second over $X$. Drawing a parallel, one can define $\bT\bar{M}$ to be the normal bundle to the $b$-diagonal $\diag_b$ in $\bar{M}^2_b$ and $\b\Tstar\bar{M}$ as its conormal bundle. The equivalence of these two definitions can be seen using the various projections and blow-down maps over $T(\bar{M}^2_b)$ and can be found in \cite{hormander1994analysispseudodifferential}, pp.129--130.

In terms of symbols, operators in the small $b$-calculus can be defined in at least two equivalent pictures. In the first, one quantises \textit{lacunary} symbols $S_{la}$. These are smooth, compactly $x-$ and $y-$supported functions $a(x,y,\xi,\eta)$ on $\bar{\R^n_+}_{x,y}\times\R^n_{\xi,\eta}$ satisfying symbol estimates
\begin{equation}\label{eq:lacunary symbol estimates}
	|\partial_{x,y}^\alpha\partial_{\xi,\eta}^\beta a(z,\zeta)|\leq C_{\alpha,\beta}\braket{\zeta}^{m-|\beta|},
\end{equation}
and the \textit{lacunary condition}, namely, that the Fourier transform $\FT_{\xi\mapsto t}a$ vanishes for $t\le -1,x\ge 0$:
\begin{equation}
	\label{eq:lacunary condition}
	\int_\R e^{-\im t\xi}a(x,y,\xi,\eta)\de\xi=0,\quad t\le-1,\,x\geq 0.
\end{equation}
One quantises a lacunary symbol $a$ by applying left-quantisation to the pull-back of $a$ under the transposed anchor map $\varrho^t$: for $u\in\Cinf_c(\bar{\R^n_+})$ we have $(\varrho^t)^\ast a(x,y,\xi,\eta)=a(x,y,x\xi,\eta)$ and set therefore
\begin{equation}\label{eq:left-quantised lacunary symbol}
	\Op_{la}(a)u(x,y)=(2\pi)^{-n/2}\int_{\R^n} e^{\im (x\xi+y\eta)}a(x,y,x\xi,\eta)\hat u(\xi,\eta)\de \xi\de\eta.
\end{equation}

There is an alternative point of view, via so-called $b$-symbols as done by \textcite{hintz_2016_quasilinear_wave}. In fact, his class, quantised via his $\Op_b$ map, and the lacunary class as above produce the same operator classes, after passage to the logarithmic coordinate $x=-\log t$. We leave this quick check to the reader.

We finally introduce the class of $b$-operators.
\begin{definition}\label{def:b-operators}
	For a manifold with boundary $\bar{M}$, the classical $b$-operators $\Psi^m_b(\bar{M})$ are the properly supported operators $A\colon\Cinf_c(\bar{M})\to\Cinf(\bar M)$ such that, equivalently:
	\begin{enumerate}
		\item $A$ has a distributional kernel $k$ obtained as the push-forward to $\bar{M}^2$ of a distribution $K$ on $\bar{M}^2_b$, conormal to $\diag_b$ of order $m-n/4$, vanishing to infinite order at $\mathrm{lb}/\mathrm{rb}$, and one-step polyhomogeneous;
		\item $A$ is obtained as the quantisation $\Op_{la}(a)$ of a one-step polyhomogeneous symbol $a\in \sym{m}(\b\Tstar\bar{M})$ satisfying the lacunary condition \eqref{eq:lacunary condition} in each boundary chart.
	\end{enumerate}
\end{definition}
As customary, when writing down local expressions, we adopt the notation $D_j=-\im \partial_j$ for the ``Fourier derivative'', $\b{D}=(xD_x,D_y)$ for the ``$b$-Fourier nabla'' and, for a constant-coefficient symbol $\psi\in \sym{}_b(\bar{M})$, $\psi(\b{D})$ for the $b$-operator defined by the $b$-quantisation of the function $\psi$.

In our class, every $b$-operator has a well-defined principal symbol $\sigma_{pr}(A)$, a smooth homogeneous function on $\b\Tstar \bar{M}\setminus O$ of degree $m$, encoding the behaviour of the operator up to symbols of degree $m-1$. It descends to the quotient by the $\R^+$ action by dilations to give a smooth function $a_m\in\Cinf(\bSs{\bar{M}})$, where the cosphere bundle is defined as the set of \textit{oriented directions} in $\b{\Tstar \bar{M}}$ and bears no relation to the metric.

We have the following continuity properties of $b$-operators.
\begin{lemma}[cf. \textcite{melrose1982transformationofboundaryproblems}]\label{lemma:continuity of b-operators on distributions}
	Let $A\in \Psi^m_b(\bar{M})$. Then $A$ extends to a continuous operator on $\dot\Ascr(\bar{M})$ and on $\Ascr'(\bar{M})$. Furthermore, for each $k\in \Z$, $A$ extends to continuous operators from $\dot\sob^k_{\comps}(\bar{M})\to\dot\sob^{k-m}_{\locs}(\bar{M})$ and $\sob_{\comps}^k(\bar M)\to\sob_{\locs}^{k-m}(\bar M)$.
\end{lemma}

The concept of ellipticity can now be defined: a classical $b$-operator is \textit{elliptic} if its principal symbol does not vanish anywhere on $\b{\Tstar \bar{M}}\setminus O$. Elliptic operators admit parametrices, namely inverses modulo $\Psi_b^{-\infty}$. Here the main difference with the boundaryless case appears: a smoothing operator is \textit{not} necessarily compact on $L^{2}(\bar{M})$. Thus, a more precise analysis is required to achieve, for example, the Fredholm property of elliptic elements. We refer here to more specialised literature, e.g. \cite{melrose_atiyah-patodi-singer_1993}, for a discussion.

For the purpose of microlocal analysis, one usually introduces various subsets of $\bTs \bar{M}$ associated with a $b$-pseudo-differential operator.
\begin{definition}\label{def:elliptic,characteristic,microsupport}
	Let $(q_0,p^0)\in\bTs\bar{M}$ and $B\in\Psi_b(\bar{M})$ be an operator with principal symbol $b_0$ and full symbol $b$ (in some local coordinate system around $(q_0,p^0)$). We say
	\begin{enumerate}
		\item $(q_0,p^0)$ is in the \textit{elliptic set} $\Ell_b(B)$ of $B$ if $b_0(q_0,p^0)\neq 0$;
		\item $(q_0,p^0)$ is in the \textit{characteristic set} $\Char_b(B)$ of $B$ if $b_0(q_0,p^0)= 0$;
		\item $(q_0,p^0)$ is \textit{not} in the \textit{microsupport} $\mu-\supp_b(B)$ of $B$ if $b\in S^{-\infty}(\Gamma)$ in some open conic neighbourhood $\Gamma$ of $(q_0,p^0)$.
	\end{enumerate}
\end{definition}
We conclude with the microlocal partition of unity.
\begin{lemma}[cf. \cite{melrose1982transformationofboundaryproblems}, Proposition 8.16]
	For a finite open cover $U_1,\dots, U_k$ of $\bSs{\bar{M}}$ we can find $B_1,\dots, B_k\in\Psi_b(\bar{M})$ such that $\mu-\supp(B_j)\subset U_j$ and $\Id-\sum_{j=1}^k B_j\in\Psi^{-\infty}_b(\bar{M})$.
\end{lemma}

\printbibliography

\end{document}